\documentclass[authoryear,preprint,11pt]{elsarticle}

\usepackage{amsmath,amssymb,amsthm}
\usepackage{mathtools}
\usepackage{geometry}
\usepackage{setspace}
\usepackage{hyperref}
\usepackage{enumitem}
\usepackage{booktabs}
\usepackage{graphicx}
\usepackage{algorithm}
\usepackage{algorithmic}
\usepackage{multirow}
\usepackage{pdflscape}
\usepackage{apptools}
\usepackage{placeins} 

\newtheorem{proposition}{Proposition}
\journal{European Journal of Operational Research}

\begin{document}

\begin{frontmatter}



\title{Clustering-enhanced adaptive Benders decomposition for energy systems planning optimization}


\author[inst1]{Jun Wen Law} 

\affiliation[inst1]{organization={MIT Energy Initiative, Massachusetts Institute of Technology},
            city={Cambridge},
            postcode={02139}, 
            state={MA},
            country={USA}}

\author[inst2]{Dharik S. Mallapragada\corref{cor1}} 

\affiliation[inst2]{organization={Chemical and Biomolecular Engineering Department, Tandon School of Engineering, New
York University},
            city={Brooklyn},
            postcode={11201}, 
            state={NY},
            country={USA}}

\ead{dharik.mallapragada@nyu.edu}    
\cortext[cor1]{Corresponding author}

\begin{abstract}
High-resolution energy system capacity expansion models (CEMs) for energy transition planning often result in large-scale mixed-integer linear programming (MILP) formulations. Benders decomposition offers a scalable solution approach by iteratively solving a master problem for investment decisions and multiple subproblems for operational decisions. However, accumulated Benders cuts generated by the subproblems can make master-problem solution a major computational bottleneck. Incomplete subproblem parallelization can also introduce further bottlenecks when subproblems exceed available CPUs. We develop clustering-enhanced Benders decomposition methods to address these challenges, by using clustering to group similar subproblems for: a) aggregated Benders cut construction and b) identification of representative subproblems to be solved most frequently. For grouped-cuts, we examine two adaptive formulations based on dual variables and a fixed-grouping formulation based on exogenous time-series inputs. We evaluate these methods in an electricity-sector CEM across varying system sizes, temporal subproblem lengths, inter-subproblem coupling strengths represented by CO$_2$ policy, computational resources, and stochastic settings. Relative to a benchmark regularized multi-cut formulation, adaptive grouped cuts outperform fixed grouping and provide substantial benefits under weak inter-temporal coupling. The largest gains occur in larger systems with shorter subproblem horizons, where the master problem accounts for a greater share of runtime. Their effectiveness declines under strong inter-temporal coupling, such as annual CO$_2$ emissions limits, where the benchmark multi-cut performs best. The representative-subproblem method outperforms the benchmark under limited parallelization when subproblem solution dominates runtime. Overall, the preferred Benders decomposition strategy depends on inter-subproblem coupling strength and whether computational burden lies in the master problem or the subproblems.

\end{abstract}



\begin{keyword}
Benders decomposition \sep capacity expansion modeling \sep clustering \sep grouped Benders cuts \sep representative subproblems


\end{keyword}

\end{frontmatter}


\section{Introduction}
\label{sec:ch4_intro}

Energy system planning models such as capacity expansion models (CEMs) are widely used to evaluate long-term decarbonization pathways by co-optimizing long-term investment and short-term operational decisions across technologies, resources, and policy constraints \citep{brown_pypsa_2018,brown_synergies_2018,pfenninger_calliope_2018,venkatesh_open_2022,macdonald_macroenergyjl_2025,he_dolphyn_2023,johnston_switch_2019}. As decarbonization pathways increasingly rely on variable renewable energy (VRE), end-use electrification, low-carbon fuels, and carbon management, CEMs require higher temporal, spatial, and technological resolution to capture operational variability and cross-sector interactions \citep{kondziella_flexibility_2016,lund_review_2015,mignone_drivers_2024,law_decarbonization_2025}. Detailed energy system models also often include constraints that couple operating variables across periods, such as storage inventory balance, technology ramping constraints, and long-term policy constraints, such as renewable share requirements and emission budgets \citep{petkov_power--hydrogen_2020,giovanniello_influence_2024,he_sector_2021}. Together, these features can produce optimization problems with tens of millions of variables and constraints that are computationally challenging to solve with off-the-shelf solvers and algorithms.

Approaches to improve the tractability of high-resolution CEMs include: a) LP relaxations, b) temporal or spatial aggregation, and c) model decomposition. Along with LP relaxations, temporal aggregation is most commonly used, which approximates system annual operations using representative periods \citep{hoffmann_review_2020,mallapragada_impact_2018,novo_planning_2022,catania_impact_2025}. These representative periods are commonly selected using clustering techniques, such as k-means, applied to time-dependent input data (e.g., electricity demand and VRE availability) \citep{teichgraeber_clustering_2019,poncelet_impact_2016,pineda_chronological_2018,arango_representative_2018,miraftabzadeh_k-means_2023}. However, since the reduced model is solved over weighted representative periods rather than the full chronological horizon, results can be sensitive to the aggregation method \citep{mallapragada_impact_2018,kotzur_impact_2018,kittel_temporal_2022}, and inter-temporal effects related to long-duration storage and multi-period policy constraints (e.g. emissions cap) may be misrepresented \citep{kotzur_time_2018, mallapragada_long-run_2020, blanke_time_2022}. Moreover, the reduced-space CEM can still scale poorly when extended to operational uncertainty, thereby limiting its applicability in such cases \citep{yi_aggregate_2021, hoettecke_enhanced_2021}.

An alternative approach is Benders decomposition (BD), which separates long-term investment decisions from short-term operational decisions \citep{geoffrion_generalized_1972,tang_improved_2013,marin_electric_1998,baringo_wind_2012,fischetti_benders_2016,you_multicut_2013,gruson_benders_2021,brandenberg_refined_2021,lumbreras_transmission_2013,kergosien_benders_2017}. In this framework, a master problem (MP) determines capacity expansion decisions, while operational subproblems (SPs) are solved independently under fixed investments and return operating information to the MP through linear constraints known as Benders cuts, which iteratively refine the MP's approximation of operating costs.

Compared with temporal aggregation approaches, BD can preserve the full chronological resolution of operations while enabling decomposition across periods, making it well suited for large-scale energy system models with high operational heterogeneity. Earlier studies established the value of BD in energy system models, but often handled inter-temporal constraints through reduced temporal representations \citep{lara_deterministic_2018,munoz_new_2016,soares_integrated_2022,zhang_stabilised_2024,li_mixed-integer_2022}, large operational SPs \citep{goke_stabilized_2024}, or simplified formulations for those constraints \citep{munoz_new_2016,lohmann_tailored_2017}. More recent work addressed this limitation by shifting inter-temporal constraints to the MP through SP-level budgeting variables \citep{jacobson_computationally_2024,pecci_regularized_2025}. This approach allows for parallelizing the solution of each operational SP for fixed values of investment and budgeting variables at each iteration, but also increases the number of MP variables and the number of Benders cuts added per iteration, accelerating the growth of the MP. 

Regularized BD improves algorithm convergence by mitigating oscillatory behavior of MP solutions between lower and upper bounds, via solving interior level-set problems to find sub-optimal feasible solutions \citep{zhang_integrated_2025,pecci_regularized_2025, lemarechal_new_1995, gondzio_new_2013}. Convergence is further enhanced by applying the algorithm to the LP relaxation of the CEM before enforcing integrality constraints, enabling the solution of significantly larger model instances \citep{pecci_regularized_2025}. Other approaches to speed up BD algorithm convergence include aggregating SPs through scenario partitioning, where partitions are iteratively refined until an exactness condition is satisfied \citep{ramirez-pico_benders_2023}. However, applying such partitioning can be challenging in energy system models because operating regimes may vary significantly across periods.

Despite these BD algorithmic advances, Benders cut management remains a key computational bottleneck for BD-enabled CEMs. Multi-cut BD generates one cut per SP, which can improve convergence but leads to rapid MP growth, while single-cut BD limits MP growth by aggregating SP information into one cut per iteration, which can provide weaker convergence guidance \citep{bertsimas_stochastic_2025,brahmbhatt_benders_2025}. Although clustering and BD have both been used to improve CEM tractability, they have generally been applied independently. This overlooks the possibility that SPs with similar operating regimes may generate redundant or weakly informative cuts that contribute little toward convergence in multi-cut BD. More broadly, the literature lacks methods that jointly leverage clustering and model decomposition to improve both cut quality and MP tractability.

We address this gap by developing clustering-enhanced BD algorithms that exploit similarities among SPs across operating periods. This study makes three main contributions. \textbf{First}, we develop grouped-cut BD formulations that aggregate Benders cuts over SP clusters rather than generating one cut per SP. These include adaptive dual-based grouping with shared or individual-recourse representations (\textit{adapt-G-S} and \textit{adapt-G-I}) that update SP clusters as the algorithm progresses, and a fixed-grouping variant based on clustering time-dependent input data with shared-recourse variables (\textit{fix-G-S}). \textbf{Second}, we propose a representative-SP method (\textit{rep-SP}) that solves all SPs only at regrouping iterations, and solves one representative SP per cluster in the intervening iterations, thus reducing SP solution effort when computing resources are limited. Such settings are increasingly relevant due to: a) growing interest in incorporating operational uncertainty in CEMs, which directly increases the number of SPs when using a scenario-based approximation of uncertainty, and b) constraints on CPU availability resulting from growing demand from AI workloads. \textbf{Third}, we evaluate these methods in electricity-sector case studies using the MACRO CEM \citep{macdonald_macroenergyjl_2025} across model and SP sizes, CO$_2$ policy settings, hardware availability, and a stochastic case with multiple weather-year realizations, comparing their computational runtime against a state-of-the-art regularized multi-cut BD formulation \citep{pecci_regularized_2025}.

The results show that no single BD formulation performs best across all settings. Rather, the most effective approach depends on the model structure, specifically the strength of inter-temporal coupling, and computational resource availability. The \textit{adapt-G-S} and \textit{adapt-G-I} formulations consistently outperform \textit{fix-G-S}, and are most effective when inter-temporal policy coupling is weak and MP solution time dominates runtime. However, their benefits decline under CO$_2$ cap constraints, where budgeting variables create strong inter-temporal coupling across SPs and the benchmark multi-cut performs best. In contrast, \textit{rep-SP} is most valuable when SP solution effort dominates, particularly under limited SP parallelization or stochastic cases with many operational subproblems.

\section{Methods}
\label{sec:ch4_methods}

\subsection{MACRO capacity expansion model formulation}
\label{subsec:ch4_macro}

We evaluate all BD methods using an electricity-sector version of the MACRO CEM for case studies derived from the IPM regions of the United States \citep{macdonald_macroenergyjl_2025, us_epa_documentation_2018}. MACRO is a graph-based CEM that co-optimizes long-term investment decisions and short-term dispatch operations under infrastructure, operational, transmission, and policy constraints.

Let \(y \in \mathbb{R}^{n_y}\) denote the vector of long-term planning decisions, with \(y_j \in \mathbb{Z}\) for \(j \in \mathcal{I}\), where \(\mathcal{I}\) indexes the integer planning variables.  Let \(\mathcal{S}\) denote the set of operational periods, \(\mathcal{T}_s\) the set of hourly time steps within each period \(s \in \mathcal{S}\), and \(x_{s,t} \in \mathbb{R}^{n_x}\) the vector of operational decisions at hour \(t\) in period \(s\). Let \(c_I \in \mathbb{R}^{n_y}\) and \(c_{s,t} \in \mathbb{R}^{n_x}\) denote the investment cost vector and hourly operating cost vector, respectively. We define \(x_s := \{x_{s,t}\}_{t \in \mathcal{T}_s}\) as the vector of operational decisions within period \(s\). Under inter-temporal constraints such as annual CO$_2$ caps, budgeting variables \(q := \{q_s\}_{s \in \mathcal{S}}\) are introduced to allocate a system-wide annual budget \(\bar{Q}\) across periods. The resulting CEM formulation is shown in Eq. \ref{eq:budget_problem}. The objective function in Eq. \ref{eq:budget_problem_obj} minimizes the total investment and operating costs, Eq. \ref{eq:budget_problem_op} represents operational constraints within each temporal period $s$, such as hourly supply-demand balance, VRE availability, generator ramping limits, inter-regional flow, and storage inventory balance. Eq. \ref{eq:budget_problem_plan} represents constraints on planning decisions, including limits on capacity expansion, network expansion, and retirement. Finally, Eq. \ref{eq:budget_problem_local} and Eq. \ref{eq:budget_problem_global} represent inter-temporal coupling constraints spanning multiple periods using budgeting variables.

\begin{subequations}
\label{eq:budget_problem}
\begin{align}
\min_{y,\{q_s\},\{x_s\}} \quad
& c_I^\top y
+ \sum_{s\in\mathcal{S}} \sum_{t\in\mathcal{T}_s} c_{s,t}^\top x_{s,t}
\label{eq:budget_problem_obj} \\
\text{s.t.} \quad
& A_s x_s + B_s y \le d_s
&& \forall s \in \mathcal{S}
\label{eq:budget_problem_op} \\
& \sum_{t\in\mathcal{T}_s} e_{s,t}^\top x_{s,t} \le q_s
&& \forall s \in \mathcal{S}
\label{eq:budget_problem_local} \\
& \sum_{s\in\mathcal{S}} q_s = \bar{Q}
\label{eq:budget_problem_global} \\
& R y \le r
\label{eq:budget_problem_plan} \\
& x_{s,t} \ge 0
&& \forall s \in \mathcal{S},\;\forall t \in \mathcal{T}_s
\label{eq:budget_problem_x_domain} \\
& q_s \ge 0
&& \forall s \in \mathcal{S}
\label{eq:budget_problem_q_domain} \\
& y \ge 0
\label{eq:budget_problem_y_domain} \\
& y_j \in \mathbb{Z}
&& \forall j \in \mathcal{I}
\label{eq:budget_problem_int}
\end{align}
\end{subequations}


\subsection{Benchmark regularized multi-cut BD}
\label{subsec:ch4_benchmark}

The benchmark BD method used in this study is the regularized multi-cut BD implementation in MACRO, based on \cite{pecci_regularized_2025}. In this approach, the CEM formulation in Eq. \ref{eq:budget_problem} is decomposed into a MP and a set of operational SPs, one for each temporal period $s \in \mathcal{S}$. The MP determines the planning decisions $y$ and budget variables \(q := \{q_s\}_{s \in \mathcal{S}}\) when inter-temporal constraints are present, while each SP determines operational decisions $x_s$ in its corresponding period.

At Benders iteration $i \geq 0$, the operational SP in Eq. \ref{eq:subproblem} is solved for each $s \in \mathcal{S}$ given fixed planning decisions $y^{(i)}$ and the corresponding budget allocation $q_s^{(i)}$:
\begin{subequations}
\label{eq:subproblem}
\begin{align}
f_s^{(i)} = \min \quad
& \sum_{t\in\mathcal{T}_s} c_{s,t}^\top x_{s,t}
\label{eq:subproblem_obj} \\
\text{s.t.} \quad
& A_s x_s + B_s y \le d_s
\label{eq:subproblem_op} \\
& \sum_{t\in\mathcal{T}_s} e_{s,t}^\top x_{s,t} \le q_s
\label{eq:subproblem_budget} \\
& y = y^{(i)}
&& :\lambda_s^{(i)}
\label{eq:subproblem_fix_y} \\
& q_s = q_s^{(i)}
&& :\pi_s^{(i)}
\label{eq:subproblem_fix_q} \\
& x_{s,t} \ge 0
&& \forall t \in \mathcal{T}_s
\label{eq:subproblem_domain}
\end{align}
\end{subequations}
where $(\lambda_s^{(i)}, \pi_s^{(i)})$ are the Lagrangian multipliers, or dual variables associated with constraints fixing $y$ and $q_s$, namely Eq. \ref{eq:subproblem_fix_y} and Eq. \ref{eq:subproblem_fix_q}, respectively. Here, we assume that all SPs of Eq. \ref{eq:subproblem} remain feasible for any values of $y^{(i)}$ and $q^{(i)}_s$, which can be easily achieved by introducing slack variables in the power supply-demand balance constraint in Eq. \ref{eq:subproblem_op} and the emissions budget constraint in Eq. \ref{eq:subproblem_budget} that are penalized in the SP objective function.

The incumbent upper bound at iteration $i$ is defined as:
\begin{equation}
\label{eq:ub_update}
UB^{(i)} =
\min_{k\in\{0,\dots,i\}}
\left\{
c_I^\top y^{(k)} + \sum_{s\in\mathcal{S}} f_s^{(k)}
\right\},
\end{equation}
where $(y^*, q^*)$ denotes the planning and budget solution associated with the best $UB^{(i)}$. Then, the MP in Eq. \ref{eq:multicut_master} is solved to obtain the new planning and budget solution $(y^{(i+1)}, q^{(i+1)})$:

\begin{subequations}
\label{eq:multicut_master}
\begin{align}
\min \quad
& c_I^\top y + \sum_{s\in\mathcal{S}} \theta_s
\label{eq:multicut_master_obj} \\
\text{s.t.} \quad
& \theta_s \ge
f_s^{(k)} + (\lambda_s^{(k)})^\top (y-y^{(k)}) + \pi_s^{(k)} (q_s-q_s^{(k)})
&& \forall s\in\mathcal{S},\;\forall k\in\{0,\dots,i\}
\label{eq:multicut_master_cut} \\
& \sum_{s\in\mathcal{S}} q_s = \bar{Q}
\label{eq:multicut_master_budget} \\
& R y \le r
\label{eq:multicut_master_plan} \\
& y \ge 0,\;\; y_j \in \mathbb{Z}
&& \forall j \in \mathcal{I}
\label{eq:multicut_master_y_domain} \\
& q_s \ge 0
&& \forall s\in\mathcal{S}
\label{eq:multicut_master_domain} \\
& \theta_s \ge 0
&& \forall s\in\mathcal{S}
\label{eq:multicut_theta_domain}
\end{align}
\end{subequations}

where \(\theta_s\) is an auxiliary variable used to approximate the recourse cost (i.e., optimal operating cost) of SP \(s\), and Eq. \ref{eq:multicut_master_cut} represents the Benders optimality cuts constructed using \(f_s^{(k)}\) and \((\lambda_s^{(k)}, \pi_s^{(k)})\) obtained from solving the SPs at previous Benders iterations \(k \in \{0,\dots,i\}\). 

The lower bound $LB^{(i)}$ is then updated as the optimal objective value of the MP in Eq. \ref{eq:multicut_master}. In cases without budgeting variables, the terms involving $q_s$ and $\pi_s^{(k)}$ in Eq. \ref{eq:multicut_master_cut} are omitted, together with the budgeting constraint Eq. \ref{eq:multicut_master_budget}.

Finally, convergence is checked by computing the optimality gap and a user-defined tolerance $\epsilon_{\mathrm{tol}}$:
\begin{equation}
\label{eq:benders_gap}
\frac{UB^{(i)}-LB^{(i)}}{LB^{(i)}} \le \epsilon_{\mathrm{tol}}.
\end{equation}

If convergence is not achieved after each MP solve, an interior-point level-set regularization step is used to select a stabilized planning solution for evaluating SPs in the next iteration as described in Eq. \ref{eq:reg_problem} in \ref{sec:benders_SI_algorithms}. In MILP instances, a two-stage procedure was used following \citep{pecci_regularized_2025}  where Benders cuts are first generated under a linear relaxation stage before enforcing integrality in the MP. More details of the benchmark multi-cut formulation, including algorithms, can be found in \ref{sec:benders_SI_algorithms}.


\subsection{Clustering-enhanced grouped-cut BD}
\label{subsec:ch4_grouped}

Grouped Benders cut methods reduce MP growth by aggregating Benders cuts across groups of SPs. The grouped-cut variants considered in this study differ in how SP group assignments are constructed and how recourse variables are represented in the MP. In \textit{fix-G-S}, the group assignments are fixed before BD iterations by clustering time-dependent input data. In \textit{adapt-G-S} and \textit{adapt-G-I}, group assignments are updated dynamically across BD iterations by clustering SP dual variables. The shared-recourse variants (\textit{fix-G-S} and \textit{adapt-G-S}) use one recourse variable per group, while \textit{adapt-G-I} retains one per SP.

Although grouping strategies have been explored in the Benders literature, they remain limited and existing approaches differ significantly from the method proposed here. For example, \cite{ramirez-pico_benders_2023} developed a method for two-stage stochastic programs in which SPs are aggregated under a fixed grouping and solved to convergence. The grouping is refined only when exactness conditions fail at convergence, with updates occurring between successive applications of the BD algorithm to the aggregated problem.  

In the main-text presentation, we focus on the \textit{adapt-G-I} formulation, with the formulations and algorithms for \textit{fix-G-S} and \textit{adapt-G-S} provided in \ref{sec:benders_SI_algorithms_fix} and \ref{sec:benders_SI_algorithms_adapt_shared} respectively. All grouped-cut BD formulations use the same regularization and two-stage integer recovery procedure as the multi-cut benchmark described in \ref{sec:benders_SI_algorithms_reg}. This ensures that any performance differences can be attributed solely to the cut grouping approach.


\subsubsection{Adaptive grouped-cut formulation with individual recourse}
\label{subsubsec:ch4_adaptive_individual_recourse}
The \textit{adapt-G-I} representation retains a recourse variable $\theta_s$ for each SP $s \in \mathcal{S}$, which allows the MP to approximate the recourse cost at the SP level, similar to the multi-cut formulation. Let $\mathcal{P}^{(k)} = \{\mathcal{S}_g^{(k)}\}_{g\in\mathcal{G}^{(k)}}$ denote the SP grouping used to generate grouped-cuts at Benders iteration $k$. Since the recourse variables remain indexed at the SP level, each \(\theta_s\) preserves the same interpretation throughout the algorithm, even when group assignments change. Therefore, previously generated grouped cuts remain valid across Benders iterations and can be utilized directly without any modification. The MP at iteration $i$ is:
\begin{subequations}
\label{eq:adaptive_disaggregated_master}
\begin{align}
\min \quad
& c_I^\top y + \sum_{s\in\mathcal{S}} \theta_s
\label{eq:adaptive_disaggregated_master_obj} \\
\text{s.t.} \quad
& \sum_{s\in\mathcal{S}_g^{(k)}} \theta_s \ge
\sum_{s\in\mathcal{S}_g^{(k)}}
\left[
f_s^{(k)} + (\lambda_s^{(k)})^\top (y-y^{(k)}) + \pi_s^{(k)} (q_s-q_s^{(k)})
\right]
&& \forall k \in \{0,\dots,i\},\;\forall g\in\mathcal{G}^{(k)}
\label{eq:adaptive_disaggregated_master_cut} \\
& \sum_{s\in\mathcal{S}} q_s = \bar{Q}
\label{eq:adaptive_disaggregated_master_budget} \\
& R y \le r
\label{eq:adaptive_disaggregated_master_plan} \\
& y \ge 0,\;\; y_j \in \mathbb{Z}
&& \forall j \in \mathcal{I}
\label{eq:adaptive_disaggregated_master_y_domain} \\
& q_s \ge 0
&& \forall s\in\mathcal{S}
\label{eq:adaptive_disaggregated_master_domain} \\
& \theta_s \ge 0
&& \forall s\in\mathcal{S}
\label{eq:adaptive_disaggregated_theta_domain}
\end{align}
\end{subequations}

The \textit{adapt-G-I} procedure with individual recourse is summarized in Algorithm \ref{alg:adaptive_disaggregated}, with the mathematical validity of the formulation established in \ref{sec:benders_SI_validity_adaptive_individual}.

\begin{algorithm}[!tb]
\caption{Adaptive grouped-cut BD with individual recourse (\textit{adapt-G-I})}
\label{alg:adaptive_disaggregated}
\begin{algorithmic}[1]
\REQUIRE Tolerance $\epsilon_{\mathrm{tol}}$, maximum iterations $I_{\max}$, 
    regrouping frequency
\STATE \textbf{Initialization:}
\STATE Solve MP in Eq. \ref{eq:adaptive_disaggregated_master} without Benders cuts 
    to obtain $(y^{(0)},q^{(0)})$
\FOR{$i = 0,\dots,I_{\max}$}
    \STATE Solve all SPs in Eq. \ref{eq:subproblem} using $(y^{(i)},q^{(i)})$
    \STATE Compute $UB^{(i)}$ using Eq. \ref{eq:ub_update}
    \IF{$i$ is a regrouping iteration}
        \STATE Construct the current group assignment $\mathcal{P}^{(i)} = 
            \{\mathcal{S}_g^{(i)}\}_{g\in\mathcal{G}^{(i)}}$ by clustering 
             the SPs based on their optimal dual variables $(\lambda_s^{(i)},\pi_s^{(i)})$
    \ENDIF
    \STATE For each $g\in\mathcal{G}^{(i)}$, construct one grouped Benders cut 
        of the form in Eq. \ref{eq:adaptive_disaggregated_master_cut} by summing 
        the SP-level cut coefficients over all 
        $s\in\mathcal{S}_g^{(i)}$
    \STATE Add the resulting grouped Benders cuts to MP in
        Eq. \ref{eq:adaptive_disaggregated_master}
    \STATE Solve the updated MP to obtain $LB^{(i)}$
    \IF{$\dfrac{UB^{(i)} - LB^{(i)}}{LB^{(i)}} \le \epsilon_{\mathrm{tol}}$}
        \STATE \textbf{Return} $(y^*,q^*)$
    \ELSE
        \STATE Solve the level-set regularization problem in Eq. \ref{eq:reg_problem} with grouped cuts to obtain $(y^{(i+1)},q^{(i+1)})$
    \ENDIF
\ENDFOR
\end{algorithmic}
\end{algorithm}

\subsection{Adaptive representative-subproblem formulation}
\label{subsec:ch4_partial_sp}

Clustering can also be leveraged to reduce the number of SPs solved per BD iteration, which is especially valuable when the number of SPs exceeds available CPUs and solving all SPs requires multiple sequential batches. Building on the adaptive grouping framework, the \textit{rep-SP} method alternates between full-SP BD iterations, which solve all SPs, and representative-SP BD iterations, which solve only one representative SP per group. 

At full-SP iterations, clustering is performed on SP dual variables to generate an SP group assignment and identify one representative SP per group, selected as the SP with dual vector closest to the cluster centroid. At representative-SP iterations, only the representative SPs from the most recent grouping are solved, and only their corresponding cuts are added to the MP. As such, the upper bound is updated only at full-SP iterations, and convergence is certified only at those iterations. Nonetheless, if the optimality gap falls below the convergence tolerance at a representative-SP iteration, a full-SP iteration is triggered to certify convergence. Full-SP solves also occur during an initial warm-start phase for $L_{\mathrm{warm}}$ iterations, as early iterations tend to exhibit more diverse SP responses when the planning solution is far from optimality. Under the two-stage MILP solution procedure adopted here as summarized in Algorithm \ref{alg:two_stage_milp} in \ref{sec:benders_SI_algorithms_reg}, the \textit{rep-SP} method is applied only to the linear-relaxation phase, since prior work \citep{pecci_regularized_2025} has shown that phase dominates overall BD algorithm runtime.

Although methods that solve subsets of SPs have been studied previously, they differ substantially from the \textit{rep-SP} method proposed here. For example, \cite{mazzi_benders_2021} developed oracle-based Benders methods where only a subset of SPs is solved exactly at a given iteration, with the remaining SPs represented using adaptive oracles that provide valid inexact cuts and upper bounds. Likewise, \cite{allen_improvements_2023} developed a machine-learning-based framework that selects SPs to solve using the historical variance of SP objective values across iterations, with periodic solving of the full SP set to recover feasible upper bounds and evaluate convergence. In contrast, the method proposed here selects representative SPs based on adaptive clustering of SP dual variables, therefore exploiting similarity in dual variables among SPs within clustering iterations, rather than historical variation of individual SP behavior across iterations. The \textit{rep-SP} method is summarized in Algorithm \ref{alg:partial_sp} in \ref{sec:benders_SI_algorithms}, and is evaluated for a selected case study under varying degrees of SP parallelization, as well as a stochastic case study in Section \ref{subsec:ch4_results_rep_SP}.

\subsection{Case studies and input assumptions}
\label{subsec:ch4_case_studies}

The clustering-enhanced BD methods were evaluated using MACRO for 2050 planning case studies based on 11, 20, and 26 zones representations of the Eastern U.S.\ electricity system, with zonal definitions based on regions defined in the EPA integrated planning model (IPM) \citep{us_epa_documentation_2018} as summarized in Table \ref{tab_benders_SI_case_study_regions}, and the geographical scope shown in Figure \ref{Benders_Regions}A. 

Existing generation capacity and operating-cost assumptions were obtained from PowerGenome as shown in Table \ref{tab_benders_SI_existing_capacity_ipm} \citep{schivley_powergenomepowergenome_2023}, while existing inter-regional transmission capacities were based on EPA IPM as shown in Table \ref{tab_benders_SI_transmission} \citep{us_epa_documentation_2018}. Cost and performance assumptions for greenfield power-sector technologies were obtained from NREL's Annual Technology Baseline (ATB) 2024 as shown in Table \ref{tab_Benders_SI_Greenfield} \citep{mirletz_2024_2024}, with regional capital-cost multipliers and fuel-price assumptions extracted from EIA's Annual Energy Outlook (AEO) 2023 as shown in Table \ref{tab_benders_SI_regional_cost_multipliers} \citep{us_eia_annual_2023}. Generation and transmission expansion decisions are represented using integer investment capacities in MACRO (Table \ref{tab_benders_SI_discrete_capacity}), resulting in a MILP formulation.

Hourly regional electricity demand was based on Princeton Net-Zero America's high-electrification (E+) scenario for 2050~\citep{larson_net-zero_2021} as shown in Table \ref{tab_benders_SI_regional_demand}, in which its original state-level profiles were mapped to the IPM regions using county-level population weights derived from the 2021 U.S.\ Census Bureau data \citep{us_census_bureau_county_nodate}. Hourly wind and solar resource availability and supply curves for each modeled region were characterized using data corresponding to 2011 weather year from the ZEPHYR (Zero-emissions Electricity system Planning with HourlY operational Resolution) framework as shown in Table \ref{tab_benders_SI_vre_cf_capacity} \citep{brown_zephyr_2022}, where site-level maximum capacity and hourly capacity factors were derived from NREL datasets such as the National Solar Radiation Database (NSRDB)~\citep{sengupta_national_2018}, WIND Toolkit (WTK)~\citep{draxl_wind_2015}, and Regional Energy Deployment System (ReEDS) model~\citep{reeds_modeling_and_analysis_team_regional_2021}. Table \ref{tab:ch4_case_study_summary} summarizes the total electricity demand, number of integer planning variables, continuous variables, and constraints for the 11, 20, and 26-zone case studies. Additional input assumptions are provided in  \ref{sec:benders_SI_inputs}.

\begin{figure}[h]
\centering
\includegraphics[width=1\textwidth]{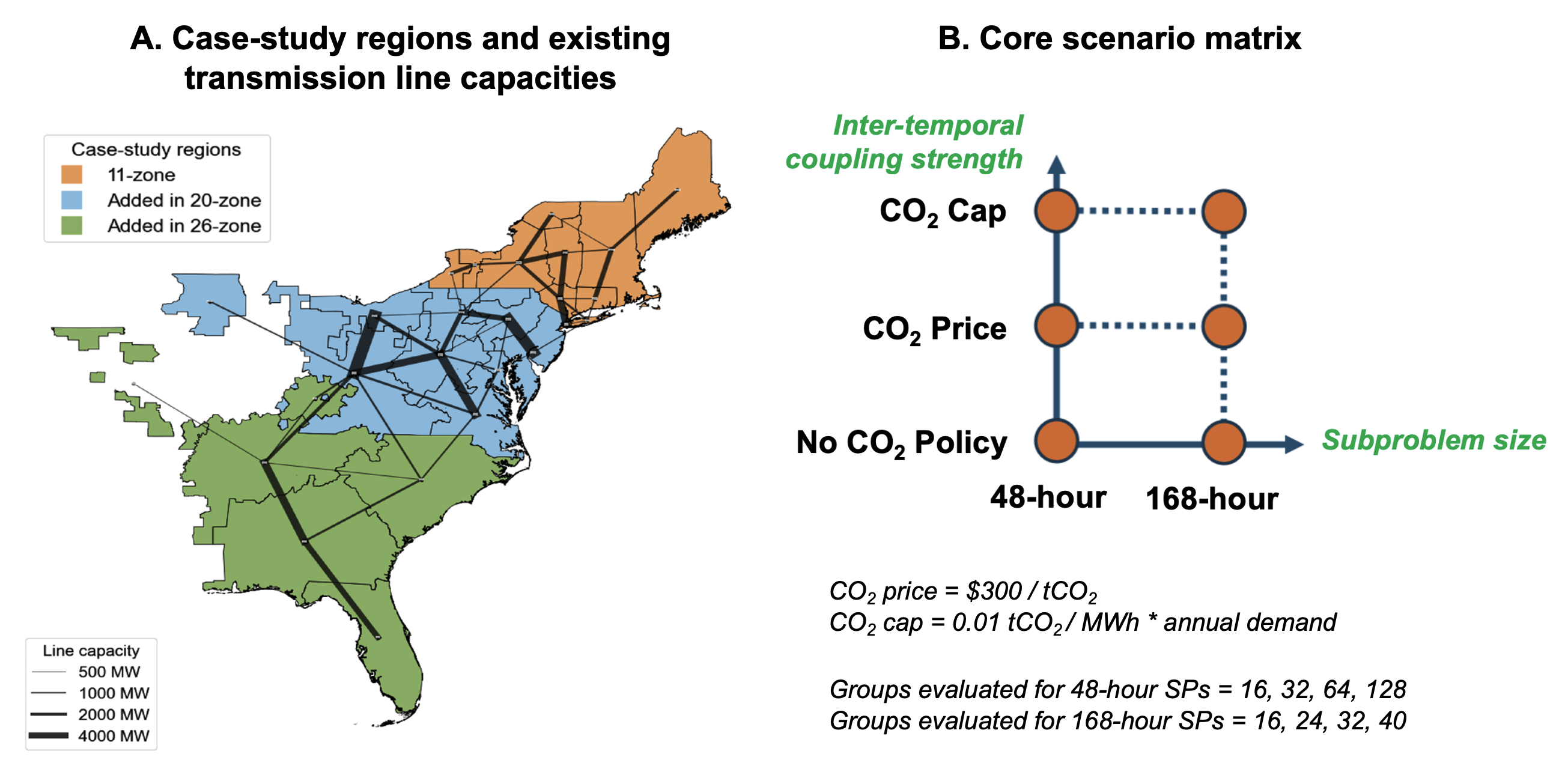}
\caption[Geographical scope and scenario design]{A) Geographical scope of the 11, 20, and 26-zone case studies based on IPM regions, with existing transmission capacities from EPA data \citep{us_epa_documentation_2018}. B) Core scenario design across SP sizes and inter-temporal coupling strengths. Grouped Benders methods were evaluated using 16, 32, 64, and 128 groups for 48-hour SPs, and 16, 24, 32, and 40 groups for 168-hour SPs.}
\label{Benders_Regions}
\end{figure}

\begin{table}[h]
\centering
\caption[Summary of Benders case studies]{Summary of the 11, 20, and 26-zone case studies evaluated in the BD analysis.}
\label{tab:ch4_case_study_summary}
\begin{tabular}{lccc}
\toprule
& \multicolumn{3}{c}{Number of zones} \\
\cmidrule(lr){2-4}
Case study & 11 & 20 & 26 \\
\midrule
Total electricity demand (TWh)              & 583   & 1{,}935 & 3{,}315 \\
Integer planning variables ($|\mathcal{I}|$)& 318   & 607     & 791     \\
Continuous variables ($\times 10^6$)        & 7.41   & 14.95    & 19.73    \\
Constraints ($\times 10^6$)                 & 13.49  & 27.21    & 35.96    \\
Number of existing generators               & 90  & 189    & 251    \\
Number of candidate generators              & 117  & 213    & 276    \\
Number of transmission lines                & 15  & 33    & 43    \\
\bottomrule
\end{tabular}
\end{table}

\subsection{Scenarios}
\label{subsec:ch4_compute}

All case studies were solved over a 52-week operational horizon (8,736~h) across various SP sizes and inter-temporal coupling strengths shown in the scenario matrix in Figure \ref{Benders_Regions}B. To study the effect of SP size on the various BD algorithms, we introduced: 1) a 48-hour SP setup with a total of $|\mathcal{S}| = 182$ temporal SPs, and 2) a 168-hour SP setup with $|\mathcal{S}| = 52$. To evaluate the effect of inter-temporal coupling strength, we considered three policy settings: 1) No CO$_2$ policy, 2) CO$_2$ price, and 3) implementing a hard annual CO$_2$ cap. 

In the no CO$_2$ policy case, SPs are only coupled via the MP planning decisions with no budgeting variables present (i.e., terms involving \(q_s\) and \(\pi_s^{(k)}\) are omitted). In the CO$_2$ price scenario, an emission penalty of \$300/tCO$_2$ is implemented via the SP objective function, and likewise does not include budgeting variables. In the CO$_2$ cap case, a hard system-wide annual emission cap constraint is implemented, where SPs are coupled through budgeting variables as described in Eq. \ref{eq:subproblem_budget} and Eq. \ref{eq:multicut_master_budget}. The annual CO$_2$ cap for each case study was defined as 0.01~tCO$_2$/MWh multiplied by the total electricity demand in Table \ref{tab:ch4_case_study_summary}. Energy storage such as batteries and hydropower storage were modeled as short-term storage, and state-of-charge is not carried across SPs. To prevent infeasible solutions across Benders iterations, slack variables with sufficiently large penalty costs were added to the electricity demand balance and CO$_2$ cap constraints. As a result, all Benders cuts generated in this study are optimality cuts.

\subsection{Implementation}
\label{subsec:ch4_implementation}

All case studies were implemented in Julia 1.10.4 using JuMP v1.25 and solved using Gurobi 12.0.3. Computational experiments were performed on the MIT Engaging high-performance computing (HPC) cluster using AMD EPYC 9654 96-core processor nodes with 384~GB of memory \citep{office_of_research_computing_and_data_mit_about_nodate}. Unless otherwise stated, SPs were distributed across CPU cores to minimize the solving of sequential SP batches subject to the total 384~GB memory limit, with the resulting CPU-core allocations summarized in Table \ref{tab:ch4_compute_setup}. In the 48-hour SP setup, the total number of 182 SPs cannot be fully parallelized within this memory limit, and CPU-core allocations were selected to give the smallest feasible integer number of SP batches: 61 cores can solve the 182 SPs in three batches for the 11 and 20-zone cases, while 46 cores can solve them in four batches for the 26-zone case. The lower core count for the 26-zone case reflects the higher memory requirement of each SP as the model size increases. In contrast, the 168-hour SP setup only has 52 SPs, which can be solved concurrently within the memory limit for all model sizes. Therefore, utilizing more than 52 CPU cores would not further improve SP parallelization.

\begin{table}[htbp]
\centering
\caption[Default CPU allocation across Benders case studies]{Default CPU-core allocation used to maximize SP parallelization across case studies for the 48-hour and 168-hour SP setups, subject to the 384~GB memory limit of the HPC nodes.}
\label{tab:ch4_compute_setup}
\vspace{6pt}
\begin{tabular}{lcccccc}
\toprule
& \multicolumn{3}{c}{48-hour SPs ($|\mathcal{S}| = 182$)} 
& \multicolumn{3}{c}{168-hour SPs ($|\mathcal{S}| = 52$)} \\
\cmidrule(lr){2-4} \cmidrule(lr){5-7}
Number of zones & 11 & 20 & 26 & 11 & 20 & 26 \\
\midrule
CPU cores & 61 & 61 & 46 & 52 & 52 & 52 \\
\bottomrule
\end{tabular}
\end{table}

Each computational run was subjected to a maximum 12~h time limit as imposed by the employed HPC cluster, and case studies exceeding this limit were classified as intractable. Convergence was defined when the Benders optimality gap falls below $\epsilon_{\mathrm{tol}} = 10^{-3}$. For all adaptive algorithms, clustering was performed every 4 Benders iterations using k-means. For consistency, a fixed random seed (\texttt{MersenneTwister}~=~42) was used for the k-means process, and clustering was repeated up to 300 times, with early stopping if no improvement was recorded after 50 consecutive restarts. For \textit{rep-SP}, all SPs were solved for an initial warm-start phase of $L_{\mathrm{warm}} = 5$ iterations before representative-SP solving was applied with clustering every 4 Benders iterations. The Gurobi solver settings were specified to be consistent with the benchmark implementation in \cite{pecci_regularized_2025}, where the barrier algorithm (\texttt{Method}~=~2) was used for both MP and SPs with \texttt{BarConvTol} set to $10^{-3}$. Crossover was enabled for SPs (\texttt{Crossover}~=~1) and disabled for MP (\texttt{Crossover}~=~0), and \texttt{Threads}~=~1 was specified for the SPs.

\section{Results}
\label{sec:ch4_results}

\subsection{Grouped-cut BD performance under no CO$_2$ policy scenario}
\label{subsec:ch4_results_no_policy}

First, we examine the computational performance of the grouped-BD formulations under the no CO$_2$ policy scenario, which does not introduce budgeting variables, across the 11, 20, and 26-zone case studies. All Benders methods achieved the same optimal objective value as the benchmark multi-cut. In addition, monolithic implementations of the original MILP CEM formulation in Eq. \ref{eq:budget_problem} were not solved within the 12~h time limit across all cases, and were considered intractable. Figure~\ref{Fig_benders_R1} compares total runtime across the benchmark multi-cut, \textit{fix-G-S}, \textit{adapt-G-S}, and \textit{adapt-G-I} formulations under the no CO$_2$ policy case.

\begin{figure}[h]
\centering
\includegraphics[width=0.75\textwidth]{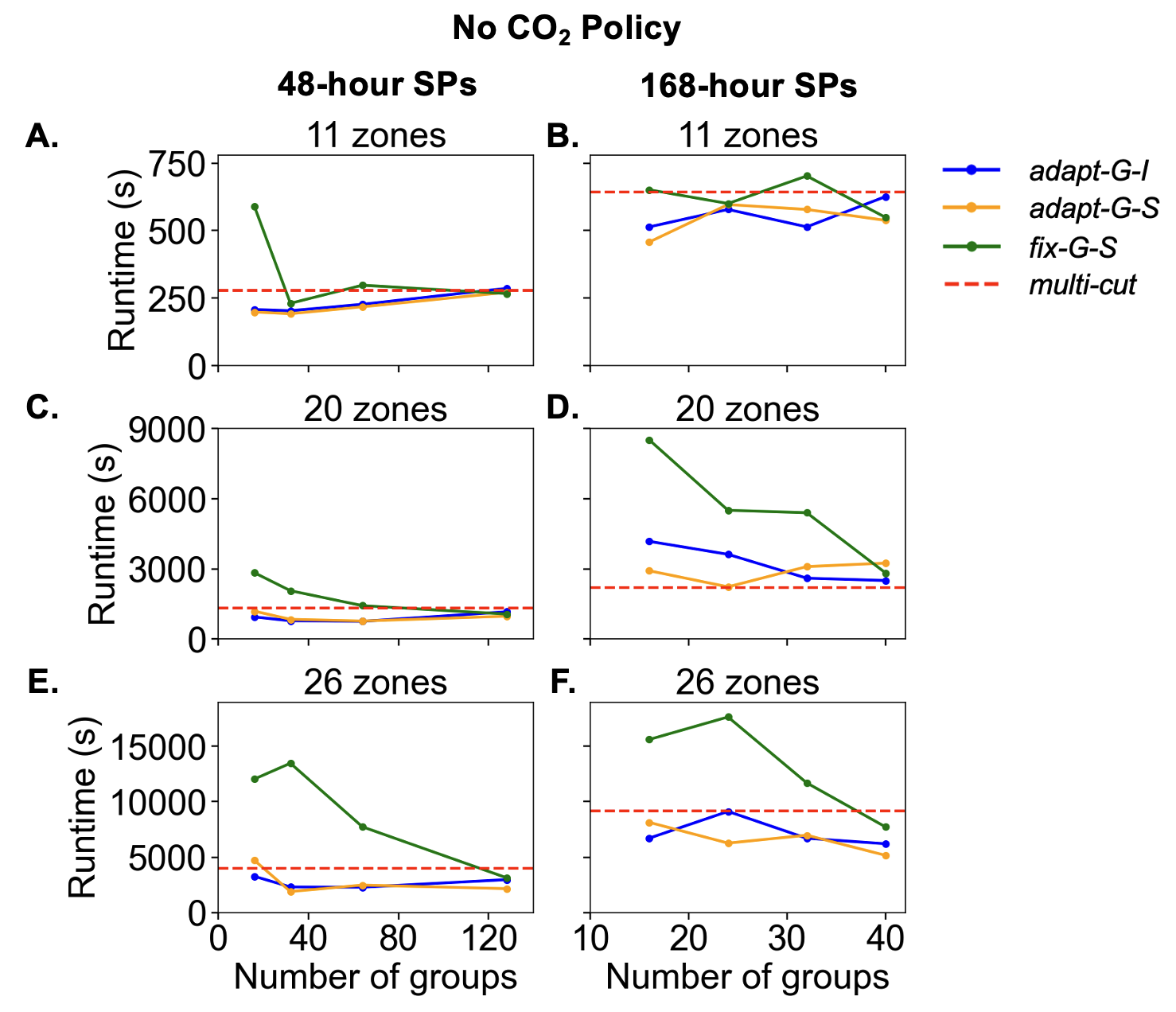}
\caption[Runtime comparison of Benders methods under no CO$_2$ policy]{Total runtime under the no CO$_2$ policy scenario for the \textit{fix-G-S}, \textit{adapt-G-S}, and \textit{adapt-G-I} formulations across 11, 20, and 26-zone systems with 48-hour and 168-hour SPs. The dashed red line shows the benchmark multi-cut runtime. Group counts follow Figure \ref{Benders_Regions}B, and additional iterations and runtime breakdowns are reported in Tables \ref{tab:si_no_co2_48h_detail}--\ref{tab:si_no_co2_168h_breakdown}.}
\label{Fig_benders_R1}
\end{figure}

Figure~\ref{Fig_benders_R1} shows that \textit{adapt-G-S} and \textit{adapt-G-I}  consistently outperform \textit{fix-G-S}, with the magnitude of runtime benefit depending on the SP size and modeled system size. In the 48-hour SP setup, the advantage of adaptive methods becomes more pronounced as system size increases. In the 11-zone case, the multi-cut benchmark required a runtime of 277~s, whereas the best-performing adaptive method (\textit{adapt-G-S} with 32 groups) reduced the runtime to 192~s, corresponding to 31\% speed up (Table \ref{tab:si_no_co2_48h_detail}). Similarly, adaptive grouping results in the fastest runtime in the 20-zone case, reducing runtime from 1322~s to 753~s, a 43\% improvement over the benchmark (Table \ref{tab:si_no_co2_48h_detail}). The largest gains occurred in the 26-zone case, where the best-performing adaptive method (\textit{adapt-G-S} with 32 groups) solved the model to optimality in 1891~s as compared to 4006~s for the benchmark, corresponding to 53\% speed up (Table \ref{tab:si_no_co2_48h_detail}). In contrast, the \textit{fix-G-S} formulation showed weaker performance and frequently underperformed the benchmark, particularly at low group counts. Across the 48-hour SP setup, both \textit{adapt-G-S} and \textit{adapt-G-I} yield similar performance, suggesting that the choice between recourse representations matters less compared to whether the grouping is updated adaptively. Compared with the 48-hour SP setup, the benefit of adaptive grouping is smaller and less consistent for the 168-hour SP setup. The best-performing adaptive formulation reduced runtime by 29\% in the 11-zone case and 44\% in the 26-zone case, but did not outperform the benchmark in the 20-zone case.

The detailed breakdown of runtime in Tables \ref{tab:si_no_co2_48h_breakdown} and \ref{tab:si_no_co2_168h_breakdown} further explains why adaptive grouping is advantageous in larger system sizes, particularly in the 48-hour SP setup. For the benchmark multi-cut formulation, MP solution accounts for an increasingly large share of total runtime as system size increases, from 21.3\% in the 11-zone case to 54.8\% in the 26-zone case (Table \ref{tab:si_no_co2_48h_breakdown}), indicating that reducing MP growth yields greater computational gains in larger systems where MP solve time dominates.

As grouped-BD methods mainly affect runtime by reducing Benders cut growth in the MP, their performance depends on the trade-off between limiting MP growth across iterations and preserving effective convergence behavior. Figure~\ref{Fig_benders_R2} highlights this trade-off by reporting iteration counts and per-iteration average MP and SP solve times for the 20-zone system under both 48-hour and 168-hour SP setups. Across all grouped-BD formulations, the per-iteration average SP solve time remains broadly similar within each setup since SPs are still solved individually, while MP solve time decreases consistently with decreasing group count, particularly in the 48-hour SP setup (Figure~\ref{Fig_benders_R2}E).

\begin{figure}[h]
\centering
\includegraphics[width=0.8\textwidth]{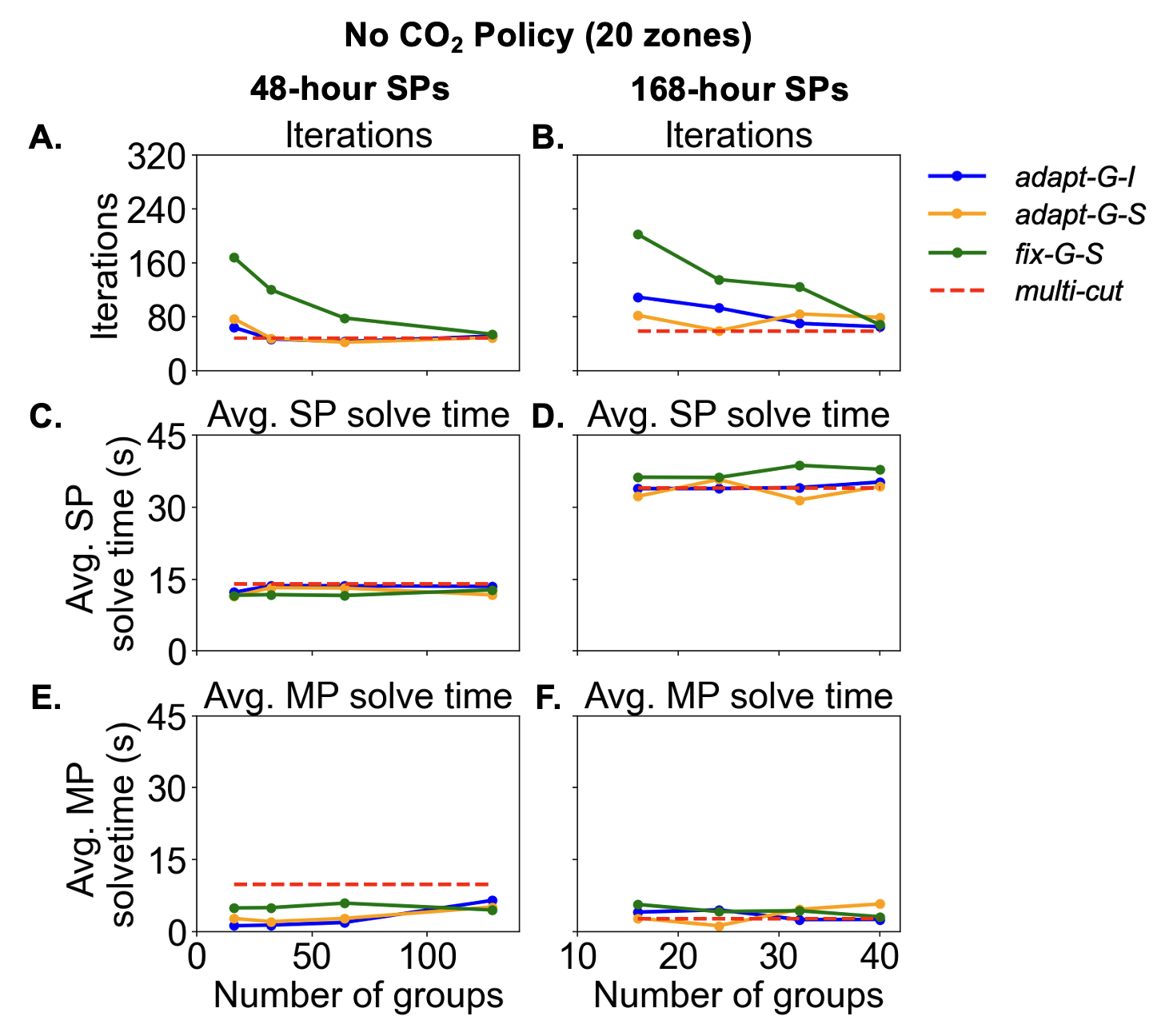}
\caption[Detailed performance of Benders methods under no CO$_2$ policy]{Total runtime, iterations, and average SP and MP solve times for the 20-zone no CO$_2$ policy case under 48-hour (A, C, E) and 168-hour (B, D, F) SP setups. The dashed red line shows the multi-cut benchmark.}
\label{Fig_benders_R2}
\end{figure}

Under the 48-hour SP setup, MP solve time accounted for 36.8\% of total runtime in the multi-cut benchmark, and represented a substantial share of computational burden (Table \ref{tab:si_no_co2_48h_breakdown}). The adaptive formulations reduce this burden substantially. For example, in \textit{adapt-G-I} with 64 groups, the average MP solve time decreases from 9.9~s in the benchmark to 1.9~s, while total iterations decrease from 49 to 43 (Table \ref{tab:si_no_co2_48h_avg_times}). In this case, adaptive grouping not only reduces MP burden but also slightly improves convergence behavior. 

As shown in Figure~\ref{Fig_benders_R2}A, iteration counts increase substantially at low number of groups (e.g., 16). This is consistent with the fact that coarser grouping places more heterogeneous SPs together, which results in each grouped cut aggregating dual variables over more diverse SPs. This could obscure meaningful differences across dual variables within each group, thus weakening the effectiveness of the grouped Benders cut, and slowing overall convergence. At intermediate group counts (e.g., 32 or 64), the grouping preserves sufficient variations in SP dual variables across groups to maintain favorable convergence performance while still limiting MP growth, yielding the best trade-off between MP solve time and iteration count. At high group counts (e.g., 128), the marginal benefit of adaptive grouping diminishes as the formulation approaches multi-cut behavior. The fact that intermediate group counts achieve similar or even lower iteration counts than the benchmark suggests that grouped cuts can preserve sufficient information for effective convergence without requiring the full set of individual SP Benders cuts. This is possible when sufficient redundancy exists across SP dual variable outcomes, as observed in the no CO$_2$ policy scenario.

Figure~\ref{Fig_benders_R2} also explains the weaker performance of the \textit{fix-G-S} formulation relative to its adaptive counterparts. Its per-iteration MP solve time reduction is smaller (Figure~\ref{Fig_benders_R2}E), and it requires substantially more iterations to converge, especially at intermediate group counts where adaptive methods perform best (Figure~\ref{Fig_benders_R2}A). These results suggest that static input-based clustering produces cuts that are less informative for approximating recourse cost than adaptive grouping based on SP dual variables, which more directly reflects the structure of individual Benders cuts.

In contrast, the 168-hour SP setup is dominated by SP solution time, with average per-iteration SP solve times substantially exceeding MP solve times across all methods (Figure~\ref{Fig_benders_R2}D and Figure~\ref{Fig_benders_R2}F). Consequently, MP solve time reductions from grouped cuts have a smaller effect on overall runtime. The 168-hour setup also contains fewer SPs, so grouped-BD methods yield a smaller absolute reduction in Benders cuts. Additionally, the longer temporal horizon may cause each SP to span a broader range of operating conditions, reducing the number of SPs with closely similar dual variables and the quality of grouped cuts. As a result, the 168-hour setup shows a less favorable trade-off between iteration counts and MP solve time reduction for adaptive grouped-cut methods (Figure~\ref{Fig_benders_R2}B and Figure~\ref{Fig_benders_R2}F), and do not outperform the benchmark multi-cut in the 20-zone case.

\subsection{Impact of CO$_2$ policy}
\label{subsec:ch4_results_co2}

\subsubsection{CO$_2$ price policy}
Figure~\ref{Fig_benders_R3}A -- F shows total runtime for the 11, 20, and 26-zone case studies under the CO$_2$ price scenario for both 48-hour and 168-hour SP setups. Since the \textit{fix-G-S} formulation consistently underperformed in the no CO$_2$ policy scenario, we focused our analysis on \textit{adapt-G-S}, \textit{adapt-G-I}, and the benchmark multi-cut BD methods. Under this scenario, \textit{adapt-G-S} and \textit{adapt-G-I} formulations generally do not outperform the benchmark, although the performance gap remains narrow in certain 48-hour SP cases.

\begin{figure}[h]
\centering
\includegraphics[width=1\textwidth]{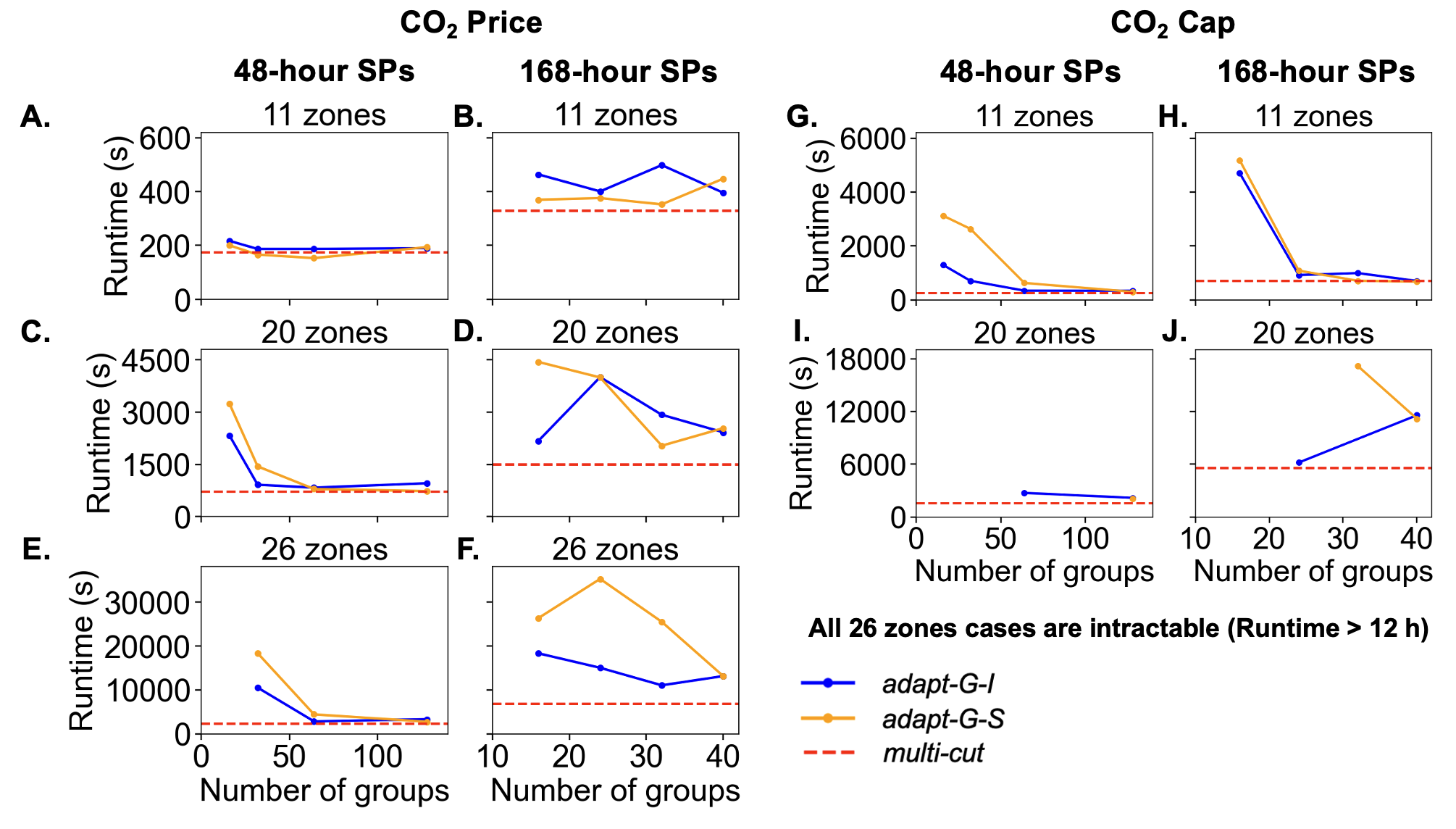}
\caption[Runtime comparison of Benders methods under CO$_2$ policies]{Total runtime under the CO$_2$ price scenario of \$300/tCO$_2$ (A--F) and CO$_2$ cap scenario (G--J) for the benchmark multi-cut, \textit{adapt-G-S}, and \textit{adapt-G-I} formulations. Intractable cases are omitted, including all 26-zone CO$_2$ cap cases. The annual CO$_2$ cap is defined as 0.01~tCO$_2$/MWh multiplied by total electricity demand. Additional iterations and runtime breakdowns are reported in Tables \ref{tab:si_co2_price_48h_detail}--\ref{tab:si_co2_cap_168h_breakdown}.}
\label{Fig_benders_R3}
\end{figure}

For the 48-hour SP setup, the adaptive grouped method outperforms the benchmark only in the 11-zone system, where \textit{adapt-G-S} reduces overall runtime from 173.8~s to 153.2~s, corresponding to an improvement of 12\%  (Table \ref{tab:si_co2_price_48h_detail}). In the 20 and 26-zone systems, the benchmark multi-cut formulation yields the fastest overall runtime (714~s and 2247~s respectively), although the best-performing adaptive grouped-cut formulation, with 128 groups, remains relatively close (731~s and 2741~s respectively). In the 168-hour SP setup, the benchmark multi-cut has a clear advantage across all three system sizes following the reasons discussed earlier (Table \ref{tab:si_co2_price_168h_detail}). As in the no CO$_2$ policy scenario, the \textit{adapt-G-S} and \textit{adapt-G-I} formulations perform similarly in the 48-hour SP setup, whereas greater variation in their relative performance is observed in the 168-hour SP setup.

Similar to the no CO$_2$ policy scenario, the performance of the adaptive grouped-cut method under the CO$_2$ price policy depends strongly on the number of groups. This is particularly evident in the 26-zone, 48-hour SP setup, where the 16-group adaptive grouped-cut formulations become intractable, whereas the corresponding no CO$_2$ policy cases remain solvable (Figure~\ref{Fig_benders_R3}E vs. Figure~\ref{Fig_benders_R1}E). Although benchmark multi-cut runtime is largely unchanged, the penalty for coarse grouping increases significantly under the CO$_2$ price (Figure~\ref{Fig_benders_R3}E vs. Figure~\ref{Fig_benders_R1}E), indicating that the CO$_2$ price worsens performance at low group counts and weakens the benefit of adaptive grouping at intermediate group counts. Consequently, larger group counts are required for adaptive grouped-cut formulations to approach, but generally not surpass, multi-cut benchmark performance.

This weaker performance stems from a less favorable trade-off between MP solve time reduction and higher iteration counts (Tables \ref{tab:si_co2_price_48h_detail} -- \ref{tab:si_co2_price_168h_avg_times}), suggesting that the CO$_2$ price affects how effectively SPs can be clustered by dual variable similarity. One possible explanation is that the CO$_2$ price penalty increases heterogeneity in SP dual variables by strengthening the dependence of the operating cost term in Eq. \ref{eq:subproblem_obj} on operating conditions. Specifically, variations in VRE availability across SPs affect the dispatch of emitting generators, and under a CO$_2$ price, this causes dual variables to a fixed planning decision to vary more strongly across SPs. This weakens grouped-cut effectiveness, requiring larger group counts to approach multi-cut benchmark performance (Figure~\ref{Fig_benders_R3}). Overall, adaptive grouped-cut formulations are less effective under the CO$_2$ price scenario.

\subsubsection{CO$_2$ cap policy}
 

Finally, in the CO$_2$ cap policy scenario, annual emissions are enforced through budgeting variables allocated across individual SPs (Eq. \ref{eq:multicut_master_budget}). Figure~\ref{Fig_benders_R3}G--J shows total runtime for the 11 and 20-zone cases under both 48-hour and 168-hour SP setups; all 26-zone cases, including the benchmark, are intractable. Runtime increases substantially across all methods, and adaptive grouped-cut formulations show no apparent advantage over the benchmark.

Performance of adaptive grouping deteriorates sharply under this scenario in both SP setups, especially at low group counts where several cases become intractable. Only the highest group counts yield runtime comparable to the benchmark across tractable cases. Among these, \textit{adapt-G-I} generally outperforms \textit{adapt-G-S}, though the improvement is limited and adaptive grouping remains much less effective overall under the CO$_2$ cap scenario.

A possible explanation is that under the CO$_2$ cap scenario, the MP determines not only the shared planning decisions \(y\) but also one budgeting variable \(q_s\) per SP, with allocations constrained to sum to the annual CO$_2$ cap (Eq. \ref{eq:multicut_master_budget}). This induces stronger SP coupling than in scenarios without budgeting variables, and the dual variables used for clustering now include multiple capacity-related duals in \(\lambda_s\) and one emission-related dual \(\pi_s\) for each SP. Since similarity in planning-related duals does not imply similarity in budget-related duals, the resulting heterogeneity reduces similarity among individual Benders cuts and the proportion of redundant cuts. Moreover, because each SP carries its own emissions budget, grouping SPs can average out the emission-specific information the MP needs to refine allocations across SPs. Together, these effects explain why adaptive grouped cuts fail to deliver meaningful MP solve time improvements under the CO$_2$ cap scenario. 

\subsection{Representative-subproblem method under constrained computational resources}
\label{subsec:ch4_results_rep_SP}
We evaluate the \textit{rep-SP} method under varying computational resources, defined by the ratio of SPs to CPUs, where high SPs-to-CPU values indicate low parallelization and vice versa. We focus on the 20-zone case with 48-hour SPs considering all three policy scenarios in two cases: 1) The same single-weather-year deterministic setup evaluated in the previous sections based on 2011 weather year, and deliberately vary the SPs-to-CPU ratio by changing the number of CPU cores allocated to the job, and 2) a stochastic case with three weather years from 2010--2012, where the maximum feasible CPU allocation is used subject to the 384~GB node memory limit.

As described in Section \ref{subsec:ch4_partial_sp}, the \textit{rep-SP} method solves the full set of SPs and adds the full set of Benders cuts only at regrouping iterations, which occur every four iterations in our numerical experiments.  In the intervening iterations, only the representative SP from each group identified by clustering the SP dual variables at the most recent regrouping step is solved, and only the corresponding Benders cuts are added to the MP.

Figure \ref{fig:partial_sp_resource_48h_20z}A--C compares total runtime across the various CO$_2$ policy scenarios for the \textit{rep-SP}, \textit{adapt-G-I}, and benchmark multi-cut formulations for the single-weather-year setup under varying SPs-to-CPU ratios. To keep the main-text comparison focused, we show \textit{adapt-G-I} with 64 groups and \textit{rep-SP} with 16 representative SPs. Figure \ref{fig:partial_sp_resource_48h_20z}D--F then compares total runtime across numbers of groups in the three-weather-year stochastic case.

\subsubsection{Single-weather-year case with varying SPs-to-CPU ratios}
\label{subsubsec:ch4_results_rep_SP_cpu}
Using the same single-weather-year deterministic setup evaluated in previous sections, we compare \textit{rep-SP},  \textit{adapt-G-I}, and the benchmark multi-cut methods under three computational settings. These settings correspond to SPs-to-CPU ratios of 3, 11, and 23, obtained by allocating 61, 16, and 8 CPU cores to the job, respectively. A higher SPs-to-CPU ratio indicates more limited parallelization, because each CPU must solve more SPs sequentially. 

\begin{figure}[h]
\centering
\includegraphics[width=0.9\textwidth]{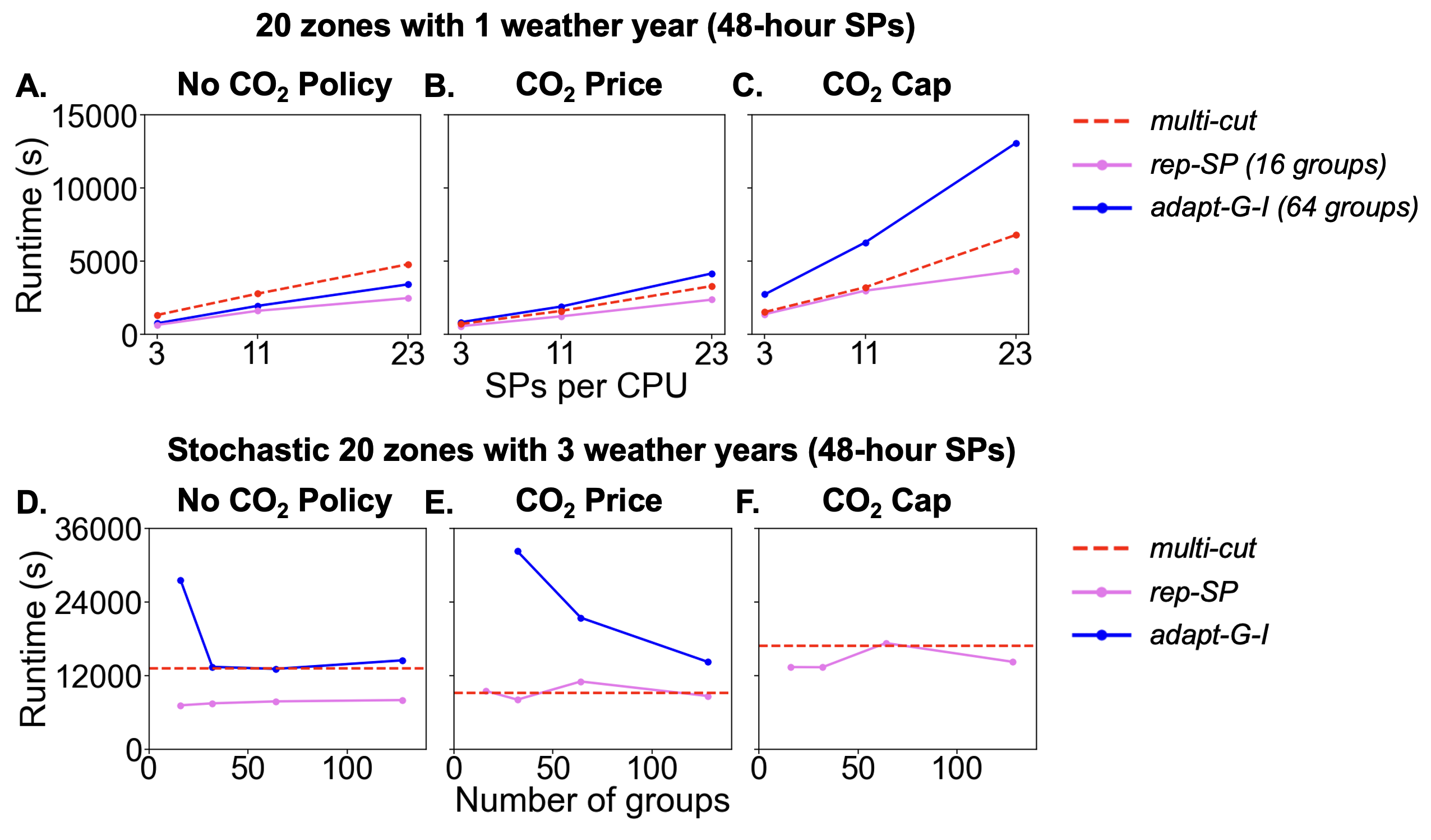}
\caption[Runtime under constrained SP parallelization]{Plots A, B, and C: Runtime of \textit{adapt-G-I} with 64 groups and \textit{rep-SP} with 16 representative SPs for the single-weather-year 20-zone case with 48-hour SPs under different SPs-to-CPU ratios under the no CO$_2$ policy, CO$_2$ price, and CO$_2$ cap cases, respectively. The 64-group \textit{adapt-G-I} case represents an intermediate grouped-cut setting, while the 16-\textit{rep-SP} case illustrates a strong SP-reduction setting evaluated. Plots D, E, and F: Runtime of \textit{adapt-G-I} and \textit{rep-SP} across different numbers of groups or representative SPs for the stochastic 20-zone case with 48-hour SPs and three weather years across CO$_2$ policy settings. The dashed red line shows the benchmark multi-cut runtime. In the CO$_2$ price scenario (Panel E), the 16-group \textit{adapt-G-I} setting is intractable. In the CO$_2$ cap scenario (Panel F), all group counts of \textit{adapt-G-I} are intractable. Additional results for the total iterations, detailed runtime breakdowns, and average component solve times are reported in Table \ref{tab:si_no_co2_48h_20z_detail} -- \ref{tab:si_stochastic_48h_20z_breakdown} in \ref{sec:benders_SI_results}.}
\label{fig:partial_sp_resource_48h_20z}
\end{figure}

As shown in Figure \ref{fig:partial_sp_resource_48h_20z}A--C, \textit{rep-SP} performs similarly to the benchmark multi-cut when SP parallelization is high, and its runtime advantage increases as the SPs-to-CPU ratio rises. At low SPs-to-CPU ratios, most SPs can be solved concurrently, thus solving only representative SPs in intermediate iterations provides limited additional benefit. As parallelization decreases, SP solution accounts for a larger share of total runtime, and reducing the number of SPs solved between full iterations leads to greater runtime savings.

While the no-policy cases (Figure \ref{fig:partial_sp_resource_48h_20z}A) show benefits of both \textit{adapt-G-I} and \textit{rep-SP} across all SPs-to-CPU ratios, \textit{adapt-G-I} performs worse in the CO$_2$ price cases (Figure \ref{fig:partial_sp_resource_48h_20z}B), with \textit{rep-SP} remaining effective particularly at higher SPs-to-CPU ratios. The CO$_2$ cap cases (Figure \ref{fig:partial_sp_resource_48h_20z}C) further show that \textit{rep-SP} is more robust than grouped cuts under strong inter-temporal coupling. When budgeting variables are present, aggregating cuts can weaken the SP-specific information needed for effective convergence, making \textit{adapt-G-I} less effective. In contrast, \textit{rep-SP} periodically solves the full SP set and adds the corresponding individual Benders cuts at regrouping iterations. These full-SP iterations preserve detailed SP-level cut information, while the intervening representative-SP iterations reduce the SP solution burden. As a result, \textit{rep-SP} remains effective under the strongly coupled CO$_2$ cap setting and outperforms the multi-cut benchmark when SP parallelization is limited.

\subsubsection{Stochastic three-weather-year case}
\label{subsubsec:ch4_results_rep_SP_stochastic}
We next evaluate the same methods in a stochastic formulation with three operational weather-year realizations for hourly solar and wind availability. In this formulation, shared investment decisions are optimized across the three weather years of system operation, which are assigned equal probability. The operational component of the objective therefore represents the expected operating cost across the three realizations, leading to a tripling in number of SPs compared to the single weather year CEM. In the CO$_2$ cap case, the annual emissions cap is enforced separately for each weather-year realization, rather than as an expected-emissions cap.

Unlike the previous SPs-to-CPU experiment, the stochastic case does not deliberately vary CPU availability. Instead, it uses the maximum feasible CPU allocation under the available node memory limit (384~GB). With three weather years, the number of 48-hour SPs increases from 182 to 546. Given the memory-constrained allocation, only 8 CPUs can be used, which corresponds to approximately 68 SPs per CPU. This is substantially more constrained than the default single-weather-year setup, which uses 61 CPUs for 182 SPs, or approximately 3 SPs per CPU. At the same time, adding weather years can introduce additional redundancy across SPs, since similar operational patterns may recur across different weather-year realizations (e.g., days in the same season). This makes the stochastic case particularly relevant for evaluating whether representative SPs can exploit redundancy as the number of operational scenarios increases.

The results in Figure \ref{fig:partial_sp_resource_48h_20z}D--F show that \textit{rep-SP} remains effective when the number of SPs increases substantially in the stochastic setup. Although the number of SPs triples from 182 to 546, the best-performing \textit{rep-SP} configurations still use a relatively small number of representative SPs: 16 in the no CO$_2$ policy case (Figure \ref{fig:partial_sp_resource_48h_20z}D)  and 32 in both the  CO$_2$ price and CO$_2$ cap cases (Figure \ref{fig:partial_sp_resource_48h_20z}E and Figure \ref{fig:partial_sp_resource_48h_20z}F). This confirms our hypothesis that considering multiple weather-year uncertainty introduces substantial redundancy across SPs, as similar seasonal operating periods can recur across different weather years.

While the no-policy cases (Figure \ref{fig:partial_sp_resource_48h_20z}D) show a clear advantage of \textit{rep-SP} over both \textit{adapt-G-I} and the benchmark multi-cut formulation, the CO$_2$ price cases (Figure \ref{fig:partial_sp_resource_48h_20z}E) show comparable performance of \textit{rep-SP} to the multi-cut benchmark, as reductions in SP solution time are partly offset by an increase in Benders iterations (\ref{tab:si_stochastic_48h_20z_avg}). Nevertheless, \textit{adapt-G-I} performs substantially worse than \textit{rep-SP} in the CO$_2$ price cases and becomes intractable under the CO$_2$ cap cases (Figure \ref{fig:partial_sp_resource_48h_20z}F). Under strong inter-temporal coupling in the CO$_2$ cap cases, \textit{rep-SP} remains computationally feasible and outperforms the benchmark multi-cut formulation under most group counts. This indicates that \textit{rep-SP} is particularly useful in stochastic settings where the number of SPs is large relative to available CPU resources.

Together, these results show that \textit{rep-SP} targets a different computational bottleneck than grouped cuts. Grouped cuts primarily reduce MP growth by limiting the number of Benders cuts. In contrast, \textit{rep-SP} reduces the number of SPs solved in intermediate iterations and also limits MP growth during those iterations, since only the Benders cuts from representative SPs are added to the MP. This makes \textit{rep-SP} especially useful when SP solution dominates total runtime, when available computational resources limit SP parallelization or when stochastic case studies substantially increase the number of SPs.

\section{Discussion}
\label{sec:ch4_discussion}

\subsection{Effectiveness of clustering-enhanced BD}

Our analysis establishes the conditions under which clustering-enhanced BD improves computational performance, as well as the conditions where its benefits become limited relative to a multi-cut benchmark formulation. In the electricity-sector capacity expansion setting considered here, grouped Benders cuts are most effective when 1) operational responses to the MP solution are sufficiently similar across SPs, and 2) MP solution constitutes a substantial share of the overall computational burden. Under these conditions, adaptive grouping can reduce the number of Benders cuts added to the MP in each iteration while maintaining convergence behavior close to the benchmark multi-cut formulation. Importantly, adaptive grouping based on clustering SP dual variables performs substantially better than fixed grouping based on exogenous time-series inputs, as dual variables more directly reflect similarities among individual Benders cuts as the planning solution evolves.

The effectiveness of adaptive grouped cuts also depends strongly on the degree of coupling between SPs. When operational decisions across SPs are weakly coupled, as in the no CO$_2$ policy scenario evaluated here, adaptive grouped cuts can yield substantial runtime improvements, indicating that relatively high degrees of Benders cut aggregation are possible in such cases. As inter-SP coupling becomes stronger, more detailed cut information is necessary to maintain efficient convergence. Adaptive grouped cuts therefore become less effective under the CO$_2$ price scenario and deteriorate further under the CO$_2$ cap scenario, where budgeting variables strongly couple SPs through a shared system-wide emissions cap. The substantial increase in overall runtime observed even for the benchmark multi-cut formulation under the CO$_2$ cap scenario further indicates that the computational challenge arises not only from the grouping strategy itself, but also from the underlying strong inter-temporal coupling induced by budgeting variables.

\subsection{Computational and decomposition insights}

The computational results also show that performance across BD methods depends strongly on the source of dominant computational burden. Adaptive grouped-cut formulations are most effective when MP solution dominates overall runtime, whereas \textit{rep-SP} strategies are most valuable when SP solution is limited by the degree of available parallelization. In particular, grouped-cut strategies are most valuable in model settings with many easy-to-evaluate SPs, where the accumulation of individual Benders cuts can substantially increase the MP size and solve time across iterations.

SP-dominant behavior can arise when individual SPs are large and expensive to solve, or when the model contains many SPs that cannot be parallelized efficiently under limited CPU and memory resources. In the latter case, reducing the number of SPs solved in each iteration can provide substantial computational benefit, as reflected in the strong performance of the \textit{rep-SP} method under constrained computational resources. The stochastic case further illustrates this point. Increasing the number of weather years substantially increases the SP computational burden, but can also create more similar operational structures for \textit{rep-SP} methods to exploit. By solving only representative SPs in intervening iterations, \textit{rep-SP} reduces computational burden while maintaining effective convergence behavior. This suggests that representative-SP methods may be particularly relevant for CEMs that incorporate operational uncertainty through scenarios, where the number of SPs can grow rapidly relative to available parallel computational resources.

In contrast, when SP runtime burden is driven by the large size of individual SPs, \textit{rep-SP} is less likely to address this bottleneck. In such cases, further decomposition of large SPs into smaller SPs may be more relevant, potentially in combination with \textit{rep-SP}. These results also highlight the importance of SP parallelization for tractability in BD, since the best-performing formulation under limited CPU resources remained slower than the worst-performing case under the default CPU setting. A further implication of this study is the increasing computational benefit of grouped Benders cuts with system size. In the MILP formulation considered here, increasing system size expands the set of planning decisions in the MP, including the number of integer variables, thus making the MP increasingly expensive to solve. As a result, MP solution accounts for a larger share of total runtime, and strategies to reduce Benders cut growth become more valuable.

Finally, the intractability of all monolithic MILP cases within the 12~h time limit reinforces the broader role of decomposition methods in enabling tractable high-resolution planning models. Taken as a whole, this study shows that no single BD strategy performs best across all cases. Rather, the preferred method depends on model structure, including overall system scale, SP temporal horizon, the degree of inter-SP coupling, and the number of operational scenarios, as well as on the available computational resources.

\subsection{Limitations and future work}

We identified several limitations in this study that point to directions for future work. First, the analysis was conducted for a single-sector electricity CEM, with cases varying inter-temporal coupling only through CO$_2$ emission policy setups. Although the main computational mechanisms identified here are unlikely to change, the adaptive grouped-cut and \textit{rep-SP} formulations may still perform differently under other optimization model structures or forms of inter-temporal coupling. In addition, the \textit{rep-SP} analysis was limited in scope, focusing only on selected cases intended to examine how constrained computational resources affect SP-side burden. Its broader performance across other SP sizes and coupling regimes therefore remains uncertain. Future work should evaluate both adaptive grouped-cut and \textit{rep-SP} methods across a wider range of model settings.

Second, budgeting variables were present only in the CO$_2$ cap scenario, where all clustering-enhanced BD algorithms performed worst, particularly the grouped-cut formulations. In this study, budgeting variables were not used for other inter-temporal constraints encountered in CEMs \citep{jacobson_computationally_2024, pecci_regularized_2025}, such as long-term energy storage dynamics or renewable share constraints. Representing such dynamics or constraints would introduce additional budgeting variables that further strengthen coupling across SPs, which may further reduce the effectiveness of both grouped-cut and multi-cut BD formulations. In addition, decarbonization studies are increasingly moving toward multi-sector energy systems that couple electricity with low-carbon hydrogen and sustainable fuels production \citep{shaker_multi-sectoral_2025, law_role_2025}. Achieving computationally feasible solutions of these models at high spatial and temporal resolution, especially when integer variables are present, may require both temporal and sectoral decomposition approaches \citep{parolin_sectoral_2026}. This could lead to substantially stronger coupling across temporal and sectoral SPs through a much larger set of linking constraints. The comparatively increased runtime observed in the CO$_2$ cap vs. no policy cases therefore points to a broader challenge for BD under budgeting-variable formulations. Future work should investigate methods to improve BD performance under such formulations, and whether alternative representations of inter-SP coupling constraints are computationally preferable.

Third, the adaptive methods developed in this study rely only on SP dual variables for clustering. While this choice is aligned with the structure of Benders cuts, it excludes other potentially useful information about SP responses, such as objective values, which likewise contribute to Benders cuts. This could matter in cases where similarity in dual variables does not fully reflect similarity across SPs. Further, in cases involving budgeting variables, the dual values associated with planning and budgeting constraints may differ substantially in both dimension and numerical scale, which can affect how well clustering identifies SPs with similar Benders cuts. Future work should therefore consider methods for scaling and weighting different types of dual variables, and for incorporating additional SP features such as objective values into the clustering procedure.

Fourth, this study only evaluates k-means clustering for adaptive grouped cuts and \textit{rep-SP} methods. Moreover, the clustering step introduced non-negligible computational overhead, particularly in the smaller models. For example, the average runtime share of clustering is approximately 12.2\% in the 11-zone case, compared to approximately 2.8\% in the 26-zone case (Table \ref{tab:si_no_co2_48h_breakdown}). Future work should therefore examine alternative clustering methods that can identify similarities among SPs more accurately and efficiently. This could include other machine-learning-based approaches \citep{barbar_representative_2022}, if they can provide faster and more effective grouping without substantially increasing overhead.

Finally, although this study includes a three-weather-year stochastic case to demonstrate the potential value of \textit{rep-SP} when SP counts increase, the analysis remains limited to a small set of equally weighted weather-year realizations. Future work should evaluate larger stochastic formulations with varying scenario probabilities, longer weather-year samples, and multi-year planning models \citep{goke_stabilized_2024, homaei_high-capture-rate_2026, sodwatana_appliance_2025}. Such settings could produce larger numbers of SPs, making full SP parallelization difficult even with significant computational resources. They could also provide clearer guidance on when increasing SP-side computational burden shifts the preferred BD strategy from grouped-cut methods toward \textit{rep-SP} methods.

\section*{Supplementary information}

Detailed computational results from all evaluated case studies, modeling input assumptions, and model implementation details are provided in the Supplementary Information.

\section*{Acknowledgments}
This research was supported by the MIT Energy Initiative. The authors gratefully acknowledge Nicole Shi for preparing the power system dataset used in the study. The authors acknowledge the MIT Office of Research Computing and Data for providing high performance computing resources that have contributed to the research results reported within this paper. The views expressed herein are solely those of the authors.

\newpage
\appendix

\renewcommand{\thealgorithm}{A\arabic{algorithm}}
\setcounter{algorithm}{0}
\setcounter{table}{0}
\setcounter{figure}{0}

\section{Additional BD formulations and algorithms}
\label{sec:benders_SI_algorithms}

\subsection{Regularized multi-cut benchmark}
\label{sec:benders_SI_algorithms_reg}

The benchmark multi-cut formulation is summarized in Algorithm \ref{alg:benchmark_multicut}. To stabilize planning decisions across iterations, the benchmark multi-cut formulation uses an interior-point level-set regularization method of Pecci \& Jenkins \cite{pecci_regularized_2025}. If convergence is not achieved after solving the updated MP in Eq. \ref{eq:multicut_master}, the following regularization problem is solved, whereby $\Phi(y,q)$ is a convex function, and $\alpha \in (0,1)$ is a level-set parameter:
\begin{subequations}
\label{eq:reg_problem}
\begin{align}
\min_{y,\{q_s\},\{\theta_s\}} \quad
& \Phi(y,q)
\label{eq:reg_problem_obj} \\
& c_I^\top y + \sum_{s\in\mathcal{S}} \theta_s
\le
LB^{(i)} + \alpha\left(UB^{(i)}-LB^{(i)}\right)
\label{eq:level_set_bound}\\
\text{s.t.} \quad
& \theta_s \ge
f_s^{(k)} + (\lambda_s^{(k)})^\top (y-y^{(k)}) + \pi_s^{(k)} (q_s-q_s^{(k)})
&& \forall s\in\mathcal{S},\;\forall k\in\{0,\dots,i\}
\label{eq:reg_problem_cut} \\
& \sum_{s\in\mathcal{S}} q_s = \bar{Q}
\label{eq:reg_problem_budget} \\
& R y \le r
\label{eq:reg_problem_plan} \\
& y \ge 0,\;\; y_j \in \mathbb{Z}
&& \forall j \in \mathcal{I}
\label{eq:reg_problem_y_domain} \\
& q_s \ge 0
&& \forall s\in\mathcal{S}
\label{eq:reg_problem_domain} \\
& \theta_s \ge 0
&& \forall s\in\mathcal{S}
\label{eq:reg_problem_theta_domain} 
\end{align}
\end{subequations}

Unless otherwise noted, $\alpha = 0.5$ is used in this study, which was found to perform best over a range of values tested in Pecci \& Jenkins \cite{pecci_regularized_2025}. In the benchmark implementation, $\Phi(y,q) = 0$ is used for non-integer formulations, in which Eq. \ref{eq:reg_problem} reduces to a feasibility problem. Solving this with a barrier method provides an interior-point solution within the feasible region of the MP in Eq. \ref{eq:multicut_master}, which reduces oscillation between extreme planning decisions across iterations and improves convergence behavior. The planning decisions obtained by solving Eq. \ref{eq:reg_problem} then replace \((y^{(i+1)}, q^{(i+1)})\), and are used to solve the SPs in the next iteration. 
When integer planning variables are present, regularization is implemented through a two-stage procedure described in Algorithm \ref{alg:two_stage_milp}. First, the BD formulation is solved under linear relaxation until convergence, in which generated Benders cuts are retained. In the second stage, integrality is enforced in the MP in Eq. \ref{eq:multicut_master}, which is solved as a MILP. The integer planning variables obtained from the MP in Eq. \ref{eq:multicut_master} are then fixed in the regularization problem Eq. \ref{eq:reg_problem}, where the regularization step only applies to continuous planning variables. This maintains the stabilization effect of the regularization step while preserving an integer-feasible planning solution. 

\begin{algorithm}[!tb]
\caption{Benchmark Regularized Multi-Cut BD \citep{pecci_regularized_2025}}
\label{alg:benchmark_multicut}
\begin{algorithmic}[1]
\REQUIRE Tolerance $\epsilon_{\mathrm{tol}}$, maximum iterations $I_{\max}$
\STATE \textbf{Initialization:}
\STATE Solve MP in Eq. \ref{eq:multicut_master} without Benders cuts to obtain 
    $(y^{(0)},q^{(0)})$
\FOR{$i = 0,\dots,I_{\max}$}
    \STATE Solve all SPs in Eq. \ref{eq:subproblem} using $(y^{(i)},q^{(i)})$
    \STATE Store Benders cut coefficients 
        $(f_s^{(i)},\lambda_s^{(i)},\pi_s^{(i)})$ for all $s\in\mathcal{S}$
    \STATE Compute $UB^{(i)}$ using Eq. \ref{eq:ub_update}
    \STATE Add Benders cuts to MP in Eq. \ref{eq:multicut_master}
    \STATE Solve the updated MP to obtain $LB^{(i)}$
    \IF{$\dfrac{UB^{(i)} - LB^{(i)}}{LB^{(i)}} \le \epsilon_{\mathrm{tol}}$}
        \STATE \textbf{Return} $(y^*,q^*)$
    \ELSE
        \STATE Solve the level-set regularization problem in Eq. \ref{eq:reg_problem} to obtain $(y^{(i+1)},q^{(i+1)})$
    \ENDIF
\ENDFOR
\end{algorithmic}
\end{algorithm}

\begin{algorithm}[H]
\caption{Benchmark Two-Stage Procedure for MILP \cite{pecci_regularized_2025}}
\label{alg:two_stage_milp}
\begin{algorithmic}[1]
\REQUIRE Tolerance $\epsilon_{\mathrm{tol}}$, maximum iterations $I_{\max}$
\STATE \textbf{Stage 1: Linear-relaxation phase}
\STATE Run Algorithm \ref{alg:benchmark_multicut} on the linear relaxation of the model until convergence
\STATE Keep all Benders cuts generated from the linear-relaxation phase
\STATE \textbf{Stage 2: MILP phase}
\STATE Run Algorithm \ref{alg:benchmark_multicut} with integrality enforced in MP in Eq. \ref{eq:multicut_master}, and fix integer planning variables $y_j$ for $j \in \mathcal{I}$ when solving the regularization problem
\end{algorithmic}
\end{algorithm}

\FloatBarrier
\newpage


\subsection{Fixed grouped-cut formulation with shared recourse}
\label{sec:benders_SI_algorithms_fix}


For the \textit{fix-G-S} formulation, we determined the SP group assignment by clustering the time-dependent input data, such as hourly VRE and demand profiles associated with each SP. This group assignment is then fixed throughout the BD iterations. Let $\mathcal{P} = \{\mathcal{S}_g\}_{g\in\mathcal{G}}$ denote a group assignment of the subproblem set $\mathcal{S}$, and each group $\mathcal{S}_g \subseteq \mathcal{S}$ contains the SPs assigned to group $g$. 
Let $\theta_g$ denote the recourse variable used to approximate the 
recourse cost of group $g \in \mathcal{G}$. The resulting MP is formulated as:
\begin{subequations}
\label{eq:fixed_grouped_master}
\begin{align}
\min \quad
& c_I^\top y + \sum_{g\in\mathcal{G}} \theta_g
\label{eq:fixed_grouped_master_obj} \\
\text{s.t.} \quad
& \theta_g \ge
\sum_{s\in\mathcal{S}_g}
\left[
f_s^{(k)} + (\lambda_s^{(k)})^\top (y-y^{(k)}) + \pi_s^{(k)} (q_s-q_s^{(k)})
\right]
&& \forall g\in\mathcal{G},\;\forall k\in\{0,\dots,i\}
\label{eq:fixed_grouped_master_cut} \\
& \sum_{s\in\mathcal{S}} q_s = \bar{Q}
\label{eq:fixed_grouped_master_budget} \\
& R y \le r
\label{eq:fixed_grouped_master_plan} \\
& y \ge 0,\;\; y_j \in \mathbb{Z}
&& \forall j \in \mathcal{I}
\label{eq:fixed_grouped_master_y_domain} \\
& q_s \ge 0
&& \forall s\in\mathcal{S}
\label{eq:fixed_grouped_master_domain}\\
& \theta_g \ge 0
&& \forall g\in\mathcal{G}
\label{eq:fixed_grouped_theta_domain}
\end{align}
\end{subequations}
Each constraint in Eq. \ref{eq:fixed_grouped_master_cut} is a grouped Benders cut constructed by summing Benders cut coefficients $(f_s^{(k)}, \lambda_s^{(k)}, \pi_s^{(k)})$ obtained from the SP solutions within $s \in \mathcal{S}_g$ for each Benders iteration $k \in \{0,\dots,i\}$.

Compared to the benchmark multi-cut BD formulation, the \textit{fix-G-S} formulation reduces both the number of MP recourse variables and the number of added cuts per BD iteration from $|\mathcal{S}|$  to $|\mathcal{G}|$. Since the group assignment remains unchanged, previously generated grouped cuts remain mathematically valid across all Benders iterations.

The \textit{fix-G-S} procedure is summarized in Algorithm \ref{alg:fixed_grouped} with the mathematical validity of the formulation established in \ref{sec:benders_SI_validity_fixed}.

\begin{algorithm}[H]
\caption{Fixed grouped-cut BD with shared recourse (\textit{fix-G-S})}
\label{alg:fixed_grouped}
\begin{algorithmic}[1]
\REQUIRE Tolerance $\epsilon_{\mathrm{tol}}$, maximum iterations $I_{\max}$
\STATE \textbf{Initialization:}
\STATE Construct fixed SP group assignments $\{\mathcal{S}_g\}_{g\in\mathcal{G}}$ 
    by clustering input time-series data associated with the SPs
\STATE Solve MP in Eq. \ref{eq:fixed_grouped_master} without Benders cuts to obtain 
    $(y^{(0)},q^{(0)})$
\FOR{$i = 0,\dots,I_{\max}$}
    \STATE Solve all SPs in Eq. \ref{eq:subproblem} using $(y^{(i)},q^{(i)})$
    \STATE Compute $UB^{(i)}$ using Eq. \ref{eq:ub_update}
    \STATE For each $g \in \mathcal{G}$, construct one grouped Benders cut of 
        the form in Eq. \ref{eq:fixed_grouped_master_cut} using the current 
        subproblem-level cut coefficients $(f_s^{(i)},\lambda_s^{(i)},
        \pi_s^{(i)})$ over all $s \in \mathcal{S}_g$
    \STATE Add the resulting grouped Benders cuts to MP in
        Eq. \ref{eq:fixed_grouped_master}
    \STATE Solve the updated MP to obtain $LB^{(i)}$
    \IF{$\dfrac{UB^{(i)} - LB^{(i)}}{LB^{(i)}} \le \epsilon_{\mathrm{tol}}$}
        \STATE \textbf{Return} $(y^*,q^*)$
    \ELSE
        \STATE Solve the level-set regularization problem in Eq. \ref{eq:reg_problem} with grouped cuts to obtain $(y^{(i+1)},q^{(i+1)})$
    \ENDIF
\ENDFOR
\end{algorithmic}
\end{algorithm}

\FloatBarrier
\newpage

\subsection{Adaptive grouped-cut formulation with shared recourse}
\label{sec:benders_SI_algorithms_adapt_shared}

In the \textit{adapt-G-S} formulation, each group has a single recourse variable $\theta_g$ approximating the total recourse cost of its assigned SPs. The key difference from the \textit{fix-G-S} formulation is that group assignments are updated periodically during BD iterations by re-clustering SPs based on their dual variables $(\lambda_s^{(k)},\pi_s^{(k)})$. Let $\mathcal{P}^{(i)} = \{\mathcal{S}_g^{(i)}\}_{g\in\mathcal{G}^{(i)}}$ denote the SP grouping used to generate grouped-cuts at Benders iteration $i$, and let 
$\theta_g$ denote the recourse variable used to approximate the recourse cost for group $g \in \mathcal{G}^{(i)}$.

When the group assignment changes during clustering iterations, grouped cuts generated in previous iterations based on previous group assignments are not directly compatible with the current group assignment. To preserve mathematical validity, historical grouped cuts are reconstructed by re-aggregating the stored Benders cut coefficients $(f_s^{(k)}, \lambda_s^{(k)}, \pi_s^{(k)})$ for all iterations $k \in \{0,\dots,i\}$ under the current group assignment $\mathcal{P}^{(i)}$. The MP solved at iteration $i$ is therefore:
\begin{subequations}
\label{eq:adaptive_grouped_master}
\begin{align}
\min \quad
& c_I^\top y + \sum_{g\in\mathcal{G}^{(i)}} \theta_g
\label{eq:adaptive_grouped_master_obj} \\
\text{s.t.} \quad
& \theta_g \ge
\sum_{s\in\mathcal{S}_g^{(i)}}
\left[
f_s^{(k)} + (\lambda_s^{(k)})^\top (y-y^{(k)}) + \pi_s^{(k)} (q_s-q_s^{(k)})
\right]
&& \forall g\in\mathcal{G}^{(i)},\;\forall k\in\{0,\dots,i\}
\label{eq:adaptive_grouped_master_cut} \\
& \sum_{s\in\mathcal{S}} q_s = \bar{Q}
\label{eq:adaptive_grouped_master_budget} \\
& R y \le r
\label{eq:adaptive_grouped_master_plan} \\
& y \ge 0,\;\; y_j \in \mathbb{Z}
&& \forall j \in \mathcal{I}
\label{eq:adaptive_grouped_master_y_domain} \\
& q_s \ge 0
&& \forall s\in\mathcal{S}
\label{eq:adaptive_grouped_master_domain} \\
& \theta_g \ge 0
&& \forall g\in\mathcal{G}^{(i)}
\label{eq:adaptive_grouped_theta_domain}
\end{align}
\end{subequations}

The \textit{adapt-G-S} procedure, including the reconstruction of historical grouped cuts, is summarized in Algorithm \ref{alg:adaptive_grouped}, with the mathematical validity of the formulation established in \ref{sec:benders_SI_validity_adaptive_shared}.

\begin{algorithm}[H]
\caption{Adaptive grouped-cut BD with shared recourse (\textit{adapt-G-S})}
\label{alg:adaptive_grouped}
\begin{algorithmic}[1]
\REQUIRE Tolerance $\epsilon_{\mathrm{tol}}$, maximum iterations $I_{\max}$, 
    regrouping frequency
\STATE \textbf{Initialization:}
\STATE Solve MP in Eq. \ref{eq:adaptive_grouped_master} without Benders cuts to 
    obtain $(y^{(0)},q^{(0)})$
\FOR{$i = 0,\dots,I_{\max}$}
    \STATE Solve all SPs in Eq. \ref{eq:subproblem} using $(y^{(i)},q^{(i)})$
    \STATE Store the subproblem-level cut coefficients 
        $(f_s^{(i)},\lambda_s^{(i)},\pi_s^{(i)})$ for all $s\in\mathcal{S}$
    \STATE Compute $UB^{(i)}$ using Eq. \ref{eq:ub_update}
    \IF{$i = 0$ \textbf{or} $i$ is a regrouping iteration}
        \STATE Construct the current group assignment $\mathcal{P}^{(i)} = 
        \{\mathcal{S}_g^{(i)}\}_{g\in\mathcal{G}^{(i)}}$ by clustering 
        the SPs based on their optimal dual variables $(\lambda_s^{(i)},\pi_s^{(i)})$
        \STATE Reconstruct all historical grouped cuts under $\mathcal{P}^{(i)}$ 
            by re-aggregating stored coefficients for all $k \in \{0,\dots,i\}$
        \STATE Replace all cuts in MP in Eq. \ref{eq:adaptive_grouped_master} with 
            the reconstructed grouped cuts
    \ELSE
        \STATE Construct one new grouped cut for each $g \in \mathcal{G}^{(i)}$ 
            using the current coefficients $(f_s^{(i)},\lambda_s^{(i)},\pi_s^{(i)})$ 
            and add to MP in Eq. \ref{eq:adaptive_grouped_master}
    \ENDIF
    \STATE Solve the updated MP to obtain $LB^{(i)}$
    \IF{$\dfrac{UB^{(i)} - LB^{(i)}}{LB^{(i)}} \le \epsilon_{\mathrm{tol}}$}
        \STATE \textbf{Return} $(y^*,q^*)$
    \ELSE
        \STATE Solve the level-set regularization problem in Eq. \ref{eq:reg_problem} with grouped cuts to obtain $(y^{(i+1)},q^{(i+1)})$
    \ENDIF
\ENDFOR
\end{algorithmic}
\end{algorithm}

\newpage
\subsection{Adaptive representative-subproblem algorithm}
\label{sec:benders_SI_algorithms_rep_sp}

The \textit{rep-SP} method is summarized in Algorithm \ref{alg:partial_sp}.

\begin{algorithm}[!tb]
\caption{Adaptive representative-SP BD (\textit{rep-SP})}
\label{alg:partial_sp}
\begin{algorithmic}[1]
\REQUIRE Tolerance $\epsilon_{\mathrm{tol}}$, maximum iterations $I_{\max}$, 
    regrouping frequency, warm-start length $L_{\mathrm{warm}}$
\STATE \textbf{Initialization:}
\STATE Solve MP in Eq. \ref{eq:multicut_master} without Benders cuts to obtain 
    $(y^{(0)},q^{(0)})$
\FOR{$i = 0,\dots,I_{\max}$}
    \IF{$i < L_{\mathrm{warm}}$ \textbf{or} $i$ is a regrouping iteration}
        \STATE Solve all SPs in Eq. \ref{eq:subproblem} using $(y^{(i)},q^{(i)})$
        \STATE Store the SP-level cut coefficients 
            $(f_s^{(i)},\lambda_s^{(i)},\pi_s^{(i)})$ for all $s\in\mathcal{S}$
        \STATE Compute $UB^{(i)}$ using Eq. \ref{eq:ub_update}
        \IF{$i$ is a regrouping iteration}
            \STATE Update the group assignment $\mathcal{P}^{(i)} = 
                \{\mathcal{S}_g^{(i)}\}_{g\in\mathcal{G}^{(i)}}$ by clustering 
                the SPs based on their optimal dual variables $(\lambda_s^{(i)},\pi_s^{(i)})$, 
                and identify one representative SP for each group
        \ENDIF
        \STATE Add Benders cuts from all solved SPs to MP in
            Eq. \ref{eq:multicut_master}
        \STATE Solve the updated MP to obtain $LB^{(i)}$
        \IF{$\dfrac{UB^{(i)} - LB^{(i)}}{LB^{(i)}} \le \epsilon_{\mathrm{tol}}$}
            \STATE \textbf{Return} $(y^*,q^*)$
        \ELSE
            \STATE Solve the level-set regularization problem described in \ref{sec:benders_SI_algorithms_reg} to obtain $(y^{(i+1)},q^{(i+1)})$
        \ENDIF
    \ELSE
        \STATE Solve only the representative SPs from the current group assignment 
            $\mathcal{P}^{(i)}$ using $(y^{(i)},q^{(i)})$
        \STATE Set $UB^{(i)} \gets UB^{(i_{\mathrm{full}})}$, where 
            $i_{\mathrm{full}}$ denotes the most recent full-solve iteration
        \STATE Add Benders cuts from the representative SPs to MP
            in Eq. \ref{eq:multicut_master}
        \STATE Solve the updated MP to obtain $LB^{(i)}$
        \IF{$\dfrac{UB^{(i)} - LB^{(i)}}{LB^{(i)}} \le \epsilon_{\mathrm{tol}}$}
            \STATE Trigger a full-solve iteration at $i+1$ to certify convergence
        \ENDIF
        \STATE Solve the level-set regularization problem described in \ref{sec:benders_SI_algorithms_reg} to obtain $(y^{(i+1)},q^{(i+1)})$
    \ENDIF
\ENDFOR
\end{algorithmic}
\end{algorithm}

\section{Additional Computational Results}
\label{sec:benders_SI_results}

This section provides additional computational results for the BD case studies in this paper. The tables here report runtime, iterations, per-iteration average MP and SP solve times, and runtime breakdowns for the case studies summarized in Table \ref{tab:ch4_case_study_summary}, across various subproblem sizes and policy settings (Tables~\ref{tab:si_no_co2_48h_detail}--\ref{tab:si_co2_cap_168h_breakdown}), computational resource constraints (Tables \ref{tab:si_no_co2_48h_20z_detail}--\ref{tab:si_co2_cap_48h_20z_breakdown}), and stochastic cases using 3 weather years (Tables \ref{tab:si_stochastic_48h_20z_detail}--\ref{tab:si_stochastic_48h_20z_breakdown}).

\begin{landscape}
\begin{table}[p]
\centering
\caption[Detailed computational results (no-policy scenario with 48-hour SPs)]{Detailed computational results for the no CO$_2$ policy case with 48-hour SPs.}
\label{tab:si_no_co2_48h_detail}
\scriptsize
\setlength{\tabcolsep}{4pt}
\begin{tabular}{c l c c c c c c c}
\toprule
Zones
& Algorithm
& Groups
& \shortstack{Total\\Iterations}
& \shortstack{Total\\Runtime\\(s)}
& \shortstack{Total SP\\Solve Time\\(s)}
& \shortstack{Total Adaptive\\Clustering\\Time (s)}
& \shortstack{Total MP\\Solve Time\\(s)}
& \shortstack{Total Interior\\Problem Solve\\Time (s)} \\
\midrule

\multirow{13}{*}{11}
& \multirow{4}{*}{\shortstack[l]{\textit{adapt-G-S}}}
& 16  & 39 & 198.0 & 158.4 & 21.3 & 8.4  & 2.3 \\
& & 32  & 33 & 191.7 & 142.9 & 23.2 & 14.7 & 3.4 \\
& & 64  & 35 & 217.4 & 154.4 & 25.2 & 18.0 & 8.8 \\
& & 128 & 34 & 270.9 & 163.9 & 39.1 & 35.2 & 18.4 \\

& \multirow{4}{*}{\shortstack[l]{\textit{adapt-G-I}}}
& 16  & 39 & 207.1 & 166.6 & 20.8 & 11.7 & 2.5 \\
& & 32  & 36 & 202.7 & 159.4 & 22.7 & 10.5 & 4.2 \\
& & 64  & 37 & 226.7 & 166.8 & 22.4 & 18.6 & 9.8 \\
\bottomrule
\end{tabular}
\end{table}
\end{landscape}


\begin{landscape}
\begin{table}[p]
\centering
\caption[Detailed computational results (no-policy scenario with 168-hour SPs)]{Detailed computational results for the no CO$_2$ policy case with 168-hour SPs.}
\label{tab:si_no_co2_168h_detail}
\scriptsize
\setlength{\tabcolsep}{4pt}
\begin{tabular}{c l c c c c c c c}
\toprule
Zones
& Algorithm
& Groups
& \shortstack{Total\\Iterations}
& \shortstack{Total\\Runtime\\(s)}
& \shortstack{Total SP\\Solve Time\\(s)}
& \shortstack{Total Adaptive\\Clustering\\Time (s)}
& \shortstack{Total MP\\Solve Time\\(s)}
& \shortstack{Total Interior\\Problem Solve\\Time (s)} \\
\midrule

\multirow{13}{*}{11}
& \multirow{4}{*}{\shortstack[l]{\textit{adapt-G-S}}}
& 16 & 39 & 457.5 & 433.3 & 10.0 & 7.0  & 2.9 \\
& & 24 & 46 & 596.4 & 560.0 & 11.1 & 13.4 & 5.3 \\
& & 32 & 46 & 578.2 & 537.3 & 13.6 & 11.8 & 6.8 \\
& & 40 & 41 & 537.9 & 492.2 & 12.2 & 15.2 & 8.7 \\

& \multirow{4}{*}{\shortstack[l]{\textit{adapt-G-I}}}
& 16 & 40 & 513.6 & 491.7 & 10.0 & 5.5  & 2.6 \\
& & 24 & 47 & 578.3 & 535.9 & 12.0 & 22.5 & 4.2 \\
& & 32 & 39 & 513.7 & 481.8 & 11.1 & 10.1 & 5.9 \\
\bottomrule
\end{tabular}
\end{table}
\end{landscape}


\begin{landscape}
\begin{table}[p]
\centering
\caption[Average SP and MP solve times (no-policy scenario with 48-hour SPs)]{Average SP and MP solve times for the no CO$_2$ policy case with 48-hour SPs.}
\label{tab:si_no_co2_48h_avg_times}
\scriptsize
\setlength{\tabcolsep}{4pt}
\begin{tabular}{c l c c c c c}
\toprule
Zones
& Algorithm
& Groups
& \shortstack{Total\\Iterations}
& \shortstack{Total\\Runtime\\(s)}
& \shortstack{Avg. SP\\Solve Time\\(s)}
& \shortstack{Avg. MP\\Solve Time\\(s)} \\
\midrule

\multirow{13}{*}{11}
& \multirow{4}{*}{\shortstack[l]{\textit{adapt-G-S}}}
& 16  & 39 & 198.0 & 4.1 & 0.2 \\
& & 32  & 33 & 191.7 & 4.3 & 0.4 \\
& & 64  & 35 & 217.4 & 4.4 & 0.5 \\
& & 128 & 34 & 270.9 & 4.8 & 1.0 \\

& \multirow{4}{*}{\shortstack[l]{\textit{adapt-G-I}}}
& 16  & 39 & 207.1 & 4.3 & 0.3 \\
& & 32  & 36 & 202.7 & 4.4 & 0.3 \\
& & 64  & 37 & 226.7 & 4.5 & 0.5 \\
\bottomrule
\end{tabular}
\end{table}
\end{landscape}


\begin{landscape}
\begin{table}[p]
\centering
\caption[Average SP and MP solve times (no-policy scenario with 168-hour SPs)]{Average SP and MP solve times for the no CO$_2$ policy case with 168-hour SPs.}
\label{tab:si_no_co2_168h_avg_times}
\scriptsize
\setlength{\tabcolsep}{4pt}
\begin{tabular}{c l c c c c c}
\toprule
Zones
& Algorithm
& Groups
& \shortstack{Total\\Iterations}
& \shortstack{Total\\Runtime\\(s)}
& \shortstack{Avg. SP\\Solve Time\\(s)}
& \shortstack{Avg. MP\\Solve Time\\(s)} \\
\midrule

\multirow{13}{*}{11}
& \multirow{4}{*}{\shortstack[l]{\textit{adapt-G-S}}}
& 16 & 39 & 457.5 & 11.1 & 0.2 \\
& & 24 & 46 & 596.4 & 12.2 & 0.3 \\
& & 32 & 46 & 578.2 & 11.7 & 0.3 \\
& & 40 & 41 & 537.9 & 12.0 & 0.4 \\

& \multirow{4}{*}{\shortstack[l]{\textit{adapt-G-I}}}
& 16 & 40 & 513.6 & 12.3 & 0.1 \\
& & 24 & 47 & 578.3 & 11.4 & 0.5 \\
& & 32 & 39 & 513.7 & 12.4 & 0.3 \\
\bottomrule
\end{tabular}
\end{table}
\end{landscape}


\begin{landscape}
\begin{table}[p]
\centering
\caption[Runtime breakdown (no-policy scenario with 48-hour SPs)]{Breakdown of total runtime for the no CO$_2$ policy case with 48-hour subproblems.}
\label{tab:si_no_co2_48h_breakdown}
\scriptsize
\setlength{\tabcolsep}{4pt}
\begin{tabular}{c l c c c c c c c}
\toprule
Zones
& Algorithm
& Groups
& \shortstack{Total\\Runtime\\(s)}
& \shortstack{SP\\(\% Runtime)}
& \shortstack{Adaptive\\Clustering\\(\% Runtime)}
& \shortstack{MP\\(\% Runtime)}
& \shortstack{Interior\\Problem\\(\% Runtime)}
& \shortstack{Others\\(\% Runtime)} \\
\midrule

\multirow{13}{*}{11}
& \multirow{4}{*}{\shortstack[l]{\textit{adapt-G-S}}}
& 16  & 198.0 & 80.0\% & 10.8\% & 4.2\%  & 1.2\% & 3.8\% \\
& & 32  & 191.7 & 74.6\% & 12.1\% & 7.7\%  & 1.8\% & 3.9\% \\
& & 64  & 217.4 & 71.0\% & 11.6\% & 8.3\%  & 4.0\% & 5.0\% \\
& & 128 & 270.9 & 60.5\% & 14.4\% & 13.0\% & 6.8\% & 5.3\% \\

& \multirow{4}{*}{\shortstack[l]{\textit{adapt-G-I}}}
& 16  & 207.1 & 80.4\% & 10.0\% & 5.7\%  & 1.2\% & 2.7\% \\
& & 32  & 202.7 & 78.6\% & 11.2\% & 5.2\%  & 2.1\% & 2.9\% \\
& & 64  & 226.7 & 73.6\% & 9.9\%  & 8.2\%  & 4.3\% & 4.0\% \\
\bottomrule
\end{tabular}
\end{table}
\end{landscape}


\begin{landscape}
\begin{table}[p]
\centering
\caption[Runtime breakdown (no-policy scenario with 168-hour SPs)]{Breakdown of total runtime for the no CO$_2$ policy case with 168-hour SPs.}
\label{tab:si_no_co2_168h_breakdown}
\scriptsize
\setlength{\tabcolsep}{4pt}
\begin{tabular}{c l c c c c c c c}
\toprule
Zones
& Algorithm
& Groups
& \shortstack{Total\\Runtime\\(s)}
& \shortstack{SP\\(\% Runtime)}
& \shortstack{Adaptive\\Clustering\\(\% Runtime)}
& \shortstack{MP\\(\% Runtime)}
& \shortstack{Interior\\Problem\\(\% Runtime)}
& \shortstack{Others\\(\% Runtime)} \\
\midrule

\multirow{13}{*}{11}
& \multirow{4}{*}{\shortstack[l]{\textit{adapt-G-S}}}
& 16  & 457.5 & 94.7\% & 2.2\% & 1.5\% & 0.6\% & 0.9\% \\
& & 24  & 596.4 & 93.9\% & 1.9\% & 2.2\% & 0.9\% & 1.1\% \\
& & 32  & 578.2 & 92.9\% & 2.4\% & 2.0\% & 1.2\% & 1.5\% \\
& & 40  & 537.9 & 91.5\% & 2.3\% & 2.8\% & 1.6\% & 1.8\% \\

& \multirow{4}{*}{\shortstack[l]{\textit{adapt-G-I}}}
& 16  & 513.6 & 95.7\% & 2.0\% & 1.1\% & 0.5\% & 0.7\% \\
& & 24  & 578.3 & 92.7\% & 2.1\% & 3.9\% & 0.7\% & 0.6\% \\
& & 32  & 513.7 & 93.8\% & 2.2\% & 2.0\% & 1.2\% & 0.9\% \\
\bottomrule
\end{tabular}
\end{table}
\end{landscape}



\begin{landscape}
\begin{table}[p]
\centering
\caption[Detailed computational results (CO$_2$-price scenario with 48-hour SPs)]{Detailed computational results for the CO$_2$ price case with 48-hour SPs.}
\label{tab:si_co2_price_48h_detail}
\scriptsize
\setlength{\tabcolsep}{4pt}
\begin{tabular}{c l c c c c c c c}
\toprule
Zones
& Algorithm
& Groups
& \shortstack{Total\\Iterations}
& \shortstack{Total\\Runtime\\(s)}
& \shortstack{Total SP\\Solve Time\\(s)}
& \shortstack{Total Adaptive\\Clustering\\Time (s)}
& \shortstack{Total MP\\Solve Time\\(s)}
& \shortstack{Total Interior\\Problem Solve\\Time (s)} \\
\midrule

\multirow{9}{*}{11}
& \multirow{4}{*}{\shortstack[l]{\textit{adapt-G-S}}}
& 16  & 35 & 200.3 & 167.0 & 18.5 & 4.3  & 2.2 \\
& & 32  & 28 & 166.1 & 130.7 & 18.5 & 8.2  & 2.5 \\
& & 64  & 25 & 153.2 & 116.1 & 18.6 & 6.5  & 3.4 \\
& & 128 & 27 & 193.9 & 129.8 & 22.8 & 19.9 & 10.3 \\

& \multirow{4}{*}{\shortstack[l]{\textit{adapt-G-I}}}
& 16  & 33 & 217.1 & 187.6 & 19.3 & 2.8  & 1.8 \\
& & 32  & 31 & 187.0 & 148.1 & 20.9 & 8.4  & 2.9 \\
& & 64  & 30 & 187.3 & 138.9 & 20.9 & 13.6 & 5.2 \\
\bottomrule
\end{tabular}
\end{table}
\end{landscape}


\begin{landscape}
\begin{table}[p]
\centering
\caption[Detailed computational results (CO$_2$-price scenario with 168-hour SPs)]{Detailed computational results for the CO$_2$ price case with 168-hour SPs.}
\label{tab:si_co2_price_168h_detail}
\scriptsize
\setlength{\tabcolsep}{4pt}
\begin{tabular}{c l c c c c c c c}
\toprule
Zones
& Algorithm
& Groups
& \shortstack{Total\\Iterations}
& \shortstack{Total\\Runtime\\(s)}
& \shortstack{Total SP\\Solve Time\\(s)}
& \shortstack{Total Adaptive\\Clustering\\Time (s)}
& \shortstack{Total MP\\Solve Time\\(s)}
& \shortstack{Total Interior\\Problem Solve\\Time (s)} \\
\midrule

\multirow{9}{*}{11}
& \multirow{4}{*}{\shortstack[l]{\textit{adapt-G-S}}}
& 16 & 37 & 369.7 & 348.8 & 10.0 & 4.4  & 2.3 \\
& & 24 & 37 & 375.9 & 343.3 & 11.9 & 11.7 & 2.9 \\
& & 32 & 34 & 352.6 & 322.3 & 11.6 & 9.6  & 3.3 \\
& & 40 & 36 & 448.1 & 409.0 & 12.4 & 13.7 & 5.0 \\

& \multirow{4}{*}{\shortstack[l]{\textit{adapt-G-I}}}
& 16 & 41 & 463.6 & 427.5 & 16.5 & 12.0 & 2.8 \\
& & 24 & 38 & 400.6 & 361.3 & 12.7 & 16.9 & 5.0 \\
& & 32 & 35 & 498.2 & 465.3 & 12.0 & 13.3 & 3.6 \\
\bottomrule
\end{tabular}
\end{table}
\end{landscape}


\begin{landscape}
\begin{table}[p]
\centering
\caption[Average SP and MP solve times (CO$_2$-price scenario with 48-hour SPs)]{Average SP and MP solve times for the CO$_2$ price case with 48-hour SPs.}
\label{tab:si_co2_price_48h_avg_times}
\scriptsize
\setlength{\tabcolsep}{4pt}
\begin{tabular}{c l c c c c c}
\toprule
Zones
& Algorithm
& Groups
& \shortstack{Total\\Iterations}
& \shortstack{Total\\Runtime\\(s)}
& \shortstack{Avg. SP\\Solve Time\\(s)}
& \shortstack{Avg. MP\\Solve Time\\(s)} \\
\midrule

\multirow{9}{*}{11}
& \multirow{4}{*}{\shortstack[l]{\textit{adapt-G-S}}}
& 16  & 35 & 200.3 & 4.8 & 0.1 \\
& & 32  & 28 & 166.1 & 4.7 & 0.3 \\
& & 64  & 25 & 153.2 & 4.6 & 0.3 \\
& & 128 & 27 & 193.9 & 4.8 & 0.7 \\

& \multirow{4}{*}{\shortstack[l]{\textit{adapt-G-I}}}
& 16  & 33 & 217.1 & 5.7 & 0.1 \\
& & 32  & 31 & 187.0 & 4.8 & 0.3 \\
& & 64  & 30 & 187.3 & 4.6 & 0.5 \\
& & 128 & 27 & 190.0 & 5.0 & 0.7 \\

& \textit{multi}
& -- & 24 & 173.8 & 5.2 & 1.1 \\
\midrule

\multirow{9}{*}{20}
& \multirow{4}{*}{\shortstack[l]{\textit{adapt-G-S}}}
& 16  & 156 & 3242.1 & 12.3 & 6.7 \\
& & 32  & 80  & 1441.7 & 11.4 & 5.2 \\
& & 64  & 45  & 793.0  & 12.7 & 2.8 \\
& & 128 & 38  & 730.9  & 13.2 & 3.2 \\

& \multirow{4}{*}{\shortstack[l]{\textit{adapt-G-I}}}
& 16  & 146 & 2329.1 & 9.9  & 5.0 \\
& & 32  & 62  & 914.4  & 10.6 & 2.8 \\
& & 64  & 47  & 830.7  & 13.5 & 2.6 \\
& & 128 & 45  & 956.4  & 14.0 & 5.0 \\

& \textit{multi}
& -- & 36 & 713.7 & 13.1 & 4.0 \\
\midrule

\multirow{9}{*}{26}
& \multirow{4}{*}{\shortstack[l]{\textit{adapt-G-S}}}
& 16  & \multicolumn{4}{l}{Intractable (Total Runtime $>$ 12h)} \\
& & 32  & 276 & 18322.2 & 18.9 & 43.3 \\
& & 64  & 103 & 4405.5  & 20.0 & 19.6 \\
& & 128 & 60  & 2741.8  & 21.1 & 19.8 \\

& \multirow{4}{*}{\shortstack[l]{\textit{adapt-G-I}}}
& 16  & \multicolumn{4}{l}{Intractable (Total Runtime $>$ 12h)} \\
& & 32  & 186 & 10494.6 & 20.5 & 34.3 \\
& & 64  & 79  & 2815.3  & 21.9 & 11.8 \\
& & 128 & 65  & 3290.9  & 25.9 & 21.5 \\

& \textit{multi}
& -- & 57 & 2247.4 & 21.6 & 14.7 \\
\bottomrule
\end{tabular}
\end{table}
\end{landscape}


\begin{landscape}
\begin{table}[p]
\centering
\caption[Average SP and MP solve times (CO$_2$-price scenario with 168-hour SPs)]{Average SP and MP solve times for the CO$_2$ price case with 168-hour SPs.}
\label{tab:si_co2_price_168h_avg_times}
\scriptsize
\setlength{\tabcolsep}{4pt}
\begin{tabular}{c l c c c c c}
\toprule
Zones
& Algorithm
& Groups
& \shortstack{Total\\Iterations}
& \shortstack{Total\\Runtime\\(s)}
& \shortstack{Avg. SP\\Solve Time\\(s)}
& \shortstack{Avg. MP\\Solve Time\\(s)} \\
\midrule

\multirow{9}{*}{11}
& \multirow{4}{*}{\shortstack[l]{\textit{adapt-G-S}}}
& 16 & 37 & 369.7 & 9.4 & 0.1 \\
& & 24 & 37 & 375.9 & 9.3 & 0.3 \\
& & 32 & 34 & 352.6 & 9.5 & 0.3 \\
& & 40 & 36 & 448.1 & 11.4 & 0.4 \\

& \multirow{4}{*}{\shortstack[l]{\textit{adapt-G-I}}}
& 16 & 41 & 463.6 & 10.4 & 0.3 \\
& & 24 & 38 & 400.6 & 9.5  & 0.4 \\
& & 32 & 35 & 498.2 & 13.3 & 0.4 \\
& & 40 & 32 & 395.3 & 11.5 & 0.2 \\

& \textit{multi}
& -- & 33 & 328.8 & 9.2 & 0.4 \\
\midrule

\multirow{9}{*}{20}
& \multirow{4}{*}{\shortstack[l]{\textit{adapt-G-S}}}
& 16 & 150 & 4441.9 & 22.8 & 6.3 \\
& & 24 & 124 & 3996.2 & 24.7 & 6.9 \\
& & 32 & 71  & 2035.3 & 24.7 & 3.2 \\
& & 40 & 82  & 2540.7 & 24.8 & 5.4 \\

& \multirow{4}{*}{\shortstack[l]{\textit{adapt-G-I}}}
& 16 & 86  & 2176.7 & 23.0 & 1.9 \\
& & 24 & 128 & 4005.9 & 23.9 & 7.0 \\
& & 32 & 99  & 2923.9 & 22.8 & 6.3 \\
& & 40 & 80  & 2418.9 & 24.7 & 5.0 \\

& \textit{multi}
& -- & 55 & 1500.2 & 24.1 & 2.6 \\
\midrule

\multirow{9}{*}{26}
& \multirow{4}{*}{\shortstack[l]{\textit{adapt-G-S}}}
& 16  & 428 & 26310.7 & 29.0 & 30.6 \\
& & 32  & 376 & 35223.8 & 49.1 & 42.2 \\
& & 64  & 341 & 25458.6 & 28.0 & 44.5 \\
& & 128 & 217 & 13088.5 & 30.9 & 27.7 \\

& \multirow{4}{*}{\shortstack[l]{\textit{adapt-G-I}}}
& 16  & 346 & 18288.0 & 31.0 & 21.4 \\
& & 32  & 216 & 15001.7 & 47.2 & 21.7 \\
& & 64  & 212 & 11020.3 & 29.6 & 21.9 \\
& & 128 & 185 & 13122.1 & 44.4 & 25.9 \\

& \textit{multi}
& -- & 116 & 6807.0 & 43.8 & 14.2 \\
\bottomrule
\end{tabular}
\end{table}
\end{landscape}


\begin{landscape}
\begin{table}[p]
\centering
\caption[Runtime breakdown (CO$_2$-price scenario with 48-hour SPs)]{Breakdown of total runtime for the CO$_2$ price case with 48-hour SPs.}
\label{tab:si_co2_price_48h_breakdown}
\scriptsize
\setlength{\tabcolsep}{4pt}
\begin{tabular}{c l c c c c c c c}
\toprule
Zones
& Algorithm
& Groups
& \shortstack{Total\\Runtime\\(s)}
& \shortstack{SP\\(\% Runtime)}
& \shortstack{Adaptive Clustering\\(\% Runtime)}
& \shortstack{MP\\(\% Runtime)}
& \shortstack{Interior Problem\\(\% Runtime)}
& \shortstack{Others\\(\% Runtime)} \\
\midrule

\multirow{9}{*}{11}
& \multirow{4}{*}{\shortstack[l]{\textit{adapt-G-S}}}
& 16  & 200.3 & 83.4\% & 9.2\%  & 2.2\%  & 1.1\% & 4.1\% \\
& & 32  & 166.1 & 78.7\% & 11.1\% & 4.9\%  & 1.5\% & 3.8\% \\
& & 64  & 153.2 & 75.8\% & 12.1\% & 4.2\%  & 2.2\% & 5.6\% \\
& & 128 & 193.9 & 67.0\% & 11.8\% & 10.2\% & 5.3\% & 5.7\% \\

& \multirow{4}{*}{\shortstack[l]{\textit{adapt-G-I}}}
& 16  & 217.1 & 86.4\% & 8.9\%  & 1.3\%  & 0.8\% & 2.5\% \\
& & 32  & 187.0 & 79.2\% & 11.2\% & 4.5\%  & 1.5\% & 3.6\% \\
& & 64  & 187.3 & 74.2\% & 11.1\% & 7.2\%  & 2.8\% & 4.7\% \\
\bottomrule
\end{tabular}
\end{table}
\end{landscape}


\begin{landscape}
\begin{table}[p]
\centering
\caption[Runtime breakdown (CO$_2$-price scenario with 168-hour SPs)]{Breakdown of total runtime for the CO$_2$ price case with 168-hour SPs.}
\label{tab:si_co2_price_168h_breakdown}
\scriptsize
\setlength{\tabcolsep}{4pt}
\begin{tabular}{c l c c c c c c c}
\toprule
Zones
& Algorithm
& Groups
& \shortstack{Total\\Runtime\\(s)}
& \shortstack{SP\\(\% Runtime)}
& \shortstack{Adaptive Clustering\\(\% Runtime)}
& \shortstack{MP\\(\% Runtime)}
& \shortstack{Interior Problem\\(\% Runtime)}
& \shortstack{Others\\(\% Runtime)} \\
\midrule

\multirow{9}{*}{11}
& \multirow{4}{*}{\shortstack[l]{\textit{adapt-G-S}}}
& 16 & 369.7 & 94.4\% & 2.7\% & 1.2\% & 0.6\% & 1.1\% \\
& & 24 & 375.9 & 91.3\% & 3.2\% & 3.1\% & 0.8\% & 1.6\% \\
& & 32 & 352.6 & 91.4\% & 3.3\% & 2.7\% & 0.9\% & 1.6\% \\
& & 40 & 448.1 & 91.3\% & 2.8\% & 3.1\% & 1.1\% & 1.8\% \\

& \multirow{4}{*}{\shortstack[l]{\textit{adapt-G-I}}}
& 16 & 463.6 & 92.2\% & 3.6\% & 2.6\% & 0.6\% & 1.0\% \\
& & 24 & 400.6 & 90.2\% & 3.2\% & 4.2\% & 1.2\% & 1.2\% \\
& & 32 & 498.2 & 93.4\% & 2.4\% & 2.7\% & 0.7\% & 0.8\% \\
\bottomrule
\end{tabular}
\end{table}
\end{landscape}


\begin{landscape}
\begin{table}[p]
\centering
\caption[Detailed computational results (CO$_2$-cap scenario with 48-hour SPs)]{Detailed computational results for the CO$_2$ cap case with 48-hour SPs.}
\label{tab:si_co2_cap_48h_detail}
\scriptsize
\setlength{\tabcolsep}{4pt}
\begin{tabular}{c l c c c c c c c}
\toprule
Zones
& Algorithm
& Groups
& \shortstack{Total\\Iterations}
& \shortstack{Total\\Runtime\\(s)}
& \shortstack{Total SP\\Solve Time\\(s)}
& \shortstack{Total Adaptive\\Clustering\\Time (s)}
& \shortstack{Total MP\\Solve Time\\(s)}
& \shortstack{Total Interior\\Problem Solve\\Time (s)} \\
\midrule

\multirow{9}{*}{11}
& \multirow{4}{*}{\shortstack[l]{\textit{adapt-G-S}}}
& 16  & 308 & 3120.3 & 1763.3 & 106.3 & 957.4  & 52.1 \\
& & 32  & 212 & 2634.8 & 1137.7 & 130.1 & 1185.8 & 41.3 \\
& & 64  & 75  & 625.0  & 395.9  & 45.7  & 137.5  & 12.7 \\
& & 128 & 37  & 296.6  & 219.2  & 30.6  & 17.5   & 6.4 \\

& \multirow{4}{*}{\shortstack[l]{\textit{adapt-G-I}}}
& 16  & 167 & 1290.7 & 896.0 & 69.9 & 283.6 & 27.7 \\
& & 32  & 93  & 706.6  & 486.5 & 35.2 & 159.2 & 15.1 \\
& & 64  & 49  & 339.6  & 273.0 & 25.2 & 26.1  & 8.5 \\
\bottomrule
\end{tabular}
\end{table}
\end{landscape}


\begin{landscape}
\begin{table}[p]
\centering
\caption[Detailed computational results (CO$_2$-cap scenario with 168-hour SPs)]{Detailed computational results for the CO$_2$ cap case with 168-hour SPs.}
\label{tab:si_co2_cap_168h_detail}
\scriptsize
\setlength{\tabcolsep}{4pt}
\begin{tabular}{c l c c c c c c c}
\toprule
Zones
& Algorithm
& Groups
& \shortstack{Total\\Iterations}
& \shortstack{Total\\Runtime\\(s)}
& \shortstack{Total SP\\Solve Time\\(s)}
& \shortstack{Total Adaptive\\Clustering\\Time (s)}
& \shortstack{Total MP\\Solve Time\\(s)}
& \shortstack{Total Interior\\Problem Solve\\Time (s)} \\
\midrule

\multirow{9}{*}{11}
& \multirow{4}{*}{\shortstack[l]{\textit{adapt-G-S}}}
& 16 & 281 & 5185.1 & 3960.5 & 17.0 & 1083.6 & 34.8 \\
& & 24 & 64  & 1081.0 & 1001.8 & 13.2 & 50.8   & 6.0 \\
& & 32 & 40  & 706.6  & 675.2  & 10.8 & 10.9   & 4.0 \\
& & 40 & 41  & 671.4  & 629.4  & 11.9 & 16.2   & 5.7 \\

& \multirow{4}{*}{\shortstack[l]{\textit{adapt-G-I}}}
& 16 & 253 & 4708.5 & 3427.2 & 14.3 & 1181.3 & 77.0 \\
& & 24 & 53  & 920.8  & 868.9  & 12.2 & 29.1   & 5.5 \\
& & 32 & 46  & 992.3  & 944.2  & 14.8 & 21.2   & 6.4 \\
\bottomrule
\end{tabular}
\end{table}
\end{landscape}


\begin{landscape}
\begin{table}[p]
\centering
\caption[Average SP and MP solve times (CO$_2$-cap scenario with 48-hour SPs)]{Average SP and MP solve times for the CO$_2$ cap case with 48-hour SPs.}
\label{tab:si_co2_cap_48h_avg_times}
\scriptsize
\setlength{\tabcolsep}{4pt}
\begin{tabular}{c l c c c c c}
\toprule
Zones
& Algorithm
& Groups
& \shortstack{Total\\Iterations}
& \shortstack{Total\\Runtime\\(s)}
& \shortstack{Avg. SP\\Solve Time\\(s)}
& \shortstack{Avg. MP\\Solve Time\\(s)} \\
\midrule

\multirow{9}{*}{11}
& \multirow{4}{*}{\shortstack[l]{\textit{adapt-G-S}}}
& 16  & 308 & 3120.3 & 5.7 & 3.1 \\
& & 32  & 212 & 2634.8 & 5.4 & 5.6 \\
& & 64  & 75  & 625.0  & 5.3 & 1.8 \\
& & 128 & 37  & 296.6  & 5.9 & 0.5 \\

& \multirow{4}{*}{\shortstack[l]{\textit{adapt-G-I}}}
& 16  & 167 & 1290.7 & 5.4 & 1.7 \\
& & 32  & 93  & 706.6  & 5.2 & 1.7 \\
& & 64  & 49  & 339.6  & 5.6 & 0.5 \\
\bottomrule
\end{tabular}
\end{table}
\end{landscape}


\begin{landscape}
\begin{table}[p]
\centering
\caption[Average SP and MP solve times (CO$_2$-cap scenario with 168-hour SPs)]{Average SP and MP solve times for the CO$_2$ cap case with 168-hour SPs.}
\label{tab:si_co2_cap_168h_avg_times}
\scriptsize
\setlength{\tabcolsep}{4pt}
\begin{tabular}{c l c c c c c}
\toprule
Zones
& Algorithm
& Groups
& \shortstack{Total\\Iterations}
& \shortstack{Total\\Runtime\\(s)}
& \shortstack{Avg. SP\\Solve Time\\(s)}
& \shortstack{Avg. MP\\Solve Time\\(s)} \\
\midrule

\multirow{9}{*}{11}
& \multirow{4}{*}{\shortstack[l]{\textit{adapt-G-S}}}
& 16 & 281 & 5185.1 & 14.1 & 3.9 \\
& & 24 & 64  & 1081.0 & 15.7 & 0.8 \\
& & 32 & 40  & 706.6  & 16.9 & 0.3 \\
& & 40 & 41  & 671.4  & 15.4 & 0.4 \\

& \multirow{4}{*}{\shortstack[l]{\textit{adapt-G-I}}}
& 16 & 253 & 4708.5 & 13.5 & 4.7 \\
& & 24 & 53  & 920.8  & 16.4 & 0.5 \\
& & 32 & 46  & 992.3  & 20.5 & 0.5 \\
\bottomrule
\end{tabular}
\end{table}
\end{landscape}


\begin{landscape}
\begin{table}[p]
\centering
\caption[Runtime breakdown (CO$_2$-cap scenario with 48-hour SPs)]{Breakdown of total runtime for the CO$_2$ cap case with 48-hour SPs.}
\label{tab:si_co2_cap_48h_breakdown}
\scriptsize
\setlength{\tabcolsep}{4pt}
\begin{tabular}{c l c c c c c c c}
\toprule
Zones
& Algorithm
& Groups
& \shortstack{Total\\Runtime\\(s)}
& \shortstack{SP\\(\% Runtime)}
& \shortstack{Adaptive Clustering\\(\% Runtime)}
& \shortstack{MP\\(\% Runtime)}
& \shortstack{Interior Problem\\(\% Runtime)}
& \shortstack{Others\\(\% Runtime)} \\
\midrule

\multirow{9}{*}{11}
& \multirow{4}{*}{\shortstack[l]{\textit{adapt-G-S}}}
& 16  & 3120.3 & 56.5\% & 3.4\%  & 30.7\% & 1.7\% & 7.7\% \\
& & 32  & 2634.8 & 43.2\% & 4.9\%  & 45.0\% & 1.6\% & 5.3\% \\
& & 64  & 625.0  & 63.3\% & 7.3\%  & 22.0\% & 2.0\% & 5.3\% \\
& & 128 & 296.6  & 73.9\% & 10.3\% & 5.9\%  & 2.2\% & 7.7\% \\

& \multirow{4}{*}{\shortstack[l]{\textit{adapt-G-I}}}
& 16  & 1290.7 & 69.4\% & 5.4\%  & 22.0\% & 2.1\% & 1.0\% \\
& & 32  & 706.6  & 68.8\% & 5.0\%  & 22.5\% & 2.1\% & 1.5\% \\
& & 64  & 339.6  & 80.4\% & 7.4\%  & 7.7\%  & 2.5\% & 2.0\% \\
\bottomrule
\end{tabular}
\end{table}
\end{landscape}


\begin{landscape}
\begin{table}[p]
\centering
\caption[Runtime breakdown (CO$_2$-cap scenario with 168-hour SPs)]{Breakdown of total runtime for the CO$_2$ cap case with 168-hour SPs.}
\label{tab:si_co2_cap_168h_breakdown}
\scriptsize
\setlength{\tabcolsep}{4pt}
\begin{tabular}{c l c c c c c c c}
\toprule
Zones
& Algorithm
& Groups
& \shortstack{Total\\Runtime\\(s)}
& \shortstack{SP\\(\% Runtime)}
& \shortstack{Adaptive Clustering\\(\% Runtime)}
& \shortstack{MP\\(\% Runtime)}
& \shortstack{Interior Problem\\(\% Runtime)}
& \shortstack{Others\\(\% Runtime)} \\
\midrule

\multirow{9}{*}{11}
& \multirow{4}{*}{\shortstack[l]{\textit{adapt-G-S}}}
& 16 & 5185.1 & 76.4\% & 0.3\% & 20.9\% & 0.7\% & 1.7\% \\
& & 24 & 1081.0 & 92.7\% & 1.2\% & 4.7\%  & 0.6\% & 0.8\% \\
& & 32 & 706.6  & 95.6\% & 1.5\% & 1.5\%  & 0.6\% & 0.8\% \\
& & 40 & 671.4  & 93.7\% & 1.8\% & 2.4\%  & 0.9\% & 1.2\% \\

& \multirow{4}{*}{\shortstack[l]{\textit{adapt-G-I}}}
& 16 & 4708.5 & 72.8\% & 0.3\% & 25.1\% & 1.6\% & 0.2\% \\
& & 24 & 920.8  & 94.4\% & 1.3\% & 3.2\%  & 0.6\% & 0.6\% \\
& & 32 & 992.3  & 95.1\% & 1.5\% & 2.1\%  & 0.6\% & 0.6\% \\
\bottomrule
\end{tabular}
\end{table}
\end{landscape}

\begin{landscape}
\begin{table}[p]
\centering
\caption[Detailed computational results across SPs-to-CPU ratios for the 20-zone no-policy case with 48-hour SPs]{Detailed computational results for the 20-zone no CO$_2$ policy case with 48-hour SPs under different SPs-to-CPU ratios.}
\label{tab:si_no_co2_48h_20z_detail}
\scriptsize
\setlength{\tabcolsep}{4pt}
\begin{tabular}{c l c c c c c c c}
\toprule
\shortstack{SPs-to-CPU\\ratio}
& Algorithm
& Groups
& \shortstack{Total\\Iterations}
& \shortstack{Total\\Runtime\\(s)}
& \shortstack{Total SP\\Solve Time\\(s)}
& \shortstack{Total Adaptive\\Clustering Time\\(s)}
& \shortstack{Total MP\\Solve Time\\(s)}
& \shortstack{Total Interior\\Problem Solve Time\\(s)} \\
\midrule

\multirow{9}{*}{3}
& \multirow{4}{*}{\shortstack[l]{\textit{adapt-G-I}}}
& 16  & 64 & 942.1  & 782.8 & 53.6 & 76.3  & 8.9  \\
& & 32  & 47 & 778.3  & 640.9 & 52.3 & 61.7  & 9.9  \\
& & 64  & 43 & 752.8  & 585.2 & 49.1 & 80.3  & 19.5 \\
& & 128 & 51 & 1162.2 & 686.6 & 59.7 & 329.0 & 64.6 \\

& \multirow{4}{*}{\shortstack[l]{\textit{rep-SP}}}
& 16  & 56 & 645.6 & 433.8 & 37.3 & 116.5 & 39.8 \\
& & 32  & 49 & 634.6 & 442.5 & 34.9 & 113.8 & 31.9 \\
& & 64  & 47 & 761.1 & 498.9 & 33.6 & 178.2 & 35.8 \\
\bottomrule
\end{tabular}
\end{table}
\end{landscape}

\begin{landscape}
\begin{table}[p]
\centering
\caption[Average SP and MP solve times across SPs-to-CPU ratios for the 20-zone no-policy case with 48-hour SPs]{Average SP and MP solve times for the 20-zone no CO$_2$ policy case with 48-hour SPs under different SPs-to-CPU ratios.}
\label{tab:si_no_co2_48h_20z_avg}
\scriptsize
\setlength{\tabcolsep}{4pt}
\begin{tabular}{c l c c c c c}
\toprule
\shortstack{SPs-to-CPU\\ratio}
& Algorithm
& Groups
& \shortstack{Total\\Iterations}
& \shortstack{Total\\Runtime\\(s)}
& \shortstack{Avg.\ SP\\Solve Time\\(s)}
& \shortstack{Avg.\ MP\\Solve Time\\(s)} \\
\midrule

\multirow{9}{*}{3}
& \multirow{4}{*}{\shortstack[l]{\textit{adapt-G-I}}}
& 16  & 64 & 942.1  & 12.2 & 1.2 \\
& & 32  & 47 & 778.3  & 13.6 & 1.3 \\
& & 64  & 43 & 752.8  & 13.6 & 1.9 \\
& & 128 & 51 & 1162.2 & 13.5 & 6.5 \\

& \multirow{4}{*}{\shortstack[l]{\textit{rep-SP}}}
& 16  & 56 & 645.6 & 7.7  & 2.1 \\
& & 32  & 49 & 634.6 & 9.0  & 2.3 \\
& & 64  & 47 & 761.1 & 10.6 & 3.8 \\
& & 128 & 48 & 930.5 & 11.8 & 4.4 \\

& \textit{multi}
& -- & 49 & 1321.6 & 14.1 & 9.9 \\
\midrule

\multirow{9}{*}{11}
& \multirow{4}{*}{\shortstack[l]{\textit{adapt-G-I}}}
& 16  & 64 & 2743.8 & 39.8 & 1.3 \\
& & 32  & 47 & 2080.5 & 41.1 & 1.2 \\
& & 64  & 43 & 1946.3 & 41.0 & 1.9 \\
& & 128 & 51 & 2658.8 & 41.0 & 7.7 \\

& \multirow{4}{*}{\shortstack[l]{\textit{rep-SP}}}
& 16  & 56 & 1608.8 & 23.6 & 2.7 \\
& & 32  & 48 & 1410.1 & 23.9 & 2.7 \\
& & 64  & 43 & 1441.6 & 27.8 & 3.0 \\
& & 128 & 47 & 2086.1 & 34.6 & 5.2 \\

& \textit{multi}
& -- & 49 & 2780.7 & 41.9 & 11.1 \\
\midrule

\multirow{9}{*}{23}
& \multirow{4}{*}{\shortstack[l]{\textit{adapt-G-I}}}
& 16  & 64 & 4845.5 & 68.7 & 3.9 \\
& & 32  & 47 & 4241.1 & 83.8 & 3.1 \\
& & 64  & 43 & 3420.0 & 71.3 & 4.0 \\
& & 128 & 51 & 5730.8 & 90.2 & 15.5 \\

& \multirow{4}{*}{\shortstack[l]{\textit{rep-SP}}}
& 16  & 54 & 2482.5 & 35.5 & 6.4 \\
& & 32  & 47 & 2536.8 & 43.6 & 4.8 \\
& & 64  & 42 & 2868.5 & 57.8 & 4.9 \\
& & 128 & 43 & 3408.2 & 62.3 & 9.8 \\

& \textit{multi}
& -- & 49 & 4795.2 & 74.0 & 18.2 \\
\bottomrule
\end{tabular}
\end{table}
\end{landscape}

\begin{landscape}
\begin{table}[p]
\centering
\caption[Runtime breakdown across SPs-to-CPU ratios for the 20-zone no-policy case with 48-hour SPs]{Breakdown of total runtime for the 20-zone no CO$_2$ policy case with 48-hour SPs under different SPs-to-CPU ratios.}
\label{tab:si_no_co2_48h_20z_breakdown}
\scriptsize
\setlength{\tabcolsep}{4pt}
\begin{tabular}{c l c c c c c c c c}
\toprule
\shortstack{SPs-to-CPU\\ratio}
& Algorithm
& Groups
& \shortstack{Total\\Iterations}
& \shortstack{Total\\Runtime\\(s)}
& \shortstack{SP\\(\% Runtime)}
& \shortstack{Adaptive Clustering\\(\% Runtime)}
& \shortstack{MP\\(\% Runtime)}
& \shortstack{Interior Problem\\(\% Runtime)}
& \shortstack{Others\\(\% Runtime)} \\
\midrule

\multirow{9}{*}{3}
& \multirow{4}{*}{\shortstack[l]{\textit{adapt-G-I}}}
& 16  & 64 & 942.1  & 83.1\% & 5.7\% & 8.1\%  & 0.9\% & 2.2\% \\
& & 32  & 47 & 778.3  & 82.3\% & 6.7\% & 7.9\%  & 1.3\% & 1.7\% \\
& & 64  & 43 & 752.8  & 77.7\% & 6.5\% & 10.7\% & 2.6\% & 2.5\% \\
& & 128 & 51 & 1162.2 & 59.1\% & 5.1\% & 28.3\% & 5.6\% & 1.9\% \\

& \multirow{4}{*}{\shortstack[l]{\textit{rep-SP}}}
& 16  & 56 & 645.6 & 67.2\% & 5.8\% & 18.1\% & 6.2\% & 2.8\% \\
& & 32  & 49 & 634.6 & 69.7\% & 5.5\% & 17.9\% & 5.0\% & 1.8\% \\
& & 64  & 47 & 761.1 & 65.5\% & 4.4\% & 23.4\% & 4.7\% & 1.9\% \\
\bottomrule
\end{tabular}
\end{table}
\end{landscape}

\begin{landscape}
\begin{table}[p]
\centering
\caption[Detailed computational results across SPs-to-CPU ratios for the 20-zone CO$_2$-price case with 48-hour SPs]{Detailed computational results for the 20-zone CO$_2$ price case with 48-hour SPs under different SPs-to-CPU ratios.}
\label{tab:si_co2_price_48h_20z_detail}
\scriptsize
\setlength{\tabcolsep}{4pt}
\begin{tabular}{c l c c c c c c c}
\toprule
\shortstack{SPs-to-CPU\\ratio}
& Algorithm
& Groups
& \shortstack{Total\\Iterations}
& \shortstack{Total\\Runtime\\(s)}
& \shortstack{Total SP\\Solve Time\\(s)}
& \shortstack{Total Adaptive\\Clustering Time\\(s)}
& \shortstack{Total MP\\Solve Time\\(s)}
& \shortstack{Total Interior\\Problem Solve Time\\(s)} \\
\midrule

\multirow{9}{*}{3}
& \multirow{4}{*}{\shortstack[l]{\textit{adapt-G-I}}}
& 16  & 146 & 2329.1 & 1440.4 & 97.0 & 730.2 & 30.8 \\
& & 32  & 62  & 914.4  & 655.7  & 62.1 & 170.8 & 10.3 \\
& & 64  & 47  & 830.7  & 635.9  & 30.2 & 120.3 & 17.1 \\
& & 128 & 45  & 956.4  & 628.0  & 34.2 & 224.3 & 38.2 \\

& \multirow{4}{*}{\shortstack[l]{\textit{rep-SP}}}
& 16  & 50 & 557.0 & 419.4 & 29.6 & 62.6  & 30.1 \\
& & 32  & 49 & 660.7 & 468.0 & 31.8 & 124.7 & 23.5 \\
& & 64  & 42 & 601.2 & 435.3 & 29.9 & 87.0  & 30.3 \\
\bottomrule
\end{tabular}
\end{table}
\end{landscape}

\begin{landscape}
\begin{table}[p]
\centering
\caption[Average SP and MP solve times across SPs-to-CPU ratios for the 20-zone CO$_2$-price case with 48-hour SPs]{Average SP and MP solve times for the 20-zone CO$_2$ price case with 48-hour SPs under different SPs-to-CPU ratios.}
\label{tab:si_co2_price_48h_20z_avg}
\scriptsize
\setlength{\tabcolsep}{4pt}
\begin{tabular}{c l c c c c c}
\toprule
\shortstack{SPs-to-CPU\\ratio}
& Algorithm
& Groups
& \shortstack{Total\\Iterations}
& \shortstack{Total\\Runtime\\(s)}
& \shortstack{Avg.\ SP\\Solve Time\\(s)}
& \shortstack{Avg.\ MP\\Solve Time\\(s)} \\
\midrule

\multirow{9}{*}{3}
& \multirow{4}{*}{\shortstack[l]{\textit{adapt-G-I}}}
& 16  & 146 & 2329.1 & 9.9  & 5.0 \\
& & 32  & 62  & 914.4  & 10.6 & 2.8 \\
& & 64  & 47  & 830.7  & 13.5 & 2.6 \\
& & 128 & 45  & 956.4  & 14.0 & 5.0 \\

& \multirow{4}{*}{\shortstack[l]{\textit{rep-SP}}}
& 16  & 50 & 557.0 & 8.4  & 1.3 \\
& & 32  & 49 & 660.7 & 9.6  & 2.5 \\
& & 64  & 42 & 601.2 & 10.4 & 2.1 \\
& & 128 & 38 & 625.3 & 11.8 & 2.2 \\

& \textit{multi}
& -- & 36 & 713.7 & 13.1 & 4.0 \\
\midrule

\multirow{9}{*}{11}
& \multirow{4}{*}{\shortstack[l]{\textit{adapt-G-I}}}
& 16  & 146 & 6144.7 & 33.4 & 6.9 \\
& & 32  & 62  & 2480.6 & 34.1 & 3.4 \\
& & 64  & 47  & 1904.6 & 33.9 & 3.0 \\
& & 128 & 45  & 1968.5 & 33.6 & 6.6 \\

& \multirow{4}{*}{\shortstack[l]{\textit{rep-SP}}}
& 16  & 50 & 1233.0 & 20.7 & 1.6 \\
& & 32  & 49 & 1531.1 & 25.0 & 3.7 \\
& & 64  & 38 & 1298.3 & 28.9 & 2.1 \\
& & 128 & 38 & 1577.3 & 34.2 & 3.5 \\

& \textit{multi}
& -- & 36 & 1604.9 & 37.6 & 4.6 \\
\midrule

\multirow{9}{*}{23}
& \multirow{4}{*}{\shortstack[l]{\textit{adapt-G-I}}}
& 16  & 146 & 10918.0 & 57.8 & 13.8 \\
& & 32  & 62  & 4522.1  & 62.0 & 7.2 \\
& & 64  & 47  & 4170.4  & 78.3 & 6.9 \\
& & 128 & 45  & 3502.8  & 62.1 & 10.3 \\

& \multirow{4}{*}{\shortstack[l]{\textit{rep-SP}}}
& 16  & 50 & 2374.3 & 40.7 & 3.0 \\
& & 32  & 47 & 2888.9 & 50.1 & 7.7 \\
& & 64  & 36 & 2497.1 & 61.1 & 4.6 \\
& & 128 & 41 & 3545.2 & 70.3 & 9.8 \\

& \textit{multi}
& -- & 36 & 3305.3 & 78.1 & 9.2 \\
\bottomrule
\end{tabular}
\end{table}
\end{landscape}

\begin{landscape}
\begin{table}[p]
\centering
\caption[Runtime breakdown across SPs-to-CPU ratios for the 20-zone CO$_2$-price case with 48-hour SPs]{Breakdown of total runtime for the 20-zone CO$_2$ price case with 48-hour SPs under different SPs-to-CPU ratios.}
\label{tab:si_co2_price_48h_20z_breakdown}
\scriptsize
\setlength{\tabcolsep}{4pt}
\begin{tabular}{c l c c c c c c c c}
\toprule
\shortstack{SPs-to-CPU\\ratio}
& Algorithm
& Groups
& \shortstack{Total\\Iterations}
& \shortstack{Total\\Runtime\\(s)}
& \shortstack{SP\\(\% Runtime)}
& \shortstack{Adaptive Clustering\\(\% Runtime)}
& \shortstack{MP\\(\% Runtime)}
& \shortstack{Interior Problem\\(\% Runtime)}
& \shortstack{Others\\(\% Runtime)} \\
\midrule

\multirow{9}{*}{3}
& \multirow{4}{*}{\shortstack[l]{\textit{adapt-G-I}}}
& 16  & 146 & 2329.1 & 61.8\% & 4.2\% & 31.4\% & 1.3\% & 1.3\% \\
& & 32  & 62  & 914.4  & 71.7\% & 6.8\% & 18.7\% & 1.1\% & 1.7\% \\
& & 64  & 47  & 830.7  & 76.5\% & 3.6\% & 14.5\% & 2.1\% & 3.3\% \\
& & 128 & 45  & 956.4  & 65.7\% & 3.6\% & 23.5\% & 4.0\% & 3.3\% \\

& \multirow{4}{*}{\shortstack[l]{\textit{rep-SP}}}
& 16  & 50 & 557.0 & 75.3\% & 5.3\% & 11.2\% & 5.4\% & 2.8\% \\
& & 32  & 49 & 660.7 & 70.8\% & 4.8\% & 18.9\% & 3.5\% & 1.9\% \\
& & 64  & 42 & 601.2 & 72.4\% & 5.0\% & 14.5\% & 5.0\% & 3.1\% \\
\bottomrule
\end{tabular}
\end{table}
\end{landscape}

\begin{landscape}
\begin{table}[p]
\centering
\caption[Detailed computational results across SPs-to-CPU ratios for the 20-zone CO$_2$-cap case with 48-hour SPs]{Detailed computational results for the 20-zone CO$_2$ cap case with 48-hour SPs under different SPs-to-CPU ratios.}
\label{tab:si_co2_cap_48h_20z_detail}
\scriptsize
\setlength{\tabcolsep}{4pt}
\begin{tabular}{c l c c c c c c c}
\toprule
\shortstack{SPs-to-CPU\\ratio}
& Algorithm
& Groups
& \shortstack{Total\\Iterations}
& \shortstack{Total\\Runtime\\(s)}
& \shortstack{Total SP\\Solve Time\\(s)}
& \shortstack{Total Adaptive\\Clustering Time\\(s)}
& \shortstack{Total MP\\Solve Time\\(s)}
& \shortstack{Total Interior\\Problem Solve Time\\(s)} \\
\midrule

\multirow{9}{*}{3}
& \multirow{4}{*}{\shortstack[l]{\textit{adapt-G-I}}}
& 16  & \multicolumn{6}{c}{Intractable (total runtime $>$ 12h)} \\
& & 32  & \multicolumn{6}{c}{Intractable (total runtime $>$ 12h)} \\
& & 64  & 93 & 2733.5 & 1646.3 & 83.3 & 911.3 & 68.7 \\
& & 128 & 65 & 2169.0 & 1174.0 & 43.6 & 792.3 & 100.8 \\

& \multirow{4}{*}{\shortstack[l]{\textit{rep-SP}}}
& 16  & 89 & 1376.5 & 922.8  & 43.2 & 292.3 & 92.4 \\
& & 32  & 93 & 2209.4 & 1125.8 & 56.1 & 934.7 & 72.7 \\
& & 64  & 74 & 1580.8 & 938.8  & 44.5 & 448.1 & 109.7 \\
\bottomrule
\end{tabular}
\end{table}
\end{landscape}

\begin{landscape}
\begin{table}[p]
\centering
\caption[Average SP and MP solve times across SPs-to-CPU ratios for the 20-zone CO$_2$-cap case with 48-hour SPs]{Average SP and MP solve times for the 20-zone CO$_2$ cap case with 48-hour SPs under different SPs-to-CPU ratios.}
\label{tab:si_co2_cap_48h_20z_avg}
\scriptsize
\setlength{\tabcolsep}{4pt}
\begin{tabular}{c l c c c c c}
\toprule
\shortstack{SPs-to-CPU\\ratio}
& Algorithm
& Groups
& \shortstack{Total\\Iterations}
& \shortstack{Total\\Runtime\\(s)}
& \shortstack{Avg.\ SP\\Solve Time\\(s)}
& \shortstack{Avg.\ MP\\Solve Time\\(s)} \\
\midrule

\multirow{9}{*}{3}
& \multirow{4}{*}{\shortstack[l]{\textit{adapt-G-I}}}
& 16  & \multicolumn{4}{c}{Intractable (total runtime $>$ 12h)} \\
& & 32  & \multicolumn{4}{c}{Intractable (total runtime $>$ 12h)} \\
& & 64  & 93 & 2733.5 & 17.7 & 9.8 \\
& & 128 & 65 & 2169.0 & 18.1 & 12.2 \\

& \multirow{4}{*}{\shortstack[l]{\textit{rep-SP}}}
& 16  & 89 & 1376.5 & 10.4 & 3.3 \\
& & 32  & 93 & 2209.4 & 12.1 & 10.1 \\
& & 64  & 74 & 1580.8 & 12.7 & 6.1 \\
& & 128 & 65 & 1360.3 & 13.4 & 4.3 \\

& \textit{multi}
& -- & 55 & 1526.0 & 15.4 & 8.9 \\
\midrule

\multirow{9}{*}{11}
& \multirow{4}{*}{\shortstack[l]{\textit{adapt-G-I}}}
& 16  & \multicolumn{4}{c}{Intractable (total runtime $>$ 12h)} \\
& & 32  & \multicolumn{4}{c}{Intractable (total runtime $>$ 12h)} \\
& & 64  & 93 & 6286.6 & 45.2 & 17.5 \\
& & 128 & 65 & 3916.9 & 45.0 & 11.8 \\

& \multirow{4}{*}{\shortstack[l]{\textit{rep-SP}}}
& 16  & 89 & 2989.7 & 27.0 & 4.1 \\
& & 32  & 81 & 3376.9 & 32.5 & 6.8 \\
& & 64  & 65 & 2804.2 & 35.6 & 4.3 \\
& & 128 & 68 & 3426.1 & 40.0 & 6.8 \\

& \textit{multi}
& -- & 55 & 3221.6 & 44.2 & 10.5 \\
\midrule

\multirow{9}{*}{23}
& \multirow{4}{*}{\shortstack[l]{\textit{adapt-G-I}}}
& 16  & \multicolumn{4}{c}{Intractable (total runtime $>$ 12h)} \\
& & 32  & \multicolumn{4}{c}{Intractable (total runtime $>$ 12h)} \\
& & 64  & 93 & 13089.6 & 106.8 & 29.2 \\
& & 128 & 65 & 7474.1  & 86.0  & 21.4 \\

& \multirow{4}{*}{\shortstack[l]{\textit{rep-SP}}}
& 16  & 83 & 4328.3 & 42.5 & 5.1 \\
& & 32  & 76 & 4730.0 & 51.6 & 5.6 \\
& & 64  & 70 & 5318.2 & 61.1 & 8.9 \\
& & 128 & 64 & 5700.1 & 71.7 & 9.8 \\

& \textit{multi}
& -- & 55 & 6802.7 & 90.7 & 24.7 \\
\bottomrule
\end{tabular}
\end{table}
\end{landscape}

\begin{landscape}
\begin{table}[p]
\centering
\caption[Runtime breakdown across SPs-to-CPU ratios for the 20-zone CO$_2$-cap case with 48-hour SPs]{Breakdown of total runtime for the 20-zone CO$_2$ cap case with 48-hour SPs under different SPs-to-CPU ratios.}
\label{tab:si_co2_cap_48h_20z_breakdown}
\scriptsize
\setlength{\tabcolsep}{4pt}
\begin{tabular}{c l c c c c c c c c}
\toprule
\shortstack{SPs-to-CPU\\ratio}
& Algorithm
& Groups
& \shortstack{Total\\Iterations}
& \shortstack{Total\\Runtime\\(s)}
& \shortstack{SP\\(\% Runtime)}
& \shortstack{Adaptive Clustering\\(\% Runtime)}
& \shortstack{MP\\(\% Runtime)}
& \shortstack{Interior Problem\\(\% Runtime)}
& \shortstack{Others\\(\% Runtime)} \\
\midrule

\multirow{9}{*}{3}
& \multirow{4}{*}{\shortstack[l]{\textit{adapt-G-I}}}
& 16  & \multicolumn{7}{c}{Intractable (total runtime $>$ 12h)} \\
& & 32  & \multicolumn{7}{c}{Intractable (total runtime $>$ 12h)} \\
& & 64  & 93 & 2733.5 & 60.2\% & 3.0\% & 33.3\% & 2.5\% & 3.3\% \\
& & 128 & 65 & 2169.0 & 54.1\% & 2.0\% & 36.5\% & 4.6\% & 3.3\% \\

& \multirow{4}{*}{\shortstack[l]{\textit{rep-SP}}}
& 16  & 89 & 1376.5 & 67.0\% & 3.1\% & 21.2\% & 6.7\% & 2.8\% \\
& & 32  & 93 & 2209.4 & 51.0\% & 2.5\% & 42.3\% & 3.3\% & 1.9\% \\
& & 64  & 74 & 1580.8 & 59.4\% & 2.8\% & 28.3\% & 6.9\% & 3.1\% \\
\bottomrule
\end{tabular}
\end{table}
\end{landscape}

\begin{landscape}
\begin{table}[p]
\centering
\caption[Detailed computational results for the stochastic 20-zone case with 48-hour SPs]{Detailed computational results for the stochastic 20-zone case with 48-hour SPs using three weather years.}
\label{tab:si_stochastic_48h_20z_detail}
\scriptsize
\setlength{\tabcolsep}{4pt}
\begin{tabular}{l l c c c c c c c}
\toprule
Policy
& Algorithm
& Groups
& \shortstack{Total\\Iterations}
& \shortstack{Total\\Runtime\\(s)}
& \shortstack{Total SP\\Solve Time\\(s)}
& \shortstack{Total Adaptive\\Clustering Time\\(s)}
& \shortstack{Total MP\\Solve Time\\(s)}
& \shortstack{Total Interior\\Problem Solve Time\\(s)} \\
\midrule

\multirow{9}{*}{\shortstack[l]{No CO$_2$\\policy}}
& \multirow{4}{*}{\shortstack[l]{\textit{adapt-G-I}}}
& 16  & 130 & 27537.6 & 23610.5 & 2441.5 & 1315.4 & 70.4 \\
& & 32  & 56  & 13404.7 & 11271.7 & 1658.4 & 377.5  & 30.7 \\
& & 64  & 50  & 13075.4 & 10933.1 & 1623.4 & 401.2  & 53.7 \\
& & 128 & 52  & 14474.3 & 11544.0 & 1564.1 & 1119.4 & 157.6 \\

& \multirow{4}{*}{\shortstack[l]{\textit{rep-SP}}}
& 16  & 58 & 7135.3 & 4344.6 & 1357.5 & 858.1 & 494.1 \\
& & 32  & 50 & 7456.3 & 4348.8 & 1746.9 & 815.0 & 474.6 \\
& & 64  & 46 & 7775.6 & 4844.4 & 1581.0 & 807.0 & 447.2 \\
\bottomrule
\end{tabular}
\end{table}
\end{landscape}

\begin{landscape}
\begin{table}[p]
\centering
\caption[Average SP and MP solve times for the stochastic 20-zone case with 48-hour SPs]{Average SP and MP solve times for the stochastic 20-zone case with 48-hour SPs using three weather years.}
\label{tab:si_stochastic_48h_20z_avg}
\scriptsize
\setlength{\tabcolsep}{4pt}
\begin{tabular}{l l c c c c c}
\toprule
Policy
& Algorithm
& Groups
& \shortstack{Total\\Iterations}
& \shortstack{Total\\Runtime\\(s)}
& \shortstack{Avg.\ SP\\Solve Time\\(s)}
& \shortstack{Avg.\ MP\\Solve Time\\(s)} \\
\midrule

\multirow{9}{*}{\shortstack[l]{No CO$_2$\\policy}}
& \multirow{4}{*}{\shortstack[l]{\textit{adapt-G-I}}}
& 16  & 130 & 27537.6 & 181.6 & 10.1 \\
& & 32  & 56  & 13404.7 & 201.3 & 6.7 \\
& & 64  & 50  & 13075.4 & 218.7 & 8.0 \\
& & 128 & 52  & 14474.3 & 222.0 & 21.5 \\

& \multirow{4}{*}{\shortstack[l]{\textit{rep-SP}}}
& 16  & 58 & 7135.3 & 74.9  & 14.8 \\
& & 32  & 50 & 7456.3 & 87.0  & 16.3 \\
& & 64  & 46 & 7775.6 & 105.3 & 17.5 \\
& & 128 & 44 & 7988.3 & 123.6 & 16.0 \\

& \textit{multi}
& -- & 42 & 13181.2 & 280.7 & 17.2 \\
\midrule

\multirow{9}{*}{CO$_2$ price}
& \multirow{4}{*}{\shortstack[l]{\textit{adapt-G-I}}}
& 16  & \multicolumn{4}{c}{Intractable (total runtime $>$ 12h)} \\
& & 32  & 135 & 32275.9 & 168.6 & 35.1 \\
& & 64  & 84  & 21428.7 & 189.7 & 28.9 \\
& & 128 & 51  & 14232.2 & 224.8 & 11.9 \\

& \multirow{4}{*}{\shortstack[l]{\textit{rep-SP}}}
& 16  & 72 & 9483.3  & 77.8  & 17.1 \\
& & 32  & 58 & 8058.3  & 84.5  & 12.7 \\
& & 64  & 53 & 11022.7 & 132.0 & 28.2 \\
& & 128 & 48 & 8656.2  & 123.0 & 16.3 \\

& \textit{multi}
& -- & 39 & 9195.5 & 198.9 & 29.0 \\
\midrule

\multirow{9}{*}{CO$_2$ cap}
& \multirow{4}{*}{\shortstack[l]{\textit{adapt-G-I}}}
& 16  & \multicolumn{4}{c}{Intractable (total runtime $>$ 12h)} \\
& & 32  & \multicolumn{4}{c}{Intractable (total runtime $>$ 12h)} \\
& & 64  & \multicolumn{4}{c}{Intractable (total runtime $>$ 12h)} \\
& & 128 & \multicolumn{4}{c}{Intractable (total runtime $>$ 12h)} \\

& \multirow{4}{*}{\shortstack[l]{\textit{rep-SP}}}
& 16  & 104 & 13373.5 & 94.5  & 7.0 \\
& & 32  & 89  & 13344.9 & 105.5 & 8.7 \\
& & 64  & 85  & 17289.2 & 134.5 & 32.4 \\
& & 128 & 69  & 14230.0 & 144.7 & 21.0 \\

& \textit{multi}
& -- & 58 & 16791.5 & 265.6 & 19.0 \\
\bottomrule
\end{tabular}
\end{table}
\end{landscape}

\begin{landscape}
\begin{table}[p]
\centering
\caption[Runtime breakdown for the stochastic 20-zone case with 48-hour SPs]{Breakdown of total runtime for the stochastic 20-zone case with 48-hour SPs using three weather years.}
\label{tab:si_stochastic_48h_20z_breakdown}
\scriptsize
\setlength{\tabcolsep}{4pt}
\begin{tabular}{l l c c c c c c c c}
\toprule
Policy
& Algorithm
& Groups
& \shortstack{Total\\Iterations}
& \shortstack{Total\\Runtime\\(s)}
& \shortstack{SP\\(\% Runtime)}
& \shortstack{Adaptive Clustering\\(\% Runtime)}
& \shortstack{MP\\(\% Runtime)}
& \shortstack{Interior Problem\\(\% Runtime)}
& \shortstack{Others\\(\% Runtime)} \\
\midrule

\multirow{9}{*}{\shortstack[l]{No CO$_2$\\policy}}
& \multirow{4}{*}{\shortstack[l]{\textit{adapt-G-I}}}
& 16  & 130 & 27537.6 & 85.7\% & 8.9\%  & 4.8\%  & 0.3\% & 0.4\% \\
& & 32  & 56  & 13404.7 & 84.1\% & 12.4\% & 2.8\%  & 0.2\% & 0.5\% \\
& & 64  & 50  & 13075.4 & 83.6\% & 12.4\% & 3.1\%  & 0.4\% & 0.5\% \\
& & 128 & 52  & 14474.3 & 79.8\% & 10.8\% & 7.7\%  & 1.1\% & 0.6\% \\

& \multirow{4}{*}{\shortstack[l]{\textit{rep-SP}}}
& 16  & 58 & 7135.3 & 60.9\% & 19.0\% & 12.0\% & 6.9\% & 1.1\% \\
& & 32  & 50 & 7456.3 & 58.3\% & 23.4\% & 10.9\% & 6.4\% & 1.0\% \\
& & 64  & 46 & 7775.6 & 62.3\% & 20.3\% & 10.4\% & 5.8\% & 1.2\% \\
\bottomrule
\end{tabular}
\end{table}
\end{landscape}

\section{Mathematical validity of grouped-cut BD methods}
\label{sec:benders_SI_validity}

\subsection{Fixed grouped-cut formulation (shared recourse)}
\label{sec:benders_SI_validity_fixed}

\begin{proposition}[Validity of the fixed grouped-cut formulation with shared recourse]
\label{prop:fixed_grouped_validity}
Let $\mathcal{P} = \{\mathcal{S}_g\}_{g\in\mathcal{G}}$ denote a fixed group assignment 
of the subproblem set $\mathcal{S}$. For each group $g \in \mathcal{G}$, define 
the shared recourse cost function
\begin{equation}
\Phi_g(y,q) := \sum_{s\in\mathcal{S}_g} \phi_s(y,q_s),
\end{equation}
where $\phi_s(y,q_s)$ denotes the optimal value of subproblem Eq. \ref{eq:subproblem} 
as a function of the planning decisions $(y,q_s)$. At any Benders iteration $i$, 
each constraint in Eq. \ref{eq:fixed_grouped_master_cut} is a valid Benders 
optimality cut for $\Phi_g(y,q)$, and the master problem 
Eq. \ref{eq:fixed_grouped_master} provides a valid lower bound on the objective 
value of the original model Eq. \ref{eq:budget_problem}.
\end{proposition}

\begin{proof}
For each subproblem $s \in \mathcal{S}$ and each Benders iteration 
$k \in \{0,\dots,i\}$, define the affine function:
\begin{equation}
\ell_s^{(k)}(y,q_s) := f_s^{(k)} + (\lambda_s^{(k)})^\top (y-y^{(k)}) 
+ \pi_s^{(k)} (q_s-q_s^{(k)}).
\end{equation}
By standard Benders cut validity, the affine 
function $\ell_s^{(k)}$ constructed from the dual variables of subproblem $s$ 
satisfies
\begin{equation}
\ell_s^{(k)}(y,q_s) \le \phi_s(y,q_s)
\qquad \forall s \in \mathcal{S},\;\forall k \in \{0,\dots,i\}.
\label{eq:fixed_grouped_subproblem_cut_validity}
\end{equation}
Now consider any group $\mathcal{S}_g \in \mathcal{P}$. Summing 
Eq. \ref{eq:fixed_grouped_subproblem_cut_validity} over all $s \in \mathcal{S}_g$ 
gives
\begin{equation}
\sum_{s\in\mathcal{S}_g} \ell_s^{(k)}(y,q_s)
\le
\sum_{s\in\mathcal{S}_g} \phi_s(y,q_s)
=
\Phi_g(y,q)
\qquad \forall g\in\mathcal{G},\;\forall k \in \{0,\dots,i\}.
\end{equation}
Hence each grouped cut in Eq. \ref{eq:fixed_grouped_master_cut} is a valid affine 
lower bound on $\Phi_g(y,q)$.

Now let $(y,q)$ be any feasible solution of the original model Eq. \ref{eq:budget_problem}. Since the MP retains all planning and budgeting constraints from the original model, $(y,q)$ 
is feasible with respect to 
Eq. \ref{eq:fixed_grouped_master_budget}--Eq. \ref{eq:fixed_grouped_master_domain}. 
Define
\[
\theta_g := \Phi_g(y,q)
\qquad \forall g\in\mathcal{G}.
\]
Because each grouped cut is a lower-bound on $\Phi_g(y,q)$, this choice of $\theta_g$ 
satisfies all constraints in Eq. \ref{eq:fixed_grouped_master_cut}. Moreover, 
since $\mathcal{P}$ is a grouping of $\mathcal{S}$,
\[
\sum_{g\in\mathcal{G}} \theta_g
=
\sum_{g\in\mathcal{G}} \Phi_g(y,q)
=
\sum_{s\in\mathcal{S}} \phi_s(y,q_s).
\]
Therefore, every feasible solution of the original model  Eq. \ref{eq:budget_problem} can be extended to a feasible solution of Eq. \ref{eq:fixed_grouped_master} with the same objective value.
\end{proof}

\subsection{Adaptive grouped-cut formulation (shared recourse)}
\label{sec:benders_SI_validity_adaptive_shared}

\begin{proposition}[Validity of the adaptive grouped-cut formulation with shared recourse]
\label{prop:adaptive_grouped_validity}
Let $\mathcal{P}^{(i)} = \{\mathcal{S}_g^{(i)}\}_{g \in \mathcal{G}^{(i)}}$ 
denote the current group assignment used in the MP solved at Benders iteration $i$. For each group 
$g \in \mathcal{G}^{(i)}$, define the shared recourse cost function
\begin{equation}
\Phi_g^{(i)}(y,q) := \sum_{s \in \mathcal{S}_g^{(i)}} \phi_s(y,q_s),
\end{equation}
where $\phi_s(y,q_s)$ denotes the optimal value of subproblem Eq. \ref{eq:subproblem} 
as a function of the planning decisions $(y,q_s)$. Then each constraint in 
Eq. \ref{eq:adaptive_grouped_master_cut} is a valid Benders optimality cut for 
$\Phi_g^{(i)}(y,q)$, and the master problem Eq. \ref{eq:adaptive_grouped_master} 
provides a valid lower bound on the objective value of the original model 
Eq. \ref{eq:budget_problem}.
\end{proposition}

\begin{proof}
For each subproblem $s \in \mathcal{S}$ and each historical Benders iteration 
$k \in \{0,\dots,i\}$, define the affine function
\begin{equation}
\ell_s^{(k)}(y,q_s) := f_s^{(k)} + (\lambda_s^{(k)})^\top (y-y^{(k)}) 
+ \pi_s^{(k)} (q_s-q_s^{(k)}).
\end{equation}
By standard Benders cut validity \citep{pecci_regularized_2025}, the affine 
function $\ell_s^{(k)}$, constructed from the dual variables of the subproblem, 
satisfies
\begin{equation}
\ell_s^{(k)}(y,q_s) \le \phi_s(y,q_s)
\qquad \forall s \in \mathcal{S},\;\forall k \in \{0,\dots,i\}.
\label{eq:adaptive_grouped_subproblem_cut_validity}
\end{equation}
Because the subproblem-level cut coefficients 
$(f_s^{(k)}, \lambda_s^{(k)}, \pi_s^{(k)})$ are stored for all $s \in \mathcal{S}$ 
and all $k \in \{0,\dots,i\}$, historical cuts can be re-aggregated under the 
current group assignment $\mathcal{P}^{(i)}$. Validity is preserved under 
re-aggregation, since summing valid subproblem-level inequalities over any subset 
of $\mathcal{S}$ yields a valid inequality for the corresponding shared recourse cost function.

Now consider any group $\mathcal{S}_g^{(i)} \in \mathcal{P}^{(i)}$. Summing 
Eq. \ref{eq:adaptive_grouped_subproblem_cut_validity} over all 
$s \in \mathcal{S}_g^{(i)}$ gives
\begin{equation}
\sum_{s \in \mathcal{S}_g^{(i)}} \ell_s^{(k)}(y,q_s)
\le
\sum_{s \in \mathcal{S}_g^{(i)}} \phi_s(y,q_s)
=
\Phi_g^{(i)}(y,q)
\qquad \forall g \in \mathcal{G}^{(i)},\;\forall k \in \{0,\dots,i\}.
\end{equation}
Thus, each reconstructed grouped cut in Eq. \ref{eq:adaptive_grouped_master_cut} 
is a valid Benders optimality cut for $\Phi_g^{(i)}(y,q)$.

Now let $(y,q)$ be any feasible solution of the original model 
Eq. \ref{eq:budget_problem}. Since the MP retains all planning and budgeting 
constraints from the original model, $(y,q)$ is feasible with respect to 
Eq. \ref{eq:adaptive_grouped_master_budget}--Eq. \ref{eq:adaptive_grouped_master_domain}. 
Define
\[
\theta_g := \Phi_g^{(i)}(y,q)
\qquad \forall g \in \mathcal{G}^{(i)}.
\]
Because each reconstructed grouped cut is a lower bound on $\Phi_g^{(i)}(y,q)$, this 
choice of $\theta_g$ satisfies all constraints in 
Eq. \ref{eq:adaptive_grouped_master_cut}. Moreover, since $\mathcal{P}^{(i)}$ is 
a grouping of $\mathcal{S}$,
\[
\sum_{g \in \mathcal{G}^{(i)}} \theta_g
=
\sum_{g \in \mathcal{G}^{(i)}} \Phi_g^{(i)}(y,q)
=
\sum_{s \in \mathcal{S}} \phi_s(y,q_s).
\]
Therefore, every feasible solution of the original model 
Eq. \ref{eq:budget_problem} can be extended to a feasible solution of 
Eq. \ref{eq:adaptive_grouped_master} with the same objective value.
\end{proof}

\subsection{Adaptive grouped-cut formulation (individual recourse)}
\label{sec:benders_SI_validity_adaptive_individual}

\begin{proposition}[Validity of the adaptive grouped-cut formulation with individual recourse]
\label{prop:adaptive_disaggregated_validity}
At any Benders iteration \(i\), each constraint in Eq. \ref{eq:adaptive_disaggregated_master_cut} is a valid Benders optimality cut obtained by aggregating valid subproblem-level cuts over the group assignment used at the iteration in which that cut was generated. Because the master problem Eq. \ref{eq:adaptive_disaggregated_master} retains subproblem-level recourse variables \(\theta_s\) for all \(s \in \mathcal{S}\), cuts generated under different historical group assignments remain directly compatible in the MP. Consequently, Eq. \ref{eq:adaptive_disaggregated_master} provides a valid lower bound on the objective value of the original model Eq. \ref{eq:budget_problem}.
\end{proposition}

\begin{proof}
For each subproblem \(s \in \mathcal{S}\) and each Benders iteration 
\(k \in \{0,\dots,i\}\), define the affine function
\begin{equation}
\ell_s^{(k)}(y,q_s) := f_s^{(k)} + (\lambda_s^{(k)})^\top (y-y^{(k)}) 
+ \pi_s^{(k)} (q_s-q_s^{(k)}).
\end{equation}
By standard Benders cut validity \citep{pecci_regularized_2025}, the affine 
function \(\ell_s^{(k)}\), constructed from the dual variables of the subproblem, 
satisfies
\begin{equation}
\ell_s^{(k)}(y,q_s) \le \phi_s(y,q_s)
\qquad \forall s \in \mathcal{S},\;\forall k \in \{0,\dots,i\}.
\label{eq:adaptive_disaggregated_subproblem_cut_validity}
\end{equation}

Now fix any historical iteration \(k \in \{0,\dots,i\}\) and let 
\(\mathcal{P}^{(k)} = \{\mathcal{S}_g^{(k)}\}_{g \in \mathcal{G}^{(k)}}\) denote 
the group assignment used to generate cuts at that iteration. Summing 
Eq. \ref{eq:adaptive_disaggregated_subproblem_cut_validity} over all 
\(s \in \mathcal{S}_g^{(k)}\) gives
\begin{equation}
\sum_{s \in \mathcal{S}_g^{(k)}} \ell_s^{(k)}(y,q_s)
\le
\sum_{s \in \mathcal{S}_g^{(k)}} \phi_s(y,q_s)
\qquad \forall g \in \mathcal{G}^{(k)}.
\end{equation}
Therefore, each grouped cut generated for Eq. \ref{eq:adaptive_disaggregated_master_cut} at iteration \(k\) is a valid Benders optimality cut.

Now let \((y,q)\) be any feasible solution of the original model 
Eq. \ref{eq:budget_problem}. Since the MP retains all planning and budgeting 
constraints from the original problem, \((y,q)\) is feasible with respect to 
Eq. \ref{eq:adaptive_disaggregated_master_budget}--Eq. \ref{eq:adaptive_disaggregated_master_domain}. 
Define
\[
\theta_s := \phi_s(y,q_s)
\qquad \forall s \in \mathcal{S}.
\]
Then, for any historical cut generated for group \(\mathcal{S}_g^{(k)}\),
\[
\sum_{s \in \mathcal{S}_g^{(k)}} \theta_s
=
\sum_{s \in \mathcal{S}_g^{(k)}} \phi_s(y,q_s)
\ge
\sum_{s \in \mathcal{S}_g^{(k)}} \ell_s^{(k)}(y,q_s),
\]
so all constraints in Eq. \ref{eq:adaptive_disaggregated_master_cut} are satisfied. Moreover,
\[
\sum_{s \in \mathcal{S}} \theta_s
=
\sum_{s \in \mathcal{S}} \phi_s(y,q_s).
\]
Therefore, every feasible solution of the original model Eq. \ref{eq:budget_problem} can be extended to a feasible solution of Eq. \ref{eq:adaptive_disaggregated_master} with the same objective value.
\end{proof}

\newpage
\section{Input Data and Assumptions}
\label{sec:benders_SI_inputs}

\begin{table}[htbp]
\centering
\caption[IPM regions across Benders case studies]{IPM regions included in each case study of 11, 20, and 26 zones to evaluate BD methods}
\label{tab_benders_SI_case_study_regions}
\small
\renewcommand{\arraystretch}{1.3}
\begin{tabular}{c l}
\toprule
Case study & IPM regions included \\
\midrule
11-zone &
\shortstack[l]{NENG\_CT, NENGREST, NENG\_ME, NY\_Z\_C\&E, NY\_Z\_F,\\
NY\_Z\_G-I, NY\_Z\_J, NY\_Z\_K, NY\_Z\_A, NY\_Z\_B, NY\_Z\_D} \\
\midrule
20-zone &
\shortstack[l]{NENG\_CT, NENGREST, NENG\_ME, NY\_Z\_C\&E, NY\_Z\_F,\\
NY\_Z\_G-I, NY\_Z\_J, NY\_Z\_K, NY\_Z\_A, NY\_Z\_B,\\
NY\_Z\_D, PJM\_WMAC, PJM\_EMAC, PJM\_SMAC, PJM\_West,\\
PJM\_AP, PJM\_COMD, PJM\_ATSI, PJM\_Dom, PJM\_PENE} \\
\midrule
26-zone &
\shortstack[l]{NENG\_CT, NENGREST, NENG\_ME, NY\_Z\_C\&E, NY\_Z\_F,\\
NY\_Z\_G-I, NY\_Z\_J, NY\_Z\_K, NY\_Z\_A, NY\_Z\_B,\\
NY\_Z\_D, PJM\_WMAC, PJM\_EMAC, PJM\_SMAC, PJM\_West,\\
PJM\_AP, PJM\_COMD, PJM\_ATSI, PJM\_Dom, PJM\_PENE,\\
S\_VACA, S\_C\_KY, S\_D\_AECI, S\_C\_TVA, S\_SOU, FRCC} \\
\bottomrule
\end{tabular}
\end{table}

\begin{table}[htbp]
\centering
\caption[Total annual electricity demand by region]{Total annual electricity demand by IPM region for the case studies listed in Table \ref{tab_benders_SI_case_study_regions}. Demand is based on the hourly electricity demand profiles in Princeton Net-Zero America's high-electrification (E+) scenario for 2050 \citep{larson_net-zero_2021}, with original state-level profiles mapped to IPM regions using county-level population weights derived from the 2021 U.S.\ Census Bureau data.}
\label{tab_benders_SI_regional_demand}
\small
\renewcommand{\arraystretch}{1.3}
\setlength{\tabcolsep}{4pt}
\begin{tabular}{lc}
\toprule
Region & Total annual electricity demand (TWh/y)\\
\midrule
NENG\_CT   & 63 \\
NENGREST   & 171 \\
NENG\_ME   & 28 \\
NY\_Z\_C\&E & 33 \\
NY\_Z\_F   & 20 \\
NY\_Z\_G-I & 38 \\
NY\_Z\_J   & 140 \\
NY\_Z\_K   & 46 \\
NY\_Z\_A   & 22 \\
NY\_Z\_B   & 17 \\
NY\_Z\_D   & 4 \\
PJM\_WMAC  & 66 \\
PJM\_EMAC  & 287 \\
PJM\_SMAC  & 112 \\
PJM\_West  & 268 \\
PJM\_AP    & 81 \\
PJM\_COMD  & 216 \\
PJM\_ATSI  & 135 \\
PJM\_Dom   & 155 \\
PJM\_PENE  & 32 \\
S\_VACA    & 336 \\
S\_C\_KY   & 58 \\
S\_D\_AECI & 11 \\
S\_C\_TVA  & 259 \\
S\_SOU     & 355 \\
FRCC       & 361 \\
\bottomrule
\end{tabular}
\end{table}

\begin{table}[htbp]
\centering
\caption[Existing capacity and cost parameters for power network]{Existing capacity and cost parameters in 2022 dollars for power transmission expansion in the case studies listed in Table \ref{tab_benders_SI_case_study_regions}. Existing transmission capacities are based on the U.S. Environmental Protection Agency (EPA) version of the Integrated Planning Model (IPM) \citep{us_epa_documentation_2018}.}
\label{tab_benders_SI_transmission}
\scriptsize
\renewcommand{\arraystretch}{1.3}
\setlength{\tabcolsep}{4pt}
\begin{tabular}{c l c c c c}
\toprule
\shortstack[c]{Network\\Lines}
& \shortstack[c]{Transmission\\Path}
& \shortstack[c]{Distance\\(Miles)}
& \shortstack[c]{Existing\\Capacity\\(MW)}
& \shortstack[c]{Line Reinforcement\\Cost\\(\$/MW-mile)}
& \shortstack[c]{Annualized Line\\Reinforcement Cost\\(\$/MW-mile/y)} \\
\midrule
1  & NENG\_CT to NENGREST   & 123.06 & 2,950 & 5,199.32 & 319.19 \\
2  & NENG\_CT to NY\_Z\_G-I & 76.01  & 800   & 4,087.49 & 250.94 \\
3  & NENG\_CT to NY\_Z\_K   & 55.71  & 760   & 6,227.74 & 382.33 \\
4  & NENGREST to NENG\_ME   & 196.54 & 2,000 & 5,199.28 & 319.19 \\
5  & NENGREST to NY\_Z\_F   & 99.48  & 800   & 4,087.63 & 250.95 \\
6  & NENGREST to NY\_Z\_D   & 152.99 & 150   & 4,086.72 & 250.89 \\
7  & NY\_Z\_C\&E to NY\_Z\_F   & 101.45 & 3,250 & 2,975.91 & 182.70 \\
8  & NY\_Z\_C\&E to NY\_Z\_G-I & 125.54 & 2,150 & 2,975.42 & 182.67 \\
9  & NY\_Z\_C\&E to NY\_Z\_B   & 94.50  & 1,300 & 2,975.07 & 182.64 \\
10 & NY\_Z\_C\&E to NY\_Z\_D   & 134.92 & 2,650 & 2,975.02 & 182.64 \\
11 & NY\_Z\_C\&E to PJM\_PENE  & 175.09 & 755   & 3,016.03 & 185.16 \\
12 & NY\_Z\_F to NY\_Z\_G-I & 112.35 & 3,475 & 2,975.88 & 182.69 \\
13 & NY\_Z\_G-I to NY\_Z\_J & 65.59  & 4,450 & 5,115.64 & 314.06 \\
14 & NY\_Z\_G-I to NY\_Z\_K & 81.06  & 1,290 & 5,115.66 & 314.06 \\
15 & NY\_Z\_J to NY\_Z\_K   & 48.40  & 283   & 7,256.04 & 445.46 \\
16 & NY\_Z\_K to PJM\_EMAC  & 147.32 & 660   & 6,141.72 & 377.05 \\
17 & NY\_Z\_A to NY\_Z\_B   & 54.24  & 2,270 & 2,974.56 & 182.61 \\
18 & NY\_Z\_A to PJM\_PENE  & 99.28  & 500   & 3,015.55 & 185.13 \\
19 & PJM\_WMAC to PJM\_EMAC & 101.94 & 6,900 & 4,118.85 & 252.86 \\
20 & PJM\_WMAC to PJM\_SMAC & 131.51 & 780   & 4,394.67 & 269.80 \\
21 & PJM\_WMAC to PJM\_AP   & 177.46 & 350   & 2,450.49 & 150.44 \\
22 & PJM\_WMAC to PJM\_PENE & 104.71 & 0     & 3,133.59 & 192.38 \\
23 & PJM\_EMAC to PJM\_SMAC & 91.48  & 3,600 & 5,303.08 & 325.56 \\
24 & PJM\_SMAC to PJM\_AP   & 136.61 & 1,100 & 3,634.64 & 223.14 \\
25 & PJM\_SMAC to PJM\_Dom  & 126.41 & 1,200 & 3,371.72 & 206.99 \\
26 & PJM\_West to PJM\_AP   & 198.40 & 4,800 & 1,440.47 & 88.43 \\
27 & PJM\_West to PJM\_COMD & 375.02 & 980   & 1,142.61 & 70.15 \\
28 & PJM\_West to PJM\_ATSI & 155.83 & 7,400 & 1,326.43 & 81.43 \\
29 & PJM\_West to PJM\_Dom  & 296.24 & 1,530 & 1,177.54 & 72.29 \\
30 & PJM\_West to S\_VACA   & 318.16 & 1,219 & 1,014.92 & 62.31 \\
31 & PJM\_West to S\_C\_KY  & 111.56 & 1,214 & 1,211.92 & 74.40 \\
32 & PJM\_West to S\_C\_TVA & 315.91 & 2,119 & 1,143.09 & 70.18 \\
33 & PJM\_AP to PJM\_ATSI   & 177.36 & 2,444 & 1,576.50 & 96.78 \\
34 & PJM\_AP to PJM\_Dom    & 174.55 & 5,400 & 1,427.62 & 87.64 \\
35 & PJM\_AP to PJM\_PENE   & 118.06 & 2,785 & 2,373.63 & 145.72 \\
36 & PJM\_ATSI to PJM\_PENE & 192.20 & 0     & 2,259.62 & 138.72 \\
37 & PJM\_Dom to S\_VACA    & 218.05 & 1,000 & 1,002.08 & 61.52 \\
38 & S\_VACA to S\_C\_TVA   & 378.98 & 216   & 967.64   & 59.40 \\
39 & S\_VACA to S\_SOU      & 331.43 & 1,400 & 1,050.74 & 64.51 \\
40 & S\_C\_KY to S\_C\_TVA  & 208.83 & 764   & 1,164.71 & 71.50 \\
41 & S\_D\_AECI to S\_C\_TVA & 364.08 & 0     & 1,158.52 & 71.12 \\
42 & S\_C\_TVA to S\_SOU    & 237.49 & 4,717 & 1,178.85 & 72.37 \\
43 & FRCC to S\_SOU         & 334.92 & 3,600 & 1,293.76 & 79.43 \\
\bottomrule
\end{tabular}
\end{table}

\begin{landscape}
\begin{table}[htbp]
\centering
\caption[Greenfield power generation technology cost and performance parameters]{Greenfield power generation technology cost and performance parameters. Data corresponds to 2045 costs (moderate) reported by the NREL Annual Technology Baseline 2024 in 2022 dollars \citep{mirletz_2024_2024}, with the assumption that 2050 energy system would consist of technologies built five years earlier. A discount rate of 4.5\% is used to annualize investment costs according to assumed technology lifetime and costs are converted to 2022 dollars. CC = Combined cycle, CT = Combustion turbine, CCS = CO$_2$ capture and storage.}
\label{tab_Benders_SI_Greenfield}
\small
\renewcommand{\arraystretch}{1.3}
\setlength{\tabcolsep}{4pt}
\begin{tabular}{lcccccccccc}
\toprule
\shortstack[c]{Technology}
& \shortstack[c]{Lifetime\\(year)}
& \multicolumn{2}{c}{\shortstack[c]{Investment cost}}
& \multicolumn{2}{c}{\shortstack[c]{Annualized\\CAPEX}}
& \multicolumn{2}{c}{\shortstack[c]{Fixed operation \\and maintenance \\(FOM) cost}}
& \multicolumn{1}{c}{\shortstack[c]{Variable operation \\and maintenance \\(VOM) cost}}
& \multicolumn{1}{c}{\shortstack[c]{Heat\\Rate}} \\
\cmidrule(lr){3-4} \cmidrule(lr){5-6} \cmidrule(lr){7-8}
& 
& \shortstack[c]{Power\\(\$/kW)}
& \shortstack[c]{Energy\\(\$/kWh)}
& \shortstack[c]{Power\\(\$/kW\\/year)}
& \shortstack[c]{Energy\\(\$/kWh\\/year)}
& \shortstack[c]{Power\\(\$/kW\\/year)}
& \shortstack[c]{Energy\\(\$/kWh\\/year)}
& \shortstack[c]{(\$/MWh)}
& \shortstack[c]{(MMBtu/MWh)} \\
\midrule
Natural Gas (CC)      & 30 & 1,014 & --  & 62  & --  & 28  & --  & 1.85 & 6.16 \\
Natural Gas (CC-CCS)  & 30 & 1,792 & --  & 110 & --  & 48  & --  & 3.76 & 6.88 \\
Natural Gas (CT)      & 30 & 908   & --  & 56  & --  & 23  & --  & 6.94 & 9.72 \\
Nuclear                                & 40 & 4,250 & --  & 231 & --  & 175 & --  & 2.80 & 10.50 \\
Solar                                  & 30 & 628   & --  & 39  & --  & 14  & --  & --   & --   \\
\shortstack[l]{Land Based Wind}        & 30 & 1,021 & --  & 63  & --  & 26  & --  & --   & --   \\
\shortstack[l]{Offshore Wind}          & 30 & 2,823 & --  & 173 & --  & 57  & --  & --   & --   \\
Battery                                & 15 & 283   & 179 & 26  & 17  & 7   & 4   & --   & --   \\
\bottomrule
\end{tabular}
\end{table}
\end{landscape}

\begin{landscape}
\begin{table}[htbp]
\centering
\caption[Existing power generator capacity by resource in each region]{Existing power generation capacity by resource in each IPM region for the case studies listed in Table \ref{tab_benders_SI_case_study_regions}. Existing generation capacity is aggregated by resource type for each IPM region based on the 2050 system setup used in this study. Data are obtained from EIA860 and Public Utility Data Liberation through PowerGenome \citep{schivley_powergenomepowergenome_2023}. Natural gas capacity is retained based on plant-lifetime assumptions from the NREL ReEDS model \citep{reeds_modeling_and_analysis_team_regional_2021}, and existing nuclear generators are assumed to receive second lifetime extensions consistent with the Inflation Reduction Act (IRA) \citep{the_white_house_inflation_2022}. CC = Combined cycle, CT = Combustion turbine. Values are reported in GW.}
\label{tab_benders_SI_existing_capacity_ipm}
\scriptsize
\renewcommand{\arraystretch}{1.3}
\setlength{\tabcolsep}{4pt}
\begin{tabular}{lcccccccc}
\toprule
Region
& Conventional Hydro
& Pumped Storage
& Natural Gas (CC)
& Natural Gas (CT)
& Nuclear
& Onshore Wind
& Small Hydro
& Solar PV \\
\midrule
NENG\_CT   & 0.04 & 0.03 & 4.12  & 0.60 & 2.10 & --   & 0.02 & 0.33 \\
NENG\_ME   & 0.35 & --   & 1.37  & 0.19 & --   & 0.02 & 0.16 & 0.33 \\
NENGREST   & 0.69 & 1.77 & 7.87  & 0.99 & 1.25 & --   & 0.17 & 1.55 \\
NY\_Z\_A   & 2.43 & 0.24 & 0.38  & 0.05 & --   & 0.12 & --   & 0.08 \\
NY\_Z\_B   & 0.04 & --   & 0.12  & --   & --   & 0.34 & --   & 0.10 \\
NY\_Z\_C\&E& 0.06 & --   & 1.65  & --   & 2.16 & 0.38 & 0.19 & 0.50 \\
NY\_Z\_D   & 0.82 & --   & 0.37  & --   & --   & --   & 0.02 & 0.03 \\
NY\_Z\_F   & 0.27 & 1.17 & 3.53  & --   & --   & --   & 0.08 & 0.21 \\
NY\_Z\_G-I & --   & --   & 1.90  & --   & --   & --   & 0.03 & 0.12 \\
NY\_Z\_J   & --   & --   & 3.79  & 2.00 & --   & --   & --   & 0.01 \\
NY\_Z\_K   & --   & --   & 0.75  & 0.58 & --   & --   & --   & 0.13 \\
PJM\_AP    & 0.10 & --   & 3.35  & 1.29 & --   & 0.12 & 0.11 & 0.21 \\
PJM\_ATSI  & --   & --   & 6.04  & 1.62 & 2.18 & --   & 0.01 & 0.37 \\
PJM\_COMD  & 0.01 & --   & 5.03  & 6.20 & 6.52 & 0.09 & 0.01 & 0.16 \\
PJM\_Dom   & 0.62 & 3.00 & 9.16  & 2.61 & 3.71 & --   & 0.01 & 3.49 \\
PJM\_EMAC  & 0.57 & 1.49 & 14.70 & 2.15 & 8.50 & --   & --   & 1.29 \\
PJM\_PENE  & 0.03 & 0.48 & 2.99  & 0.33 & --   & --   & 0.05 & 0.04 \\
PJM\_SMAC  & --   & --   & 1.86  & 0.27 & 1.77 & --   & --   & 0.11 \\
PJM\_WMAC  & 0.71 & --   & 9.09  & 0.14 & 2.49 & 0.08 & 0.01 & 0.04 \\
PJM\_West  & 0.79 & 0.24 & 10.27 & 6.43 & 4.15 & 0.45 & 0.10 & 1.39 \\
S\_C\_KY   & 0.30 & --   & 0.68  & 1.50 & --   & --   & 0.01 & 0.01 \\
S\_C\_TVA  & 4.66 & 1.68 & 10.80 & 2.32 & 8.54 & --   & 0.10 & 1.12 \\
S\_D\_AECI & --   & 0.03 & 1.92  & 0.35 & --   & 0.10 & --   & --   \\
S\_SOU     & 3.91 & 1.90 & 22.87 & 5.87 & 10.29& --   & 0.05 & 4.14 \\
S\_VACA    & 2.17 & 2.80 & 9.97  & 7.83 & 12.10& --   & 0.12 & 7.07 \\
FRCC       & 0.04 & --   & 36.98 & 7.95 & 3.74 & --   & --   & 5.45 \\
\bottomrule
\end{tabular}
\end{table}
\end{landscape}

\begin{landscape}
\begin{table}[htbp]
\centering
\caption[Regional cost multipliers for new power generation technologies]{Regional cost multipliers for investment cost of new power generation technologies by IPM region for the case studies listed in Table \ref{tab_benders_SI_case_study_regions}, obtained from the electricity market module in EIA’s Annual Energy Outlook (AEO) 2023 \citep{us_eia_annual_2023}. CC = Combined cycle, CT = Combustion turbine, CCS = CO$_2$ capture and storage. These multipliers are applied to the baseline investment costs of each greenfield technology in Table \ref{tab_Benders_SI_Greenfield} as region-specific model inputs to account for regional variation in investment costs.}
\label{tab_benders_SI_regional_cost_multipliers}
\scriptsize
\renewcommand{\arraystretch}{1.3}
\setlength{\tabcolsep}{4pt}
\begin{tabular}{lccccccccc}
\toprule
Region & Natural Gas (CC) & Natural Gas (CT) & Natural Gas (CC-CCS) & Nuclear & Battery & Hydro & Onshore Wind & Offshore Wind & Solar \\
\midrule
NENG\_CT   & 1.19 & 1.13 & 1.06 & 1.13 & 1.03 & 0.66 & 1.27 & 1.00 & 1.02 \\
NENG\_ME   & 1.19 & 1.13 & 1.06 & 1.13 & 1.03 & 0.66 & 1.27 & 1.00 & 1.02 \\
NENGREST   & 1.19 & 1.13 & 1.06 & 1.13 & 1.03 & 0.66 & 1.27 & 1.00 & 1.02 \\
NY\_Z\_A   & 1.17 & 1.09 & 1.04 & 1.06 & 1.00 & 1.34 & 1.54 & 1.22 & 1.01 \\
NY\_Z\_B   & 1.17 & 1.09 & 1.04 & 1.06 & 1.00 & 1.34 & 1.54 & 1.22 & 1.01 \\
NY\_Z\_C\&E & 1.17 & 1.09 & 1.04 & 1.06 & 1.00 & 1.34 & 1.54 & 1.22 & 1.01 \\
NY\_Z\_D   & 1.17 & 1.09 & 1.04 & 1.06 & 1.00 & 1.34 & 1.54 & 1.22 & 1.01 \\
NY\_Z\_F   & 1.17 & 1.09 & 1.04 & 1.06 & 1.00 & 1.34 & 1.54 & 1.22 & 1.01 \\
NY\_Z\_G-I & 1.17 & 1.09 & 1.04 & 1.06 & 1.00 & 1.34 & 1.54 & 1.22 & 1.01 \\
NY\_Z\_J   & 1.62 & 1.46 & 1.20 & --   & 1.03 & --   & --   & 1.01 & 1.20 \\
NY\_Z\_K   & 1.62 & 1.46 & 1.20 & --   & 1.03 & --   & --   & 1.01 & 1.20 \\
PJM\_AP    & 0.98 & 0.96 & 0.97 & 0.98 & 1.00 & 1.28 & 0.96 & 1.00 & 0.98 \\
PJM\_ATSI  & 0.98 & 0.96 & 0.97 & 0.98 & 1.00 & 1.28 & 0.96 & 1.00 & 0.98 \\
PJM\_COMD  & 1.25 & 1.24 & 1.10 & 1.22 & 1.01 & 1.22 & 1.26 & 1.32 & 1.07 \\
PJM\_Dom   & 1.10 & 1.02 & 1.01 & 1.02 & 1.01 & 1.24 & 1.32 & 1.04 & 0.98 \\
PJM\_EMAC  & 1.19 & 1.12 & 1.05 & 1.11 & 1.01 & 1.40 & 1.27 & 1.00 & 1.04 \\
PJM\_PENE  & 1.19 & 1.12 & 1.05 & 1.11 & 1.01 & 1.40 & 1.27 & 1.00 & 1.04 \\
PJM\_SMAC  & 1.19 & 1.12 & 1.05 & 1.11 & 1.01 & 1.40 & 1.27 & 1.00 & 1.04 \\
PJM\_West  & 0.98 & 0.96 & 0.97 & 0.98 & 1.00 & 1.28 & 0.96 & 1.00 & 0.98 \\
PJM\_WMAC  & 1.19 & 1.12 & 1.05 & 1.11 & 1.01 & 1.40 & 1.27 & 1.00 & 1.04 \\
S\_C\_KY   & 0.94 & 0.95 & 0.96 & 1.02 & 1.03 & 0.47 & 0.96 & --   & 0.96 \\
S\_C\_TVA  & 0.94 & 0.95 & 0.96 & 1.02 & 1.03 & 0.47 & 0.96 & --   & 0.96 \\
S\_D\_AECI & 1.05 & 1.05 & 1.01 & 1.11 & 1.03 & --   & 0.95 & --   & 1.02 \\
S\_SOU     & 0.93 & 0.92 & 0.94 & 1.01 & 1.02 & 1.49 & 1.29 & --   & 0.96 \\
S\_VACA    & 0.91 & 0.91 & 0.94 & 1.01 & 1.03 & 0.69 & 1.14 & 0.90 & 1.00 \\
FRCC       & 0.93 & 0.93 & 0.96 & 0.97 & 1.00 & 1.78 & --   & 1.00 & 0.95 \\
\bottomrule
\end{tabular}
\end{table}
\end{landscape}

\begin{landscape}
\begin{table}[htbp]
\centering
\caption[Unit commitment parameters for thermal power generators]{Unit commitment parameters for thermal power generation resources. Parameters from nuclear power generators obtained from \cite{sepulveda_role_2018, jenkins_benefits_2018}, existing gas power generators from \cite{schivley_powergenomepowergenome_2023}, and new gas power generators from \cite{mit_energy_initiative_future_2022}.}
\label{tab_benders_SI_thermal_uc}
\scriptsize
\renewcommand{\arraystretch}{1.3}
\setlength{\tabcolsep}{4pt}
\begin{tabular}{lcccccc}
\toprule
\shortstack[c]{Technology}
& \shortstack[c]{Start cost\\(\$/MW)}
& \shortstack[c]{Min up\\time (h)}
& \shortstack[c]{Min down\\time (h)}
& \shortstack[c]{Max hourly ramping rate\\(\% of nameplate\\capacity)}
& \shortstack[c]{Min output\\(\% of nameplate\\capacity)}
& \shortstack[c]{Max output\\(\% of nameplate\\capacity)} \\
\midrule
Existing Natural Gas (CC) 
& 101 & 6 & 6 & 64\% & 0--85.4\% & 100\% \\

Existing Natural Gas (CT) Power Generator
& 131 & 1 & 1 & 64\% & 13.5--67.7\% & 100\% \\

Existing Nuclear 
& 1,130 & 36 & 36 & 25\% & 50\% & 100\% \\

New Natural Gas (CC) 
& 69 & 4 & 4 & 100\% & 30\% & 100\% \\

New Natural Gas (CC-CCS)
& 110 & 4 & 4 & 100\% & 50\% & 100\% \\

New Natural Gas (CT)
& 158 & -- & -- & 100\% & 25\% & 100\% \\

New Nuclear Power Generator 
& 1,130 & 36 & 36 & 25\% & 30\% & 100\% \\
\bottomrule
\end{tabular}
\end{table}
\end{landscape}

\begin{landscape}
\begin{table}[htbp]
\centering
\caption[Fuel cost by region]{Regional cost of fuels for power generation by IPM region for the case studies listed in Table \ref{tab_benders_SI_case_study_regions}, obtained from EIA’s AEO 2023 reference scenario \citep{us_eia_annual_2023}.}
\label{tab_benders_SI_fuel_cost}
\scriptsize
\renewcommand{\arraystretch}{1.3}
\setlength{\tabcolsep}{6pt}
\begin{tabular}{lcc}
\toprule
Region & Natural Gas Price (\$/MMBtu) & Uranium Price (\$/MMBtu)\\
\midrule
NENG\_CT    & 3.51 & 0.71 \\
NENGREST    & 3.51 & 0.71 \\
NENG\_ME    & 3.51 & 0.71 \\
NY\_Z\_C\&E & 2.88 & 0.71 \\
NY\_Z\_F    & 2.88 & 0.71 \\
NY\_Z\_G-I  & 2.88 & 0.71 \\
NY\_Z\_J    & 2.88 & 0.71 \\
NY\_Z\_K    & 2.88 & 0.71 \\
NY\_Z\_A    & 2.88 & 0.71 \\
NY\_Z\_B    & 2.88 & 0.71 \\
NY\_Z\_D    & 2.88 & 0.71 \\
PJM\_WMAC   & 2.88 & 0.71 \\
PJM\_EMAC   & 3.01 & 0.71 \\
PJM\_SMAC   & 4.08 & 0.71 \\
PJM\_WEST   & 3.35 & 0.71 \\
PJM\_AP     & 3.39 & 0.71 \\
PJM\_COMD   & 3.07 & 0.71 \\
PJM\_ATSI   & 3.06 & 0.71 \\
PJM\_DOM    & 4.08 & 0.71 \\
PJM\_PENE   & 2.88 & 0.71 \\
S\_VACA     & 4.08 & 0.71 \\
S\_C\_KY    & 3.88 & 0.71 \\
S\_D\_AECI  & 3.66 & 0.71 \\
S\_C\_TVA   & 3.89 & 0.71 \\
S\_SOU      & 4.03 & 0.71 \\
FRCC        & 4.08 & 0.71 \\
\bottomrule
\end{tabular}
\end{table}
\end{landscape}

\begin{landscape}
\begin{table}[htbp]
\centering
\caption[Weighted regional average annual capacity factor and maximum available capacity expansion of greenfield VRE]{Weighted regional average annual capacity factor and maximum available capacity expansion of greenfield variable renewable energy (VRE) resources by IPM region for the case studies listed in Table \ref{tab_benders_SI_case_study_regions}. Solar and wind resource profiles were obtained using ZEPHYR (Zero-emissions Electricity system Planning with HourlY operational Resolution), based on 2011 weather data from NREL NSRDB (National Solar Radiation Database) and WTK (WIND Toolkit) \citep{sengupta_national_2018, draxl_wind_2015, brown_zephyr_2022}, which is consistent with the weather year used in the Net-Zero America demand profiles \citep{larson_net-zero_2021}. Together with transmission data from NREL ReEDS, ZEPHYR provides regional supply-curve bins with hourly capacity factors, maximum capacity expansion, and interconnection costs for solar, onshore wind, and offshore wind where available \citep{reeds_modeling_and_analysis_team_regional_2021}. The interconnection cost of each bin is added onto the investment cost of the corresponding greenfield technology in Table \ref{tab_Benders_SI_Greenfield}.}

\label{tab_benders_SI_vre_cf_capacity}
\scriptsize
\renewcommand{\arraystretch}{1.3}
\setlength{\tabcolsep}{6pt}
\begin{tabular}{lcccccc}
\toprule
Region
& \shortstack[c]{Offshore Wind\\ Capacity (GW)}
& \shortstack[c]{Onshore Wind\\ Capacity (GW)}
& \shortstack[c]{Solar\\ Capacity (GW)}
& \shortstack[c]{Offshore Wind\\Avg. CF}   
& \shortstack[c]{Onshore Wind\\Avg. CF}
& \shortstack[c]{Solar\\ Avg. CF}  \\
\midrule
NENGREST   & 10.6 & 1{,}388.6  & 24{,}475.3  & 0.49 & 0.38 & 0.21 \\
NENG\_CT   & --   & 12.7       & 224.0       & --   & 0.40 & 0.22 \\
NENG\_ME   & 10.6 & 122.2      & 2{,}113.2   & 0.56 & 0.41 & 0.21 \\
NY\_Z\_A   & --   & 16.8       & 287.7       & --   & 0.41 & 0.20 \\
NY\_Z\_B   & --   & 14.1       & 249.8       & --   & 0.41 & 0.20 \\
NY\_Z\_C\&E & --  & 68.6       & 1{,}205.7   & --   & 0.39 & 0.20 \\
NY\_Z\_D   & --   & 21.7       & 369.1       & --   & 0.40 & 0.20 \\
NY\_Z\_F   & --   & 38.1       & 670.7       & --   & 0.37 & 0.21 \\
NY\_Z\_G-I & --   & 14.3       & 253.6       & --   & 0.36 & 0.21 \\
NY\_Z\_J   & 4.9  & --         & --          & 0.54 & --   & --   \\
NY\_Z\_K   & 10.9 & 0.8        & 21.2        & 0.50 & 0.45 & 0.23 \\
PJM\_AP    & --   & 1{,}747.9  & 30{,}662.0  & --   & 0.35 & 0.21 \\
PJM\_ATSI  & --   & 252.5      & 4{,}784.2   & --   & 0.41 & 0.20 \\
PJM\_COMD  & --   & 33.9       & 1{,}133.3   & --   & 0.45 & 0.22 \\
PJM\_Dom   & 11.4 & 426.3      & 7{,}484.9   & 0.49 & 0.35 & 0.24 \\
PJM\_EMAC  & 8.5  & 880.9      & 15{,}611.2  & 0.55 & 0.39 & 0.23 \\
PJM\_PENE  & --   & 64.6       & 1{,}137.0   & --   & 0.38 & 0.20 \\
PJM\_SMAC  & --   & 9.5        & 166.9       & --   & 0.38 & 0.23 \\
PJM\_WMAC  & --   & 41.5       & 830.7       & --   & 0.37 & 0.21 \\
PJM\_West  & --   & 11{,}312.7 & 222{,}838.4 & --   & 0.38 & 0.22 \\
S\_C\_KY   & --   & 44.5       & 890.3       & --   & 0.38 & 0.22 \\
S\_C\_TVA  & --   & 17{,}676.3 & 312{,}825.9 & --   & 0.36 & 0.24 \\
S\_D\_AECI & --   & 70.0       & 1{,}346.2   & --   & 0.44 & 0.23 \\
S\_SOU     & --   & 6{,}401.4  & 102{,}222.4 & --   & 0.33 & 0.26 \\
S\_VACA    & 10.2 & 2{,}173.4  & 38{,}361.6  & 0.51 & 0.35 & 0.26 \\
FRCC       & 11.0 & 123.2      & 2{,}194.3   & 0.33 & 0.30 & 0.27 \\
\bottomrule
\end{tabular}
\end{table}
\end{landscape}

\begin{table}[htbp]
\centering
\caption[Discrete capacity sizes for power generation and storage resources]{Discrete capacity sizes for existing and new power generation and storage resource types. Capacity sizes for nuclear power generators are obtained from \cite{jenkins_benefits_2018, sepulveda_role_2018}, existing natural gas power generators from \cite{schivley_powergenomepowergenome_2023}, new natural gas power generators from \cite{mit_energy_initiative_future_2022}, new battery and variable renewable energy technologies from NREL ATB 2024 \citep{mirletz_2024_2024}, transmission lines from NREL ReEDS and supporting WECC transmission information \citep{reeds_modeling_and_analysis_team_regional_2021,wecc_transmission_nodate}. Values are reported in MW. For existing resource types, ranges indicate the minimum and maximum unit capacities represented in the 2050 setup, while single values for new resource types denote the discrete investment block sizes used in the MILP formulation.}
\label{tab_benders_SI_discrete_capacity}
\small
\renewcommand{\arraystretch}{1.3}
\setlength{\tabcolsep}{4pt}
\begin{tabular}{lc}
\toprule
Technology &
Discrete capacity size (MW) \\
\midrule
Existing Hydro Pumped Storage                    & 9.7--500.5 \\
Existing Conventional Hydro                      & 0.4--187.3 \\
Existing Small Hydro                             & 0.27--12.4 \\
Existing Natural Gas (CC) Power Generator        & 2--1472.2 \\
Existing Natural Gas (CT) Power Generator        & 0.5--183.1 \\
Existing Nuclear Power Generator                 & 852.7--1301.1 \\
Existing Onshore Wind                            & 15.3--340 \\
Existing Solar                                   & 1.9--50.9 \\
\midrule
New Natural Gas (CC-CCS) Power Generator         & 377 \\
New Natural Gas (CC) Power Generator             & 573 \\
New Natural Gas (CT) Power Generator             & 237 \\
New Nuclear Power Generator                      & 1000 \\
New Onshore Wind                                 & 3.2 \\
New Battery                                      & 60 \\
New Solar                                        & 100 \\
New Offshore Wind                                & 12 \\
Transmission Line                                & 1500 \\
\bottomrule
\end{tabular}
\end{table}

\section{Model implementation details}
\label{Benders_SI_Implementation}

The MACRO model codebase used in this analysis is available in the following fork and branch of the MacroEnergy.jl GitHub repository: \url{https://github.com/Junwenlaw/MacroEnergy.jl/tree/Clustering-Enhanced-Benders-Decomposition}. 

The BD algorithms are available in the following fork and branch of the MacroEnergySolvers.jl GitHub repository: \url{https://github.com/Junwenlaw/JWL_MacroEnergySolvers.jl/tree/Clustering-Enhanced-Benders-Decomposition}. 
    \bibliographystyle{elsarticle-harv}
    \bibliography{References}

@article{he_sector_2021,
	title = {Sector coupling \textit{via} hydrogen to lower the cost of energy system decarbonization},
	volume = {14},
	issn = {1754-5692, 1754-5706},
	doi = {10.1039/D1EE00627D},
	pages = {4635--4646},
	number = {9},
	journal = {Energy Environ. Sci.},
	author = {He, Guannan and Mallapragada, Dharik S. and Bose, Abhishek and Heuberger-Austin, Clara F. and Gençer, Emre},
	urldate = {2024-06-11},
	year = {2021},
	langid = {english},
}

@report{larson_net-zero_2021,
	location = {Princeton, {NJ}},
	title = {Net-Zero America: Potential Pathways, Infrastructure, and Impacts},
	institution = {Princeton University},
	author = {Larson, E and Greig, C and Jenkins, J and Mayfield, E and Pascale, A and Zhang, C and Drossman, J and Williams, R and Pacala, S and Socolow, R and Baik, {EJ} and Birdsey, R and Duke, R and Jones, R and Haley, B and Leslie, E and Paustian, K and Swan, A},
	year = {2021},
    url = {https://netzeroamerica.princeton.edu/?explorer=year&state=national&table=2020&limit=200},
}

@report{venkatesh_open_2022,
	title = {An Open Energy Outlook: Decarbonization Pathways for the {USA}},
	institution = {Carnegie Mellon University \& North Carolina State University},
	author = {Venkatesh, Aranya and Jordan, Katherine and Sinha, Aditya and Johnson, Jeremiah and Jaramillo, Paulina},
	year = {2022},
    url = {https://energy.cmu.edu/key-initiatives/open-energy-outlook/index.html},
}

@article{mallapragada_long-run_2020,
	title = {Long-run system value of battery energy storage in future grids with increasing wind and solar generation},
	volume = {275},
	issn = {03062619},
	doi = {10.1016/j.apenergy.2020.115390},
	pages = {115390},
	journal = {Applied Energy},
	author = {Mallapragada, Dharik S. and Sepulveda, Nestor A. and Jenkins, Jesse D.},
	urldate = {2024-06-11},
	year = {2020},
	langid = {english},
}

@article{draxl_wind_2015,
	title = {The Wind Integration National Dataset ({WIND}) Toolkit},
	volume = {151},
	issn = {03062619},
	doi = {10.1016/j.apenergy.2015.03.121},
	pages = {355--366},
	journal = {Applied Energy},
	author = {Draxl, Caroline and Clifton, Andrew and Hodge, Bri-Mathias and {McCaa}, Jim},
	urldate = {2024-06-11},
	year = {2015},
	langid = {english},
}

@article{sengupta_national_2018,
	title = {The National Solar Radiation Data Base ({NSRDB})},
	volume = {89},
	issn = {13640321},
	doi = {10.1016/j.rser.2018.03.003},
	pages = {51--60},
	journal = {Renewable and Sustainable Energy Reviews},
	author = {Sengupta, Manajit and Xie, Yu and Lopez, Anthony and Habte, Aron and Maclaurin, Galen and Shelby, James},
	urldate = {2024-06-11},
	year = {2018},
	langid = {english},
}

@report{reeds_modeling_and_analysis_team_regional_2021,
	title = {Regional Energy Deployment System ({ReEDS}) Model Documentation: Version 2020},
	url = {https://www.osti.gov/servlets/purl/1788425/},
	doi = {10.2172/1788425},
	pages = {NREL/TP--6A20--78195, 1788425, MainId:32104},
	number = {{NREL}/{TP}-6A20-78195, 1788425, {MainId}:32104},
	author = {{ReEDS Modeling and Analysis Team} and Ho, Jonathan and Becker, Jonathon and Brown, Maxwell and Brown, Patrick and Chernyakhovskiy, Ilya and Cohen, Stuart and Cole, Wesley and Corcoran, Sean and Eurek, Kelly and Frazier, Will and Gagnon, Pieter and Gates, Nathaniel and Greer, Daniel and Jadun, Paige and Khanal, Saroj and Machen, Scott and Macmillan, Madeline and Mai, Trieu and Mowers, Matthew and Murphy, Caitlin and Rose, Amy and Schleifer, Anna and Sergi, Brian and Steinberg, Daniel and Sun, Yinong and Zhou, Ella},
	urldate = {2024-06-11},
	year = {2021},
}

@software{brown_zephyr_2022,
	title = {{ZEPHYR}},
	url = {https://github.com/patrickbrown4/zephyr},
	author = {Brown, Patrick},
	year = {2022},
}

@software{schivley_powergenomepowergenome_2023,
	title = {{PowerGenome}/{PowerGenome}: v0.6.1},
	rights = {Open Access},
	url = {https://zenodo.org/record/8329515},
	doi = {10.5281/ZENODO.8329515},
	shorttitle = {{PowerGenome}/{PowerGenome}},
	version = {v0.6.1},
	publisher = {Zenodo},
	author = {Schivley, Greg and Welty, Ethan and Patankar, Neha and Jacobson, Anna and Xu, Qingyu and {Aneesha Manocha} and Pecora, Braden and Bhandarkar, Riti and Jenkins, Jesse D.},
	urldate = {2024-06-11},
	year = {2023},
}

@report{us_epa_documentation_2018,
	location = {Washington, D.C.},
	title = {Documentation for {EPA}’s Power Sector Modeling Platform v6 Using the Integrated Planning Model},
	institution = {{EPA}},
	author = {{U.S. EPA}},
	year = {2018},
    url ={https://www.epa.gov/power-sector-modeling/documentation-ipm-platform-v6-all-chapters}
}

@article{sepulveda_role_2018,
	title = {The Role of Firm Low-Carbon Electricity Resources in Deep Decarbonization of Power Generation},
	volume = {2},
	issn = {25424351},
	url = {https://linkinghub.elsevier.com/retrieve/pii/S2542435118303866},
	doi = {10.1016/j.joule.2018.08.006},
	pages = {2403--2420},
	number = {11},
	journal = {Joule},
	author = {Sepulveda, Nestor A. and Jenkins, Jesse D. and De Sisternes, Fernando J. and Lester, Richard K.},
	urldate = {2025-02-03},
	year = {2018},
	langid = {english},
}

@article{petkov_power--hydrogen_2020,
	title = {Power-to-hydrogen as seasonal energy storage: an uncertainty analysis for optimal design of low-carbon multi-energy systems},
	volume = {274},
	issn = {03062619},
	doi = {10.1016/j.apenergy.2020.115197},
	shorttitle = {Power-to-hydrogen as seasonal energy storage},
	pages = {115197},
	journal = {Applied Energy},
	author = {Petkov, Ivalin and Gabrielli, Paolo},
	urldate = {2025-02-03},
	year = {2020},
	langid = {english},
}

@article{jenkins_benefits_2018,
	title = {The benefits of nuclear flexibility in power system operations with renewable energy},
	volume = {222},
	issn = {03062619},
	url = {https://linkinghub.elsevier.com/retrieve/pii/S0306261918303180},
	doi = {10.1016/j.apenergy.2018.03.002},
	pages = {872--884},
	journal = {Applied Energy},
	author = {Jenkins, J.D. and Zhou, Z. and Ponciroli, R. and Vilim, R.B. and Ganda, F. and De Sisternes, F. and Botterud, A.},
	urldate = {2025-02-03},
	year = {2018},
	langid = {english},
}

@article{mignone_drivers_2024,
	title = {Drivers and implications of alternative routes to fuels decarbonization in net-zero energy systems},
	volume = {15},
	rights = {2024 Bezos Earth Fund, {EPRI}, {ExxonMobil} Technology and Engineering Co., {NREL}, Battelle Memorial Institute and The Author(s)},
	issn = {2041-1723},
	doi = {10.1038/s41467-024-47059-0},
	pages = {3938},
	number = {1},
	journal = {Nature Communications},
	author = {Mignone, Bryan K. and Clarke, Leon and Edmonds, James A. and Gurgel, Angelo and Herzog, Howard J. and Johnson, Jeremiah X. and Mallapragada, Dharik S. and {McJeon}, Haewon and Morris, Jennifer and O’Rourke, Patrick R. and Paltsev, Sergey and Rose, Steven K. and Steinberg, Daniel C. and Venkatesh, Aranya},
	urldate = {2025-05-26},
	year = {2024},
	langid = {english},
}

@article{giovanniello_influence_2024,
	title = {The influence of additionality and time-matching requirements on the emissions from grid-connected hydrogen production},
	volume = {9},
	rights = {2024 The Author(s), under exclusive licence to Springer Nature Limited},
	issn = {2058-7546},
	doi = {10.1038/s41560-023-01435-0},
	pages = {197--207},
	number = {2},
	journal = {Nature Energy},
	author = {Giovanniello, Michael A. and Cybulsky, Anna N. and Schittekatte, Tim and Mallapragada, Dharik S.},
	urldate = {2025-05-26},
	year = {2024},
	langid = {english},
}

@software{he_dolphyn_2023,
	title = {{DOLPHYN}: decision optimization for low-carbon power and hydrogen networks},
	rights = {{GPL}-2.0},
	url = {https://github.com/macroenergy/Dolphyn.jl},
	publisher = {{MacroEnergy}},
	author = {He, Guannan and Mallapragada, Dharik and Macdonald, Ruaridh and Law, Jun Wen and Shaker, Youssef and Zhang, Yuheng and Cybulsky, Anna and Chakraborty, Shantanu and Giovanniello, Michael},
	urldate = {2025-05-31},
	year = {2023},
}

@article{law_role_2025,
	title = {Role of Technology Flexibility and Grid Coupling on Hydrogen Deployment in Net-Zero Energy Systems},
	volume = {59},
	rights = {https://creativecommons.org/licenses/by/4.0/},
	issn = {0013-936X, 1520-5851},
	doi = {10.1021/acs.est.4c12166},
	pages = {4974--4988},
	number = {10},
	journal = {Environ. Sci. Technol.},
	author = {Law, Jun Wen and Mignone, Bryan K. and Mallapragada, Dharik S.},
	urldate = {2025-06-01},
	year = {2025},
	langid = {english},
}

@report{us_eia_annual_2023,
	title = {Annual Energy Outlook 2023},
	author = {{U.S. EIA}},
	year = {2023},
    url ={https://www.eia.gov/outlooks/aeo/pdf/aeo2023_narrative.pdf},
}

@misc{mirletz_2024_2024,
	title = {2024 Annual Technology Baseline ({ATB}) Cost and Performance Data for Electricity Generation Technologies},
	url = {https://www.osti.gov/servlets/purl/2377191/},
	doi = {10.25984/2377191},
	publisher = {{DOE} Open Energy Data Initiative ({OEDI}); National Renewable Energy Laboratory ({NREL})},
	author = {Mirletz, Brian and Vimmerstedt, Laura and Stehly, Tyler and Stright, Dana and Cohen, Stuart and Cole, Wesley and Duffy, Patrick and Feldman, David and Kurup, Parthiv and Ramasamy, Vignesh and Zuboy, Jarett and Oladosu, Gbadebo and Hoffmann, Jeffrey and Eberle, Annika and Roberts, Owen and Mulas Hernando, Daniel and Avery, Greg and Rosenlieb, Evan and Schleifer, Anna and Akindipe, Dayo and Witter, Eric and Fuchs, Becca and Zuckerman, Gabriel and Zolan, Alex and Abou Jaoude, Abdalla and Larson, Levi and Lohse, Chris and Guaita, Nahuel and Trivedi, Ishita and Joseck, Fred and Hedalen, Ty and Rakov, Ben and Sekar, Ashok},
	urldate = {2025-09-08},
	year = {2024},
	langid = {english},
}

@report{the_white_house_inflation_2022,
	title = {Inflation Reduction Act Guidebook - Clean Energy},
	url = {https://bidenwhitehouse.archives.gov/cleanenergy/inflation-reduction-act-guidebook/},
	author = {{The White House}},
	year = {2022},
}

@report{mit_energy_initiative_future_2022,
	location = {Cambridge, {MA}},
	title = {The Future of Energy Storage},
	institution = {{MITEI}},
	author = {{MIT Energy Initiative}},
	year = {2022},
    url = {https://energy.mit.edu/wp-content/uploads/2022/05/The-Future-of-Energy-Storage.pdf},
}

@article{shaker_multi-sectoral_2025,
	title = {Multi-sectoral impacts of H$_{\textrm{2}}$ and synthetic fuels adoption for heavy-duty transportation decarbonization},
	volume = {3},
	issn = {2753-3751},
	doi = {10.1088/2753-3751/ae58ad},
	pages = {025006},
	number = {2},
	journal = {Environ. Res.: Energy},
	author = {Shaker, Youssef and Law, Jun Wen and Botterud, Audun and Mallapragada, Dharik},
	urldate = {2026-04-21},
	year = {2026},
}

@article{law_decarbonization_2025,
	title = {Decarbonization pathways for liquid fuels: a multi-sector energy system perspective},
	issn = {2398-4902},
	doi = {10.1039/D5SE01654A},
	shorttitle = {Decarbonization pathways for liquid fuels},
	pages = {10.1039.D5SE01654A},
	journal = {Sustainable Energy \& Fuels},
	author = {Law, Jun Wen and Mignone, Bryan K. and Mallapragada, Dharik S.},
	urldate = {2026-04-25},
	year = {2026},
	langid = {english},
}

@article{brown_synergies_2018,
	title = {Synergies of sector coupling and transmission reinforcement in a cost-optimised, highly renewable European energy system},
	volume = {160},
	issn = {03605442},
	doi = {10.1016/j.energy.2018.06.222},
	pages = {720--739},
	journal = {Energy},
	author = {Brown, T. and Schlachtberger, D. and Kies, A. and Schramm, S. and Greiner, M.},
	urldate = {2026-03-13},
	year = {2018},
	langid = {english},
}

@article{kondziella_flexibility_2016,
	title = {Flexibility requirements of renewable energy based electricity systems – a review of research results and methodologies},
	volume = {53},
	issn = {13640321},
	doi = {10.1016/j.rser.2015.07.199},
	pages = {10--22},
	journal = {Renewable and Sustainable Energy Reviews},
	author = {Kondziella, Hendrik and Bruckner, Thomas},
	urldate = {2026-03-13},
	year = {2016},
	langid = {english},
}

@article{lund_review_2015,
	title = {Review of energy system flexibility measures to enable high levels of variable renewable electricity},
	volume = {45},
	issn = {13640321},
	doi = {10.1016/j.rser.2015.01.057},
	pages = {785--807},
	journal = {Renewable and Sustainable Energy Reviews},
	author = {Lund, Peter D. and Lindgren, Juuso and Mikkola, Jani and Salpakari, Jyri},
	urldate = {2026-03-13},
	year = {2015},
	langid = {english},
}

@article{brown_pypsa_2018,
	title = {{PyPSA}: Python for Power System Analysis},
	volume = {6},
	rights = {http://creativecommons.org/licenses/by/4.0},
	issn = {2049-9647},
	doi = {10.5334/jors.188},
	shorttitle = {{PyPSA}},
	pages = {4},
	number = {1},
	journal = {{JORS}},
	author = {Brown, Thomas and Hörsch, Jonas and Schlachtberger, David},
	urldate = {2026-03-13},
	year = {2018},
}

@article{pfenninger_calliope_2018,
	title = {Calliope: a multi-scale energy systems modelling framework},
	volume = {3},
	rights = {http://creativecommons.org/licenses/by/4.0/},
	issn = {2475-9066},
	doi = {10.21105/joss.00825},
	shorttitle = {Calliope},
	pages = {825},
	number = {29},
	journal = {{JOSS}},
	author = {Pfenninger, Stefan and Pickering, Bryn},
	urldate = {2026-03-13},
	year = {2018},
}

@article{hoffmann_review_2020,
	title = {A Review on Time Series Aggregation Methods for Energy System Models},
	volume = {13},
	issn = {1996-1073},
	doi = {10.3390/en13030641},
	pages = {641},
	number = {3},
	journal = {Energies},
	author = {Hoffmann, Maximilian and Kotzur, Leander and Stolten, Detlef and Robinius, Martin},
	urldate = {2026-03-14},
	year = {2020},
	langid = {english},
}

@article{novo_planning_2022,
	title = {Planning the decarbonisation of energy systems: The importance of applying time series clustering to long-term models},
	volume = {15},
	issn = {25901745},
	doi = {10.1016/j.ecmx.2022.100274},
	shorttitle = {Planning the decarbonisation of energy systems},
	pages = {100274},
	journal = {Energy Conversion and Management: X},
	author = {Novo, Riccardo and Marocco, Paolo and Giorgi, Giuseppe and Lanzini, Andrea and Santarelli, Massimo and Mattiazzo, Giuliana},
	urldate = {2026-03-14},
	year = {2022},
	langid = {english},
}

@article{miraftabzadeh_k-means_2023,
	title = {K-Means and Alternative Clustering Methods in Modern Power Systems},
	volume = {11},
	rights = {https://creativecommons.org/licenses/by-nc-nd/4.0/},
	issn = {2169-3536},
	doi = {10.1109/ACCESS.2023.3327640},
	pages = {119596--119633},
	journal = {{IEEE} Access},
	author = {Miraftabzadeh, Seyed Mahdi and Colombo, Cristian Giovanni and Longo, Michela and Foiadelli, Federica},
	urldate = {2026},
	year = {2023},
}

@article{kotzur_impact_2018,
	title = {Impact of different time series aggregation methods on optimal energy system design},
	volume = {117},
	issn = {09601481},
	doi = {10.1016/j.renene.2017.10.017},
	pages = {474--487},
	journal = {Renewable Energy},
	author = {Kotzur, Leander and Markewitz, Peter and Robinius, Martin and Stolten, Detlef},
	urldate = {2026-03-14},
	year = {2018},
	langid = {english},
}

@article{teichgraeber_clustering_2019,
	title = {Clustering methods to find representative periods for the optimization of energy systems: An initial framework and comparison},
	volume = {239},
	issn = {03062619},
	doi = {10.1016/j.apenergy.2019.02.012},
	shorttitle = {Clustering methods to find representative periods for the optimization of energy systems},
	pages = {1283--1293},
	journal = {Applied Energy},
	author = {Teichgraeber, Holger and Brandt, Adam R.},
	urldate = {2026-03-14},
	year = {2019},
	langid = {english},
}

@article{catania_impact_2025,
	title = {The impact of temporal clustering on long-term energy system models},
	volume = {399},
	issn = {03062619},
	doi = {10.1016/j.apenergy.2025.126354},
	pages = {126354},
	journal = {Applied Energy},
	author = {Catania, Matteo and Muliere, Giuseppe and Fattori, Fabrizio and Colbertaldo, Paolo},
	urldate = {2026-03-14},
	year = {2025},
	langid = {english},
}

@article{kotzur_time_2018,
	title = {Time series aggregation for energy system design: Modeling seasonal storage},
	volume = {213},
	issn = {03062619},
	doi = {10.1016/j.apenergy.2018.01.023},
	shorttitle = {Time series aggregation for energy system design},
	pages = {123--135},
	journal = {Applied Energy},
	author = {Kotzur, Leander and Markewitz, Peter and Robinius, Martin and Stolten, Detlef},
	urldate = {2026-03-14},
	year = {2018},
	langid = {english},
}

@article{kittel_temporal_2022,
	title = {Temporal aggregation of time series to identify typical hourly electricity system states: A systematic assessment of relevant cluster algorithms},
	volume = {247},
	issn = {03605442},
	doi = {10.1016/j.energy.2022.123458},
	shorttitle = {Temporal aggregation of time series to identify typical hourly electricity system states},
	pages = {123458},
	journal = {Energy},
	author = {Kittel, Martin and Hobbie, Hannes and Dierstein, Constantin},
	urldate = {2026-03-14},
	year = {2022},
	langid = {english},
}

@article{jacobson_computationally_2024,
	title = {A Computationally Efficient Benders Decomposition for Energy Systems Planning Problems with Detailed Operations and Time-Coupling Constraints},
	volume = {6},
	issn = {2575-1484, 2575-1492},
	doi = {10.1287/ijoo.2023.0005},
	pages = {32--45},
	number = {1},
	journal = {{INFORMS} Journal on Optimization},
	author = {Jacobson, Anna and Pecci, Filippo and Sepulveda, Nestor and Xu, Qingyu and Jenkins, Jesse},
	urldate = {2026-03-14},
	year = {2024},
	langid = {english},
}

@article{pecci_regularized_2025,
	title = {Regularized Benders Decomposition for High Performance Capacity Expansion Models},
	volume = {40},
	rights = {https://ieeexplore.ieee.org/Xplorehelp/downloads/license-information/{IEEE}.html},
	issn = {0885-8950, 1558-0679},
	doi = {10.1109/TPWRS.2025.3526413},
	pages = {3105--3116},
	number = {4},
	journal = {{IEEE} Trans. Power Syst.},
	author = {Pecci, Filippo and Jenkins, Jesse D.},
	urldate = {2026-03-14},
	year = {2025},
}

@article{brahmbhatt_benders_2025,
	title = {Benders Decomposition Using Graph Modeling and Multi-Parametric Programming},
	volume = {64},
	rights = {https://creativecommons.org/licenses/by/4.0/},
	issn = {0888-5885, 1520-5045},
	doi = {10.1021/acs.iecr.5c03189},
	pages = {21684--21700},
	number = {45},
	journal = {Ind. Eng. Chem. Res.},
	author = {Brahmbhatt, Parth and Cole, David L. and Zavala, Victor M. and Avraamidou, Styliani},
	urldate = {2026-03-14},
	year = {2025},
	langid = {english},
}

@article{geoffrion_generalized_1972,
	title = {Generalized Benders decomposition},
	volume = {10},
	rights = {http://www.springer.com/tdm},
	issn = {0022-3239, 1573-2878},
	doi = {10.1007/BF00934810},
	pages = {237--260},
	number = {4},
	journal = {J Optim Theory Appl},
	author = {Geoffrion, A. M.},
	urldate = {2026-03-14},
	year = {1972},
	langid = {english},
}

@article{tang_improved_2013,
	title = {An improved Benders decomposition algorithm for the logistics facility location problem with capacity expansions},
	volume = {210},
	rights = {http://www.springer.com/tdm},
	issn = {0254-5330, 1572-9338},
	doi = {10.1007/s10479-011-1050-9},
	pages = {165--190},
	number = {1},
	journal = {Ann Oper Res},
	author = {Tang, Lixin and Jiang, Wei and Saharidis, Georgios K. D.},
	urldate = {2026-03-14},
	year = {2013},
	langid = {english},
}

@article{marin_electric_1998,
	title = {Electric capacity expansion under uncertain demand: decomposition approaches},
	volume = {13},
	rights = {https://ieeexplore.ieee.org/Xplorehelp/downloads/license-information/{IEEE}.html},
	issn = {08858950},
	doi = {10.1109/59.667347},
	shorttitle = {Electric capacity expansion under uncertain demand},
	pages = {333--339},
	number = {2},
	journal = {{IEEE} Trans. Power Syst.},
	author = {Marin, A. and Salmeron, J.},
	urldate = {2026-03-14},
	year = {1998},
}

@article{baringo_wind_2012,
	title = {Wind Power Investment: A Benders Decomposition Approach},
	volume = {27},
	rights = {https://ieeexplore.ieee.org/Xplorehelp/downloads/license-information/{IEEE}.html},
	issn = {0885-8950, 1558-0679},
	doi = {10.1109/TPWRS.2011.2167764},
	shorttitle = {Wind Power Investment},
	pages = {433--441},
	number = {1},
	journal = {{IEEE} Trans. Power Syst.},
	author = {Baringo, Luis and Conejo, Antonio J.},
	urldate = {2026-03-14},
	year = {2012},
}

@article{mallapragada_impact_2018,
	title = {Impact of model resolution on scenario outcomes for electricity sector system expansion},
	volume = {163},
	issn = {03605442},
	doi = {10.1016/j.energy.2018.08.015},
	pages = {1231--1244},
	journal = {Energy},
	author = {Mallapragada, Dharik S. and Papageorgiou, Dimitri J. and Venkatesh, Aranya and Lara, Cristiana L. and Grossmann, Ignacio E.},
	urldate = {2026-03-14},
	year = {2018},
	langid = {english},
}

@article{soares_integrated_2022,
	title = {An Integrated Progressive Hedging and Benders Decomposition With Multiple Master Method to Solve the Brazilian Generation Expansion Problem},
	volume = {37},
	rights = {https://ieeexplore.ieee.org/Xplorehelp/downloads/license-information/{IEEE}.html},
	issn = {0885-8950, 1558-0679},
	doi = {10.1109/TPWRS.2022.3141993},
	pages = {4017--4027},
	number = {5},
	journal = {{IEEE} Trans. Power Syst.},
	author = {Soares, Alessandro and Street, Alexandre and Andrade, Tiago and Garcia, Joaquim Dias},
	urldate = {2026-03-14},
	year = {2022},
}

@article{goke_stabilized_2024,
	title = {Stabilized Benders decomposition for energy planning under climate uncertainty},
	volume = {316},
	issn = {03772217},
	doi = {10.1016/j.ejor.2024.01.016},
	pages = {183--199},
	number = {1},
	journal = {European Journal of Operational Research},
	author = {Göke, Leonard and Schmidt, Felix and Kendziorski, Mario},
	urldate = {2026-03-14},
	year = {2024},
	langid = {english},
}

@article{zhang_integrated_2025,
	title = {Integrated investment, retrofit and abandonment energy system planning with multi-timescale uncertainty using stabilised adaptive Benders decomposition},
	volume = {325},
	issn = {03772217},
	doi = {10.1016/j.ejor.2025.04.005},
	pages = {261--280},
	number = {2},
	journal = {European Journal of Operational Research},
	author = {Zhang, Hongyu and Grossmann, Ignacio E. and {McKinnon}, Ken and Knudsen, Brage Rugstad and Nava, Rodrigo Garcia and Tomasgard, Asgeir},
	urldate = {2026-03-14},
	year = {2025},
	langid = {english},
}

@article{you_multicut_2013,
	title = {Multicut Benders decomposition algorithm for process supply chain planning under uncertainty},
	volume = {210},
	rights = {http://www.springer.com/tdm},
	issn = {0254-5330, 1572-9338},
	doi = {10.1007/s10479-011-0974-4},
	pages = {191--211},
	number = {1},
	journal = {Ann Oper Res},
	author = {You, Fengqi and Grossmann, Ignacio E.},
	urldate = {2026-03-14},
	year = {2013},
	langid = {english},
}

@article{zhang_stabilised_2024,
	title = {A stabilised Benders decomposition with adaptive oracles for large-scale stochastic programming with short-term and long-term uncertainty},
	volume = {167},
	issn = {03050548},
	doi = {10.1016/j.cor.2024.106665},
	pages = {106665},
	journal = {Computers \& Operations Research},
	author = {Zhang, Hongyu and Mazzi, Nicolò and {McKinnon}, Ken and Nava, Rodrigo Garcia and Tomasgard, Asgeir},
	urldate = {2026-03-14},
	year = {2024},
	langid = {english},
}

@misc{macdonald_macroenergyjl_2025,
	title = {{MacroEnergy}.jl: A large-scale multi-sector energy system framework},
	rights = {Creative Commons Attribution 4.0 International},
	doi = {10.48550/ARXIV.2510.21943},
	shorttitle = {{MacroEnergy}.jl},
	publisher = {{arXiv}},
	author = {Macdonald, Ruaridh and Pecci, Filippo and Bonaldo, Luca and Law, Jun Wen and Weng, Yu and Mallapragada, Dharik and Jenkins, Jesse},
	urldate = {2026-03-14},
	year = {2025},
}

@article{lohmann_tailored_2017,
	title = {Tailored Benders Decomposition for a Long-Term Power Expansion Model with Short-Term Demand Response},
	volume = {63},
	issn = {0025-1909, 1526-5501},
	doi = {10.1287/mnsc.2015.2420},
	pages = {2027--2048},
	number = {6},
	journal = {Management Science},
	author = {Lohmann, Timo and Rebennack, Steffen},
	urldate = {2026-03-20},
	year = {2017},
	langid = {english},
}

@article{lara_deterministic_2018,
	title = {Deterministic electric power infrastructure planning: Mixed-integer programming model and nested decomposition algorithm},
	volume = {271},
	issn = {03772217},
	doi = {10.1016/j.ejor.2018.05.039},
	shorttitle = {Deterministic electric power infrastructure planning},
	pages = {1037--1054},
	number = {3},
	journal = {European Journal of Operational Research},
	author = {Lara, Cristiana L. and Mallapragada, Dharik S. and Papageorgiou, Dimitri J. and Venkatesh, Aranya and Grossmann, Ignacio E.},
	urldate = {2026-03-20},
	year = {2018},
	langid = {english},
}

@article{li_mixed-integer_2022,
	title = {Mixed-integer linear programming models and algorithms for generation and transmission expansion planning of power systems},
	volume = {297},
	issn = {03772217},
	doi = {10.1016/j.ejor.2021.06.024},
	pages = {1071--1082},
	number = {3},
	journal = {European Journal of Operational Research},
	author = {Li, Can and Conejo, Antonio J. and Liu, Peng and Omell, Benjamin P. and Siirola, John D. and Grossmann, Ignacio E.},
	urldate = {2026-03-20},
	year = {2022},
	langid = {english},
}

@article{mazzi_benders_2021,
	title = {Benders decomposition with adaptive oracles for large scale optimization},
	volume = {13},
	issn = {1867-2949, 1867-2957},
	doi = {10.1007/s12532-020-00197-0},
	pages = {683--703},
	number = {4},
	journal = {Math. Prog. Comp.},
	author = {Mazzi, Nicolò and Grothey, Andreas and {McKinnon}, Ken and Sugishita, Nagisa},
	urldate = {2026-03-21},
	year = {2021},
	langid = {english},
}

@article{ramirez-pico_benders_2023,
	title = {Benders Adaptive-Cuts Method for Two-Stage Stochastic Programs},
	volume = {57},
	issn = {0041-1655, 1526-5447},
	doi = {10.1287/trsc.2022.0073},
	pages = {1252--1275},
	number = {5},
	journal = {Transportation Science},
	author = {Ramírez-Pico, Cristian and Ljubić, Ivana and Moreno, Eduardo},
	urldate = {2026-03-21},
	year = {2023},
	langid = {english},
}

@article{munoz_new_2016,
	title = {New bounding and decomposition approaches for {MILP} investment problems: Multi-area transmission and generation planning under policy constraints},
	volume = {248},
	issn = {03772217},
	doi = {10.1016/j.ejor.2015.07.057},
	shorttitle = {New bounding and decomposition approaches for {MILP} investment problems},
	pages = {888--898},
	number = {3},
	journal = {European Journal of Operational Research},
	author = {Munoz, F.D. and Hobbs, B.F. and Watson, J.-P.},
	urldate = {2026-03-21},
	year = {2016},
	langid = {english},
}

@article{allen_improvements_2023,
	title = {Improvements for decomposition based methods utilized in the development of multi-scale energy systems},
	volume = {170},
	issn = {00981354},
	doi = {10.1016/j.compchemeng.2023.108135},
	pages = {108135},
	journal = {Computers \& Chemical Engineering},
	author = {Allen, R. Cory and Iseri, Funda and Demirhan, C. Doga and Pappas, Iosif and Pistikopoulos, Efstratios N.},
	urldate = {2026-03-21},
	year = {2023},
	langid = {english},
}

@article{brandenberg_refined_2021,
	title = {Refined cut selection for benders decomposition: applied to network capacity expansion problems},
	volume = {94},
	issn = {1432-2994, 1432-5217},
	doi = {10.1007/s00186-021-00756-8},
	shorttitle = {Refined cut selection for benders decomposition},
	pages = {383--412},
	number = {3},
	journal = {Math Meth Oper Res},
	author = {Brandenberg, René and Stursberg, Paul},
	urldate = {2026-03-21},
	year = {2021},
	langid = {english},
}

@inproceedings{lumbreras_transmission_2013,
	location = {Grenoble, France},
	title = {Transmission expansion planning using an efficient version of Benders' decomposition. A case study},
	isbn = {978-1-4673-5669-5},
	doi = {10.1109/PTC.2013.6652091},
	eventtitle = {2013 {IEEE} Grenoble {PowerTech}},
	pages = {1--7},
	booktitle = {2013 {IEEE} Grenoble Conference},
	publisher = {{IEEE}},
	author = {Lumbreras, Sara and Ramos, Andres},
	urldate = {2026-03-21},
	year = {2013},
}

@misc{barbar_representative_2022,
	title = {Representative period selection for power system planning using autoencoder-based dimensionality reduction},
	rights = {Creative Commons Attribution 4.0 International},
	doi = {10.48550/ARXIV.2204.13608},
	publisher = {{arXiv}},
	author = {Barbar, Marc and Mallapragada, Dharik S.},
	urldate = {2026-04-05},
	year = {2022},
}

@article{homaei_high-capture-rate_2026,
	title = {High-capture-rate carbon capture and storage enables cost-effective decarbonization of Europe’s power sector},
	volume = {1},
	issn = {3059-4308},
	doi = {10.1038/s44458-026-00036-8},
	pages = {34},
	number = {1},
	journal = {Commun. Sustain.},
	author = {Homaei, Shamim and Anantharaman, Rahul and Backe, Stian and Roussanaly, Simon and Tomasgard, Asgeir},
	urldate = {2026-04-05},
	year = {2026},
	langid = {english},
}

@article{sodwatana_appliance_2025,
	title = {Appliance decarbonization and its impacts on California’s energy transition},
	volume = {390},
	issn = {03062619},
	doi = {10.1016/j.apenergy.2025.125769},
	pages = {125769},
	journal = {Applied Energy},
	author = {Sodwatana, Mo and Saad, Dimitri M. and Ahumada-Paras, Mareldi and Brandt, Adam R.},
	urldate = {2026-04-05},
	year = {2025},
	langid = {english},
}

@report{wecc_transmission_nodate,
	title = {Transmission},
	url = {https://feature.wecc.org/soti/topic-sections/transmission/index.html},
	institution = {Western Electricity Coordinating Council},
	author = {{WECC}},
}

@report{us_census_bureau_county_nodate,
	title = {County Population Totals and Components of Change: 2020-2025},
	url = {https://www.census.gov/data/tables/time-series/demo/popest/2020s-counties-total.html},
	author = {{U.S. Census Bureau}},
    year = {2026},
}

@misc{office_of_research_computing_and_data_mit_about_nodate,
	title = {About the Engaging Cluster},
	url = {https://orcd.mit.edu/resources/about-engaging-cluster},
	author = {{Office of Research Computing and Data (MIT)}},
}

@article{johnston_switch_2019,
	title = {Switch 2.0: A modern platform for planning high-renewable power systems},
	volume = {10},
	issn = {23527110},
	doi = {10.1016/j.softx.2019.100251},
	shorttitle = {Switch 2.0},
	pages = {100251},
	journal = {{SoftwareX}},
	author = {Johnston, Josiah and Henriquez-Auba, Rodrigo and Maluenda, Benjamín and Fripp, Matthias},
	urldate = {2026-05-14},
	year = {2019},
	langid = {english},
}

@article{poncelet_impact_2016,
	title = {Impact of the level of temporal and operational detail in energy-system planning models},
	volume = {162},
	issn = {03062619},
	doi = {10.1016/j.apenergy.2015.10.100},
	pages = {631--643},
	journal = {Applied Energy},
	author = {Poncelet, Kris and Delarue, Erik and Six, Daan and Duerinck, Jan and D’haeseleer, William},
	urldate = {2026-05-14},
	year = {2016},
	langid = {english},
}

@article{arango_representative_2018,
  author={Tejada-Arango, Diego A. and Domeshek, Maya and Wogrin, Sonja and Centeno, Efraim},
  journal={IEEE Transactions on Power Systems}, 
  title={Enhanced Representative Days and System States Modeling for Energy Storage Investment Analysis}, 
  year={2018},
  volume={33},
  number={6},
  pages={6534-6544},
  doi={10.1109/TPWRS.2018.2819578}}

@ARTICLE{pineda_chronological_2018,
  author={Pineda, Salvador and Morales, Juan M.},
  journal={IEEE Transactions on Power Systems}, 
  title={Chronological Time-Period Clustering for Optimal Capacity Expansion Planning With Storage}, 
  year={2018},
  volume={33},
  number={6},
  pages={7162-7170},
  doi={10.1109/TPWRS.2018.2842093}}

@article{blanke_time_2022,
	title = {Time series aggregation for energy system design: review and extension of modelling seasonal storages},
	volume = {5},
	issn = {2520-8942},
	doi = {10.1186/s42162-022-00208-5},
	shorttitle = {Time series aggregation for energy system design},
	pages = {17},
	issue = {S1},
	journal = {Energy Inform},
	author = {Blanke, Tobias and Schmidt, Katharina S. and Göttsche, Joachim and Döring, Bernd and Frisch, Jérôme and Van Treeck, Christoph},
	urldate = {2026-05-14},
	year = {2022},
	langid = {english},
}

@article{yi_aggregate_2021,
	title = {Aggregate Operation Model for Numerous Small-Capacity Distributed Energy Resources Considering Uncertainty},
	volume = {12},
	rights = {https://ieeexplore.ieee.org/Xplorehelp/downloads/license-information/{IEEE}.html},
	issn = {1949-3053, 1949-3061},
	doi = {10.1109/TSG.2021.3085885},
	pages = {4208--4224},
	number = {5},
	journal = {{IEEE} Trans. Smart Grid},
	author = {Yi, Zhongkai and Xu, Yinliang and Gu, Wei and Yang, Lun and Sun, Hongbin},
	urldate = {2026-05-19},
	year = {2021},
}

@inproceedings{hoettecke_enhanced_2021,
	location = {Vaasa, Finland},
	title = {Enhanced time series aggregation for long-term investment planning models of energy supply infrastructure in production plants},
	rights = {https://ieeexplore.ieee.org/Xplorehelp/downloads/license-information/{IEEE}.html},
	isbn = {978-1-7281-7660-4},
	doi = {10.1109/SEST50973.2021.9543162},
	eventtitle = {2021 International Conference on Smart Energy Systems and Technologies ({SEST})},
	pages = {1--6},
	booktitle = {2021 International Conference on Smart Energy Systems and Technologies ({SEST})},
	publisher = {{IEEE}},
	author = {Hoettecke, Lukas and Thiem, Sebastian and Niessen, Stefan},
	urldate = {2026-05-14},
	year = {2021},
}

@article{lemarechal_new_1995,
	title = {New variants of bundle methods},
	volume = {69},
	rights = {http://www.springer.com/tdm},
	issn = {0025-5610, 1436-4646},
	doi = {10.1007/BF01585555},
	pages = {111--147},
	number = {1},
	journal = {Mathematical Programming},
	author = {Lemaréchal, Claude and Nemirovskii, Arkadii and Nesterov, Yurii},
	urldate = {2026-05-14},
	year = {1995},
	langid = {english},
}

@article{gondzio_new_2013,
	title = {New developments in the primal–dual column generation technique},
	volume = {224},
	rights = {https://www.elsevier.com/tdm/userlicense/1.0/},
	issn = {03772217},
	doi = {10.1016/j.ejor.2012.07.024},
	pages = {41--51},
	number = {1},
	journal = {European Journal of Operational Research},
	author = {Gondzio, Jacek and González-Brevis, Pablo and Munari, Pedro},
	urldate = {2026-05-14},
	year = {2013},
	langid = {english},
}

@article{bertsimas_stochastic_2025,
	title = {A Stochastic Benders Decomposition Scheme for Large-Scale Stochastic Network Design},
	volume = {37},
	issn = {1091-9856, 1526-5528},
	doi = {10.1287/ijoc.2023.0074},
	pages = {1163--1181},
	number = {5},
	journal = {{INFORMS} Journal on Computing},
	author = {Bertsimas, Dimitris and Cory-Wright, Ryan and Pauphilet, Jean and Petridis, Periklis},
	urldate = {2026-05-14},
	year = {2025},
	langid = {english},
}

@article{parolin_sectoral_2026,
	title = {Sectoral and spatial decomposition methods for multi-sector capacity expansion models},
	volume = {358},
	issn = {01968904},
	doi = {10.1016/j.enconman.2026.121356},
	pages = {121356},
	journal = {Energy Conversion and Management},
	author = {Parolin, Federico and Weng, Yu and Colbertaldo, Paolo and Macdonald, Ruaridh},
	urldate = {2026-05-17},
	year = {2026},
	langid = {english},
}

@article{fischetti_benders_2016,
	title = {Benders decomposition without separability: A computational study for capacitated facility location problems},
	volume = {253},
	issn = {03772217},
	doi = {10.1016/j.ejor.2016.03.002},
	shorttitle = {Benders decomposition without separability},
	pages = {557--569},
	number = {3},
	journal = {European Journal of Operational Research},
	author = {Fischetti, Matteo and Ljubić, Ivana and Sinnl, Markus},
	urldate = {2026-05-19},
	year = {2016},
	langid = {english},
}

@article{gruson_benders_2021,
	title = {Benders decomposition for a stochastic three-level lot sizing and replenishment problem with a distribution structure},
	volume = {291},
	issn = {03772217},
	doi = {10.1016/j.ejor.2020.09.019},
	pages = {206--217},
	number = {1},
	journal = {European Journal of Operational Research},
	author = {Gruson, Matthieu and Cordeau, Jean-François and Jans, Raf},
	urldate = {2026-05-19},
	year = {2021},
	langid = {english},
}

@article{kergosien_benders_2017,
	title = {A Benders decomposition-based heuristic for a production and outbound distribution scheduling problem with strict delivery constraints},
	volume = {262},
	issn = {03772217},
	doi = {10.1016/j.ejor.2017.03.028},
	pages = {287--298},
	number = {1},
	journal = {European Journal of Operational Research},
	author = {Kergosien, Y. and Gendreau, M. and Billaut, J.-C.},
	urldate = {2026-05-19},
	year = {2017},
	langid = {english},
}

\end{document}